\documentclass[11pt]{article}
\usepackage[small]{titlesec}
\usepackage[top = 0.66in,textwidth = 6.5in, textheight=9.1in]{geometry}

\usepackage{amsmath}
\usepackage{graphicx}
\usepackage{latexsym}
\usepackage{color}
\usepackage{amssymb}
\usepackage{tabularx}
\usepackage{fancyhdr}
\usepackage{verbatim}
\usepackage{multirow}
\usepackage{framed}
\usepackage{float, subfig}
\usepackage{enumitem}
\usepackage{mathtools}
\usepackage{mathrsfs}
\usepackage{amsfonts}
\usepackage{listings}
\usepackage{amsthm}
\usepackage{grffile}
\usepackage{sidecap}
\usepackage{pbox}
\usepackage{algorithm}
\usepackage{longtable}
\usepackage[noend]{algpseudocode}
\usepackage{alphalph}
\usepackage{etoolbox}
\usepackage{url}
\usepackage{natbib}
\bibpunct[, ]{(}{)}{,}{a}{}{,}%

\newcommand{\cA}{{\cal A}}

\newcommand{\cG}{{\cal G}}

\newcommand{\cM}{{\cal M}}
\newcommand{\cN}{{\cal N}}
\newcommand{\cS}{{\cal S}}
\newcommand{\cU}{{\cal U}}

\newcommand{\ms}{\medskip}

\newcommand{\tII}{{{I}}}
\allowdisplaybreaks 

\newtheorem{theorem}{Theorem}
\newtheorem{lemma}[theorem]{Lemma}

\newtheorem{assumption}{Assumption}

\newtheorem{definition}{Definition}

\renewcommand{\Re}{{\mathbb R}}

\patchcmd{\subequations}{\alph{equation}}{\alphalph{\value{equation}}}{}{}
\usepackage{accents}
\renewcommand{\underline}{\underaccent{\bar}}
\renewcommand{\overline}{\bar}
\renewcommand{\epsilon}{\varepsilon}

\allowdisplaybreaks

\begin{document}
\baselineskip0.25in

\begin{center}
\begin{large}
\begin{bf}

Robust Optimization for Electricity Generation \ms

\today \ms

\end{bf}
Chaithanya Bandi, Krishnamurthy Dvijotham, David Morton, Haoxiang Yang \ms
\end{large}
\end{center}

\begin{abstract}
	We consider a robust optimization problem in an electric power system under uncertain demand and availability of renewable energy resources. Solving the deterministic alternating current optimal power flow (ACOPF) problem has been considered challenging since the 1960s due to its nonconvexity. Linear approximation of the AC power flow system sees pervasive use, but does not guarantee a physically feasible system configuration. In recent years, various convex relaxation schemes for the ACOPF problem have been investigated, and under some assumptions, a physically feasible solution can be recovered. Based on these convex relaxations, we construct a robust convex optimization problem with recourse to solve for optimal controllable injections (fossil fuel, nuclear, etc.) in electric power systems under uncertainty (renewable energy generation, demand fluctuation, etc.). We propose a cutting-plane method to solve this robust optimization problem, and we establish convergence and other desirable properties. Experimental results indicate that our robust convex relaxation of the ACOPF problem can provide a tight lower bound.
\end{abstract}

\section{Introduction}
\label{sec:intro}
The alternating current optimal power flow (ACOPF) problem has been a topic of interest in the academic literature since the 1960s \citep{carpentier1962}. The ACOPF problem is used to determine the output for all generators and establish the system's \emph{configuration}, i.e., the voltage and phase angle at each bus and resulting power flows on lines. The goal is usually to minimize the generation cost and keep the system configuration within a stable range; see, for example, \citet{bienstock2015electrical} for a detailed discussion. While the ACOPF problem can be formulated as a quadratically constrained quadratic program, realistic instances are challenging to solve within time limits commensurate with an operational schedule---usually a few minutes---because of their scale and nonconvexities \citep[see, e.g., discussions on such challenges in][]{lavaei2012zero, low2014convex, verma2010power}. Linearizing the power flow equations simplifies the nonconvex ACOPF problem to what the literature calls a DCOPF approximation, which is a linear program, and this approximation is frequently applied. However, optimality and feasibility of the solution to the original ACOPF problem cannot be guaranteed because the voltage at each bus is assumed to be fixed and reactive power is ignored; see, e.g., \citet{momoh1999review} and \citet{stott2009dc} for reviews of such DCOPF approximations. In recent years, there has been an increasing focus on the ACOPF problem, and employing convex relaxations rooted in semidefinite programming and second-order cone programming as an approximation of this nonconvex problem~\citep{bai2011semidefinite, bai2008semidefinite, coffrin2016qc, jabr2006radial, kocuk2016strong, lavaei2012zero, low2014convexII}, and under some circumstances, these relaxed solutions recover the exact optimal solution of the original nonconvex ACOPF problem.

Electric power systems operate under significant uncertainty due to system load, failure of generation and transmission assets, and uncertain generation from renewable energy sources (RESs) including wind and solar resources. In this context, we seek an economic dispatch decision that is robust to uncertainty in load and RES generation. With stochastic realizations of power from wind farms, \citet{phan2014two} model economic dispatch under ACOPF as a two-stage stochastic program, and use a sample average approximation. \citet{monticelli1987security} introduce a security-constrained variant of an economic dispatch model in which the goal is to obtain a solution that can adapt to failure of a subset of system components explicitly modeled through a set of contingencies. Instead of enforcing feasibility for all modeled contingencies, \citet{lubin2016robust} formulate a chance-constrained model that ensures feasibility with high probability. Robust optimization is a natural modeling framework for security-constrained problems in that such models yield solutions that can handle any contingency within a specified uncertainty set. \citet{jabr2013adjustable} and \citet{louca2017robust} propose an adaptive robust optimization model, in which recourse decisions are represented as an affine function of realizations of uncertainty such as available power from RESs. \citet{conejo2018adaptive} propose a tri-level decomposition algorithm, where in the second level a DCOPF relaxation is solved to obtain worst-case scenarios, which are further used to construct a large-scale extensive formulation of the robust ACOPF problem.

Although significant progress has been made both in convex relaxations of nonconvex ACOPF problems and in modeling dispatch under uncertainty, there is much less work that combines these two threads; i.e., most stochastic or robust models for economic dispatch use the linear DCOPF approximation. \citet{ferris2015security} solve a scenario-based security-constrained ACOPF problem, in which for each contingency the ACOPF is relaxed as a semidefinite program (SDP). \citet{lorca2018adaptive} model a multi-period two-stage robust ACOPF problem using a conic relaxation, which is similar to our approach, but we focus more on the feasibility guarantee and the properties of our robust solution. 

We solve a robust convex approximation, without specifying scenarios but by constructing an uncertainty set, to simultaneously reap the benefit of a tighter relaxation and include uncertainty in our model. We assume an uncontrollable injection represents net load at each bus. Here, net load captures demand and RES generation, which are subject to simple bounds and further constraints that define the uncertainty set. Our broad goal is to find a robust and economical energy generation plan. Here, robustness means that for all contingencies modeled by our uncertainty set, we can find a configuration that satisfies the system's physical and operational constraints. We call such a plan a robust optimal solution to the ACOPF problem. 

Our formulation is unique in that, in addition to using a convex relaxation of the ACOPF problem rather than a DC approximation, we employ a ``full recourse" solution rather than relying on simpler approximations like linear decision rules. There are three possible outcomes from solving our model. First, the solution to the convex approximation may be feasible to the robust nonconvex ACOPF problem, which means we exactly recover a robust solution. Second, due to the convex relaxation, the solution we obtain may not be feasible to the robust nonconvex ACOPF problem, but we obtain a lower bound on the optimal cost of the nonconvex counterpart, which yields a bound on the optimality gap when coupled with a heuristically obtained feasible solution. Third, if the convex relaxation is infeasible, we identify infeasibility of the robust nonconvex ACOPF problem. Our specific goals are to understand: (i)~whether our convex relaxation of the robust ACOPF problem yields a high quality lower bound on the nonconvex model's optimal value; (ii)~whether the solution we obtain from the convex relaxation is feasible, or nearly feasible, to the robust ACOPF model; and, (iii)~how solutions to our robust convex relaxation compare to alternatives from simpler deterministic models. Addressing these goals requires developing an algorithm to handle problems of reasonable scale.

In Section~\ref{sec:formulation}, we formulate our convex relaxation of the ACOPF problem. A cutting-plane method is proposed in Section~\ref{sec:cutting}, and the proof of its convergence is detailed. Experimental results are reported in Section~\ref{sec:exp}, and conclusions are drawn in Section~\ref{sec:conclusion}.

\section{Problem Formulation}
\label{sec:formulation}
In this section we formulate the robust nonconvex ACOPF problem and its convex relaxation. We index the set of buses in the power system by \(\cN\), and the set of lines by \(\cA\). The set of controllable generators is denoted by \(\cG\), and the subset of generators connected to bus \(i\) is indexed by \(\cG_i\). Each controllable generator \(g \in \cG\) injects active power \(s_g^p\) and reactive power \(s_g^q\) at the single bus \(i\) satisfying \(g \in \cG_i\). Each bus \(i \in \cN\) has an uncontrollable injection, which may be negative, consisting of the uncertain net load due to actual demand and RES generation at that bus. At bus \(i \in \cN\), the uncontrollable active power injected, \(u_i^p\), is bounded within an uncertainty set \([\underline{u}_i^p,  \overline{u}_i^p]\), where \( \underline{u}_i^p \leq \overline{u}_i^p,\ \forall i \in \cN \). The uncontrollable reactive power is bounded in a similar way, where \(u_i^q \in [\underline{u}_i^q, \overline{u}_i^q],\ \forall i \in \cN\). 

In addition to simple bounds, we introduce a ``budget constraint'' in our uncertainty set, which limits the magnitude of deviation from a nominal injection, summed across all buses. Such budget-constrained uncertainty sets have been widely applied in robust optimization, starting with \citet{bertsimas2003robust,bertsimas2004price}. Here we denote the nominal uncontrollable active and reactive power injection as \(u^{p,0}\) and \(u^{q,0}\), which are both vectors with \(|\cN|\) components and satisfy \((\underline{u}^p_i,\underline{u}^q_i) \leq (u^{p,0},u^{q,0}) \leq (\overline{u}^p_i,\overline{u}^q_i)\), for \(i \in \cN\). Appendix~\ref{appen:facLoc} details how we cluster the set of buses \(\cN\) into \(|\cM|\) subgroups, denoted by \(\cN_m,\ m \in \cM\), using a facility location problem. There has been significant work regarding the geographical correlation of renewable generation and load in power systems~\citep[e.g.,][]{bernstein2014power,fang2018modelling,klima2015geographic,lohmann2016local,malvaldi2017spatial,xie2018distributionally}, which is broadly consistent with our clustering scheme. We assume within each cluster the relative magnitude of deviation is the same for both the active power and reactive power at every bus. We define the uncertainty set with the following constraints:
	\begin{subequations}
	\begin{align}
	& 0 \leq u_i^{p,+} \leq \overline{u}_i^p - u_i^{p,0} \qquad 0 \leq u_i^{p,-} \leq u_i^{p,0} - \underline{u}_i^p \qquad \qquad  \qquad \qquad \forall i \in \cN \label{eqn:upBounds}\\
	& 0 \leq u_i^{q,+} \leq \overline{u}_i^q - u_i^{q,0} \qquad 0 \leq u_i^{q,-} \leq u_i^{q,0} - \underline{u}_i^q \qquad \qquad  \qquad \qquad  \forall i \in \cN \label{eqn:uqBounds}\\
	& \frac{u_i^{p,+}}{\overline{u}_i^p - u_i^{p,0}} = \frac{u_i^{q,+}}{\overline{u}_i^q - u_i^{q,0}} = u_m^{+} \qquad \frac{u_i^{p,-}}{u_i^{p,0} - \underline{u}_i^p} = \frac{u_i^{q,-}}{u_i^{q,0} - \underline{u}_i^q} = u_m^{-} \qquad \forall m \in \cM, i \in \cN_m \label{eqn:groupBounds}
	\end{align}
\end{subequations}
\begin{align}\label{eqn:setU}
\cU = \left\{(u^{p,+},u^{p,-},u^{q,+},u^{q,-}) \in \Re^{4 |\cN|} \  \left| \  
\begin{aligned}
& \text{\eqref{eqn:upBounds}-\eqref{eqn:uqBounds} and $\exists \, u_m^{+}, u_m^{-}, m \in \cM$, }\\
&\text{satisfying~\eqref{eqn:groupBounds} and} \ \sum_{m \in \cM} \left(u_m^+ + u_m^- \right) \leq \Gamma \\
\end{aligned}
\right. \right\}.
\end{align}
Budget parameter \(\Gamma\) controls the deviation from nominal values, summed across all buses. We can substitute out variables $u_m^{+}, u_m^{-}, m \in \cM$, and we assume this has been done when referencing $\cU$ in what follows. 

In most of the power systems literature, lines are assumed to be undirected, and an orientation indicates the direction of flow. We represent multiple lines by a triple \((i,j,n)\), which uses the orientation to indicate that positive flow is from \(i\) to \(j\) on the \(n\)-th line between these two buses and negative flow is the opposite. Each bus has a voltage, \(v_i\), and a phase angle, \(\theta_i\). These configurations, along with the line parameters (complex admittance \(y_k = g_k + \sqrt{-1} \, b_k\)), the charging susceptance \(b_k^c\), and the shunt admittance of a bus \(y_i^{sh} = g_i^{sh} + \sqrt{-1} \, b_i^{sh}\) determine the  power flow on line \(k = (i,j,n) \in \cA \), where \(P_k\) and \(Q_k\) denote active and reactive power flow, respectively:
\begin{subequations}\label{eqn:pf}
	\begin{align}
	P_k = &g_k \frac{v_i^2}{\tau_{1,k}^2} - g_k \frac{v_i v_j }{\tau_{1,k} \tau_{2,k}} \cos(\theta_i - \sigma_k - \theta_j) - \nonumber \\
	& b_k \frac{v_i v_j}{\tau_{1,k} \tau_{2,k}} \sin(\theta_i - \sigma_k - \theta_j), \qquad \forall k = (i,j,n) \in \cA\\
	Q_k = &-(b_k + \frac{b_k^c}{2}) \frac{v_i^2}{\tau_{1,k}^2} + b_k \frac{v_i v_j }{\tau_{1,k} \tau_{2,k}}\cos(\theta_i - \sigma_k - \theta_j) - \nonumber \\
	& g_k \frac{v_i v_j}{\tau_{1,k} \tau_{2,k}} \sin(\theta_i - \sigma_k - \theta_j), \qquad \forall k = (i,j,n) \in \cA.
	\end{align}
\end{subequations}

Here we split the tap ratio for each line \(k = (i,j,n) \in \cA\) into \(\tau_{1,k}\) and \(\tau_{2,k}\) to represent the change of voltage at two ends of that line. We have \(\tau_{1,k} = \tau\) and \(\tau_{2,k} = 1\) if a transformer with tap ratio \(\tau\) is located at the bus \(i\) of line \(k = (i,j,n) \in \cA\), while we have \(\tau_{1,k} = 1\), \(\tau_{2,k} = \tau\) if a transformer with tap ratio \(\tau\) is located at the bus \(j\) of line \(k = (i,j,n) \in \cA\). Similarly, for the transformer phase angle shift, if a transformer with phase angle shift is located at the bus \(i\) of line \(k = (i,j,n) \in \cA\), we set \(\sigma_k = \sigma\); otherwise, if a transformer with phase angle shift is located at the bus \(j\) of line \(k = (i,j,n) \in \cA\), we set \(\sigma_k = -\sigma\).

At each bus \(i \in \cN\), we enforce flow conservation of active and reactive power via equations~\eqref{eqn:balance}. The left-hand side of constraint~\eqref{eqn:balance} is the net active and reactive power flowing out of bus \(i\), and they equal the sum of controllable and uncontrollable injections:
\begin{subequations} \label{eqn:balance}
	\begin{equation}
	\sum_{k = (i,j,n) \in \cA} P_k + g_i^{sh} (v_i)^2 = \sum_{g \in \cG_i} s_g^p + \left(u_i^{p,0} + u_i^{p,+} - u_i^{p,-} \right), \ \forall i \in \cN
	\end{equation}
	\begin{equation}
	\sum_{k = (i,j,n) \in \cA} Q_k - b_i^{sh} (v_i)^2 = \sum_{g \in \cG_i} s_g^q + \left(u_i^{q,0} + u_i^{q,+} - u_i^{q,-} \right), \ \forall i \in \cN.
	\end{equation}
\end{subequations}

Constraint~\eqref{eqn:angleDiff} bounds the difference in phase angle between adjacent buses, constraint~\eqref{eqn:lineConstr} limits the apparent power flowing through each line \(k\), and constraints~\eqref{eqn:volConstr}-\eqref{eqn:sqConstr} provide simple bounds on voltage and phase angle at each bus and active and reactive power at each generator:
\begin{subequations}
	\begin{align}
	&\underline{\Delta}_{k} \leq \theta_i - \sigma_k - \theta_j \leq \overline{\Delta}_{k} \qquad \qquad \forall k=(i,j,n) \in \cA \label{eqn:angleDiff}\\
	& P_k^2 + Q_k^2 \leq W_k^2 \qquad \qquad  \qquad \quad \;\;\; \forall k \in \cA \label{eqn:lineConstr}\\
	& \underline{v}_i \leq v_i \leq \overline{v}_i \qquad \qquad  \qquad \qquad \quad \forall i \in \cN \label{eqn:volConstr} \\
	& \underline{\theta}_i \leq \theta_i \leq \overline{\theta}_i \qquad \qquad  \qquad \qquad \quad \forall i \in \cN \label{eqn:angConstr} \\
	& \underline{s}_g^p \leq s_g^p \leq \overline{s}_g^p \qquad \qquad  \qquad \qquad \;\;\; \forall g \in \cG \label{eqn:spConstr}\\
	& \underline{s}_g^q \leq s_g^q \leq \overline{s}_g^q \qquad \qquad  \qquad \qquad \;\;\; \forall g \in \cG. \label{eqn:sqConstr}
	\end{align}
\end{subequations}

We denote the cost of controllable injections as \(c(s^p,s^q)\), and assume \(c\) is convex, where \(s^p\) and \(s^q\) are \(|\cG|\)-dimensional vectors with respective components \(s^p_g\) and \(s^q_g,\ g \in \cG\). The first-stage decision variables, \(s^p\) and \(s^q\), denote controllable injections that cannot adapt to the realized scenario. We allow small adjustments to these injections via variables \(o^{p,+}, o^{p,-}, o^{q,+}, o^{q,-}\), which can be selected once the uncertainty is revealed. These denote near real-time compensation in net generation; these variables have upper bounds proportional to the generation capacity at each bus so that the upper bound is zero for buses without generators. This setting permits greater flexibility than the linearly adaptive control used in previous research \citep{bienstock2014chance, jabr2013adjustable, louca2017robust, lubin2016robust}. We seek a robust optimal controllable injection such that for all possible uncontrollable injections in \(\cU\), there is a feasible system configuration via variables \((v,\theta,o^{p,+},o^{p,-},o^{q,+},o^{q,-},P, Q)\).

We minimize the set point cost, and consider linear and convex quadratic cost functions: 
\begin{eqnarray*}
	&& c(s^p,s^q) = \sum_{g \in \cG} \left[c^p_{g,2} \left(s^p_g\right)^2 + c^p_{g,1} s^p_g + c^q_{g,2} \left(s^q_g\right)^2 + c^q_{g,1} s^q_g \right], \label{eqn:quadCost}
\end{eqnarray*}
where \(c^p_{g,2} \ge 0\) and \(c^q_{g,2} \geq 0\) for all \(g \in \cG\) and take value zero in the linear case.

Convexity of the cost function is important because, although the ACOPF problem has nonconvex constraints, a convex objective function, together with the convex relaxation of the feasible region to be discussed below, yields a convex program. Our robust optimization formulation can be expressed as follows:
\begin{subequations} \label{prob:rACOPF}
	\begin{align}
	\min \quad & c(s^p,s^q)  & \label{eqn:obj}\\
	\text{s.t.} \quad & \underline{s}_g^p \leq s_g^p \leq \overline{s}_g^p \qquad \qquad \qquad \qquad \qquad \qquad \quad \forall g \in \cG \label{eqn:spConstr1}\\
	&  \underline{s}_g^q \leq s_g^q \leq \overline{s}_g^q \qquad \qquad \qquad \qquad \qquad \qquad \quad \forall g \in \cG \label{eqn:sqConstr1}\\
	& P_k^u = g_k \frac{(v_i^u)^2}{\tau_{1,k}^2} - g_k \frac{v_i^u v_j^u}{\tau_{1,k} \tau_{2,k}} \cos(\theta_i^u - \sigma_k - \theta_j^u) -   \nonumber \\
	& \qquad b_k \frac{v_i^u v_j^u}{\tau_{1,k} \tau_{2,k}} \sin(\theta_i^u - \sigma_k - \theta_j^u) \qquad \qquad \forall k = (i,j,n) \in \cA, \ u \in \cU \label{eqn:Pk}\\
	& Q_k^u = -(b_k + \frac{b_k^c}{2}) \frac{(v_i^u)^2}{\tau_{1,k}^2} + b_k \frac{v_i^u v_j^u}{\tau_{1,k} \tau_{2,k}} \cos(\theta_i^u - \sigma_k - \theta_j^u) -  \nonumber\\
	&\qquad g_k \frac{v_i^u v_j^u}{\tau_{1,k} \tau_{2,k}} \sin(\theta_i^u - \sigma_k - \theta_j^u) \qquad \qquad \forall k = (i,j,n) \in \cA,\ u \in \cU \label{eqn:Qk}\\
	& \underline{\Delta}_{k} \leq \theta_i^u - \sigma_k - \theta_j^u \leq \overline{\Delta}_{k} \qquad \qquad \qquad \quad \; \forall k = (i,j,n) \in \cA,\ u \in \cU \label{eqn:thetaDiff}\\
	& (P_k^u)^2 + (Q_k^u)^2 \leq W_k^2 \qquad \qquad \qquad \qquad \;\;\; \forall k \in \cA,\ u \in \cU \label{eqn:lineCons}\\
	& \underline{v}_i \leq v_i^u \leq \overline{v}_i \qquad \qquad \qquad \qquad \qquad \qquad \; \; \forall i \in \cN,\ u \in \cU \label{eqn:volCons}\\
	& \underline{\theta}_i \leq \theta_i^u \leq \overline{\theta}_i \qquad \qquad \qquad \qquad \qquad \qquad \; \; \forall i \in \cN \label{eqn:angCons} \\
	& \sum_{k = (i,j,n) \in \cA} P_k^u + g_i^{sh} (v_i^u)^2 + o_i^{p,-,u} - o_i^{p,+,u}\nonumber \\
	& \qquad \qquad = \sum_{g \in \cG_i} s_g^p + \left(u_i^{p,0} + u_i^{p,+} - u_i^{p,-} \right) \qquad \qquad \forall i \in \cN,\ u \in \cU \label{eqn:balanceP}\\
	& \sum_{k = (i,j,n) \in \cA} Q_k^u - b_i^{sh} (v_i^u)^2 + o_i^{q,-,u} - o_i^{q,+,u} \nonumber \\
	& \qquad \qquad = \sum_{g \in \cG_i} s_g^q + \left(u_i^{q,0} + u_i^{q,+} - u_i^{q,-} \right) \qquad \qquad \forall i \in \cN,\ u \in \cU \label{eqn:balanceQ}\\
	& o_i^{p,+,u} \leq \overline{o}^p_i + \left(h^p_i + \zeta_i^+ u_i^{p,+} - \zeta_i^- u_i^{p,-}\right) \qquad \qquad \quad \;\;\; \forall i \in \cN, u \in \cU \label{eqn:oppbounds}\\
	& o_i^{q,+,u} \leq \overline{o}^q_i + \left(h^q_i + \zeta_i^+ u_i^{q,+} - \zeta_i^- u_i^{q,-}\right) \qquad \qquad \quad \;\;\; \forall i \in \cN, u \in \cU \label{eqn:oqpbounds}\\
	& o_i^{p,-,u} \leq \overline{o}^p_i \qquad \qquad \qquad \qquad \qquad \qquad \qquad \qquad \;\; \forall i \in \cN, u \in \cU \label{eqn:opmbounds}\\
	& o_i^{q,-,u} \leq \overline{o}^q_i \qquad \qquad \qquad \qquad \qquad \qquad \qquad \qquad \;\; \forall i \in \cN, u \in \cU \label{eqn:oqmbounds}\\
	& o_i^{p,+,u}, o_i^{p,-,u}, o_i^{q,+,u}, o_i^{q,-,u} \geq 0 \qquad \qquad \qquad \qquad \;\;\; \forall i \in \cN, u \in \cU. \label{eqn:opositive}
	\end{align}
\end{subequations}

Model~\eqref{prob:rACOPF} seeks a first stage vector of generation dispatch decisions, \((s^p, s^q)\), that minimizes controllable generation cost. All other decision variables, including power compensations, voltages and phase angles at buses, as well as power flow on lines, adapt to the realization of uncertainty. Constraints~\eqref{eqn:spConstr1}-\eqref{eqn:sqConstr1} replicate the simple bounds on injections \eqref{eqn:spConstr}-\eqref{eqn:sqConstr}, constraints~\eqref{eqn:Pk}-\eqref{eqn:Qk} replicate the power flow equations \eqref{eqn:pf} for each \(u \in \cU\), and constraints~\eqref{eqn:thetaDiff}-\eqref{eqn:angCons} similarly replicate~\eqref{eqn:angleDiff}-\eqref{eqn:angConstr}. Constraints~\eqref{eqn:balanceP} and \eqref{eqn:balanceQ} modify constraints~\eqref{eqn:balance} by incorporating the deviation variables, whose values are limited by~\eqref{eqn:oppbounds}-\eqref{eqn:oqmbounds}. The maximum adjustment at a bus, due to traditional generators, is denoted by \(\overline{o}\). Net load uncertainty includes generation uncertainty, due to renewable sources and demand uncertainty. When an uncertain parameter is larger than its nominal value, this can be because load is low or because RES generation is high. In the latter case, we allow for curtailment of RES generation. Parameters \(\zeta^+_i\) and \(\zeta_i^-\) represent the fraction of total uncertainty due to RES generation, and \(h^p_i\) and \(h^q_i\) denote nominal renewable generation. The right-hand sides of constraints~\eqref{eqn:oppbounds} and~\eqref{eqn:oqpbounds} capture the option for curtailment, and we discuss this in greater detail in Section~\ref{subsec:alphSel}. It is well known that the power flow equations~\eqref{eqn:pf}, as well as the shunt components in~\eqref{eqn:balanceP} and~\eqref{eqn:balanceQ}, are nonconvex, and so model~(\ref{prob:rACOPF}) is an infinite-dimensional nonconvex robust optimization problem with recourse.

There are multiple convex relaxation schemes for ACOPF problems. In the semidefinite programming relaxation of \citet{bai2011semidefinite} and \citet{bai2008semidefinite}, the vector of voltage variables in model~\eqref{prob:rACOPF} is re-expressed as a higher-dimensional matrix, coupled with a rank-one constraint and a positive semidefinite requirement, along with a collection of linear constraints. After dropping the rank-one constraint, the relaxed problem becomes an SDP and can be solved by an interior point method. Experience on realistically sized instances suggests that such SDP formulations are computationally expensive, and so \citet{jabr2006radial} proposes a further relaxation of the positive semidefinite constraint, yielding a second-order cone program (SOCP). Although computationally easier to solve, this SOCP relaxation has the disadvantage of tending to exhibit a larger optimality gap than the SDP relaxation for many test cases. See \citet{low2014convex} for a detailed review of such SDP and SOCP relaxations.

We use the convex relaxation that \citet{coffrin2016qc} call the quadratic convex (QC) relaxation. While the QC formulation is also an SOCP, it tightens the relaxation compared to previous SOCP formulations.~\citet{coffrin2016qc} suggest relaxing equation~\eqref{eqn:pf} by replacing trigonometric functions by quadratic functions and using a McCormick relaxation to linearize the multi-linear terms. {The quadratic terms in~\eqref{eqn:balanceP} and~\eqref{eqn:balanceQ}, \(v_i^2\), are replaced by \(\hat{v}_i\), which is constrained by a linear upper bound and a quadratic lower bound.} The formulation of the QC relaxation of model~\eqref{prob:rACOPF} is detailed in Appendix~\ref{appen:QCRelaxation}. Here, we use generic notation \(x\) to represent the system configuration and express the convex relaxation of model~\eqref{prob:rACOPF}, as formulated in Appendix~\ref{appen:QCRelaxation}, more compactly in model~\eqref{prob:rcACOPF} below. 

In what follows, we largely use a vector form to denote the controllable and uncontrollable injections for conciseness. A symbol without a subscript represents a vector, while a subscript-indexed symbol represents a specific component within that vector. Here we denote \(u^{p,+}\), \(u^{p,-}\), \(u^{q,+}\) and \(u^{q,-}\) as \(|\cN|\)-dimensional vectors of uncontrollable active and reactive deviation. Similar notation is used for \(\overline{u}, \underline{u}, u^0, s, \overline{s}\) and \(\underline{s}\) as:
\[\overline{u} = \begin{bmatrix}
\overline{u}^p\\
\overline{u}^q
\end{bmatrix},\ \underline{u} = \begin{bmatrix}
\underline{u}^p\\
\underline{u}^q
\end{bmatrix},\ u^0 = \begin{bmatrix}
u^{p,0}\\
u^{q,0}
\end{bmatrix},\ s = \begin{bmatrix}
s^p\\
s^q
\end{bmatrix},\ \overline{s} = \begin{bmatrix}
\overline{s}^p\\
\overline{s}^q
\end{bmatrix},\ \underline{s} = \begin{bmatrix}
\underline{s}^p\\
\underline{s}^q
\end{bmatrix}.\]
In this context, we also represent the active and reactive controllable injections of each bus \(i \in \cN\) as a linear transformation of the vector of generation \(s^p\) and \(s^q\):
\[Ds^p = \left[
\sum_{g \in \cG_i} s^p_g \right]_{i \in \cN} \quad \text{ and }\quad Ds^q = \left[
\sum_{g \in \cG_i} s^q_g \right]_{i \in \cN}, \]
for an appropriate matrix \(D\). We use \(\zeta^+\) and \(\zeta^-\) to denote \(|\cN| \times |\cN|\) diagonal matrices with entries \(\zeta_i^+\) and \(\zeta_i^-\), \(\forall i \in \cN\). This leads to the following compact formulation for the convex relaxation of model~\eqref{prob:rACOPF}: 
\begin{subequations} \label{prob:rcACOPF}
	\begin{align}
	\min \quad & c(s) &  \\
	\text{s.t.} \quad & \underline{s} \leq s \leq \overline{s} \qquad \qquad \label{eqn:csConstr}\\
	& Ax^u \leq b \qquad \qquad \qquad \qquad \qquad \qquad \qquad \quad \; \forall u \in \cU \label{eqn:clinear}\\
	& \|B_i x^u + a_i \|_2 \leq e_i^\top x^u + f_i \qquad \qquad \qquad \qquad \forall i = 1, \dots, m_c , \ u \in \cU \label{eqn:csoc}\\
	& A^{op} x^u \leq \overline{o}^{p} + h^p + \zeta^+ u^{p,+} - \zeta^- u^{p,-} \qquad \qquad \forall u \in \cU \label{eqn:cop}\\
	& A^{oq} x^u \leq \overline{o}^{q} + h^q + \zeta^+ u^{q,+} - \zeta^- u^{q,-} \qquad \qquad \forall u \in \cU \label{eqn:coq}\\
	& A^p x^u = Ds^p + u^{p,0} + u^{p,+} - u^{p,-} \qquad \qquad \quad \forall u \in \cU \label{eqn:cbalanceP} \\
	& A^q x^u = Ds^q + u^{q,0} + u^{q,+} - u^{q,-} \qquad \qquad \quad \forall u \in \cU. \label{eqn:cbalanceQ}
	\end{align}
\end{subequations}

Constraint~\eqref{eqn:csConstr} replicates the analogous constraints~\eqref{eqn:spConstr1} and \eqref{eqn:sqConstr1}. The linear inequality~\eqref{eqn:clinear} and the SOCP constraint~\eqref{eqn:csoc} capture constraint~\eqref{eqn:lineCons}, and the relaxation of the nonlinear terms in constraints~\eqref{eqn:Pk}-\eqref{eqn:Qk} and~\eqref{eqn:balanceP}-\eqref{eqn:balanceQ}, while the linear inequality~\eqref{eqn:clinear} also includes~\eqref{eqn:thetaDiff}, \eqref{eqn:volCons}-\eqref{eqn:angCons}, and~\eqref{eqn:opmbounds}-\eqref{eqn:opositive}. Constraints~\eqref{eqn:cop} and~\eqref{eqn:coq} match their counterparts~\eqref{eqn:oppbounds} and~\eqref{eqn:oqpbounds}. Finally, constraints~\eqref{eqn:cbalanceP} and \eqref{eqn:cbalanceQ} replicate linearized constraints~\eqref{eqn:balanceP} and \eqref{eqn:balanceQ}. Model~(\ref{prob:rcACOPF}) can also represent the robust convex relaxation of the ACOPF problem in which we replace the QC relaxation with alternative convex relaxations discussed in \citet{bai2008semidefinite,jabr2006radial} and \citet{kocuk2016strong}.

Model~\eqref{prob:rcACOPF} is an infinite-dimensional convex optimization problem, and is an example of robust optimization with recourse. Such models have been discussed in the context of linear programming in \citet{terry2009thesis} and \citet{thiele2009robust}. In power systems optimization, similar formulations have been applied to unit commitment problems \citep{jiang2012robust,jiang2014two}, joint reserve and energy dispatch~\citep{zugno2015robust}, and microgrid operations~\citep{khodaei2014resiliency}. There has been limited work of which we are aware involving conic programming, or more general convex programming, variants of such models~\citep[although we can point to][]{terry2009thesis}. In the next section we discuss reformulation of this problem and an algorithm to solve the reformulated finite-dimensional problem. 

\section{A Cutting-Plane Method}
\label{sec:cutting}
In this section we propose a cutting-plane method to solve the robust convex optimization problem~\eqref{prob:rcACOPF}. To facilitate decomposition of model~\eqref{prob:rcACOPF}, we project onto the set of feasible \((s^p,s^q)\) variables, and we employ an outer approximation to iteratively characterize this set. At each iteration, given a candidate solution, we compute, and add to the master problem, the most-violated inequality. With introduction of auxiliary binary decision variables, we can transform what would otherwise be an infinite number of constraints in \eqref{prob:rcACOPF} into a finite formulation and obtain a solution within some acceptable tolerance from the feasible set. 
\subsection{Master Problem and Subproblems}
Similar to the generalized Benders' decomposition method of \citet{geoffrion1972generalized}, we can rewrite model~\eqref{prob:rcACOPF} as: 
\begin{subequations}
	\begin{align}
	\min \quad &c(s)\\
	\text{s.t.} \quad &  s \in \cS \equiv {\cap_{u \in \cU} \cS^u} \cap \{s \mid \underline{s} \leq s \leq \overline{s} \},  \label{eqn:inSet}
	\end{align}
\end{subequations}
where for each \(u \in \cU\) we have the induced feasibility set
\begin{equation} \label{eqn:calSudef}
\cS^u = \left\{ s\ \left| \ \exists x\ \text{s.t. }
\begin{aligned}
& A x \leq b\\
& \| B_i x + a_i\|_2 \leq e_i^\top x + f_i \quad \forall i = 1, \dots, m_c \\
& A^{op} x \leq \overline{o}^{p} + h^p + \zeta^+ u^{p,+} - \zeta^- u^{p,-}\\
& A^{oq} x \leq \overline{o}^{q} + h^q + \zeta^+ u^{q,+} - \zeta^- u^{q,-}\\
& A^p x = D s^p + u^{p,0} + u^{p,+} - u^{p,-}\\ 
& A^q x = D s^q + u^{q,0} + u^{q,+} - u^{q,-}\\
\end{aligned}\right.
\right\}.
\end{equation}

This reformulation motivates a cutting-plane algorithm in which we iteratively solve a master problem and a collection of SOCP subproblems. In addition to the simple bounds, constraint~\eqref{eqn:inSet} requires that \(s\) be in the intersection of \(\cS^u,\ \forall u \in \cU\). When we solve the subproblems, we either find a feasible \(x^u\) for each \(u \in \cU\), or we generate linear cuts, each of which is a valid outer approximation for \(\cS\). The master \((M)\) and the subproblem \((S^u)\) are as follows:
\begin{subequations}\label{prob:master}
	\begin{align}
	(M) \quad V^* = \min \quad & c(s) \\
	\text{s.t.} \quad & \underline{s} \leq s \leq \overline{s} \\
	& -{\lambda^{p,k}}^\top D s^p - {\lambda^{q,k}}^\top D s^q + z^{k} \leq 0  \qquad \qquad \forall k = 1, 2, \dots \label{eqn:cut}
	\end{align}
\end{subequations}
\vspace{-0.5cm}
\begin{subequations}\label{prob:sub}
	\begin{align}
	(S^u) \quad \min \quad & 1^\top \left(l^{p,+}+l^{p,-}+l^{q,+}+l^{q,-} \right) \\
	\text{s.t.} \quad & A x \leq b  \label{eqn:linearCon}\\
	& \|B_i x + a_i \|_2 \leq e_i^\top x + f_i \qquad \qquad \qquad \qquad \; \forall i = 1,\dots, m_c \label{eqn:coneCon}\\
	& A^{op} x \leq \overline{o}^{p} + h^p + \zeta^+ u^{p,+} - \zeta^- u^{p,-} \label{eqn:cop1}\\
	& A^{oq} x \leq \overline{o}^{q} + h^q + \zeta^+ u^{q,+} - \zeta^- u^{q,-} \label{eqn:coq1}\\
	& A^p x + l^{p,+} -l^{p,-} = D \hat{s}^p + u^{p,0} + u^{p,+} - u^{p,-} \label{eqn:pCon}\\ 
	& A^q x + l^{q,+} -l^{q,-} = D \hat{s}^q + u^{q,0} + u^{q,+} - u^{q,-} \label{eqn:qCon}\\
	& l^{p,+}, l^{p,-}, l^{q,+} , l^{q,-} \geq 0. \label{eqn:lnonneg}
	\end{align}
\end{subequations}

Here, \(\lambda^p\) and \(\lambda^q\) denote dual variables for constraints~\eqref{eqn:pCon} and~\eqref{eqn:qCon}, respectively. We introduce artificial variables \(l^{p,+}, l^{p,-}, l^{q,+}\), and \(l^{q,-}\) in $(S^u)$ to represent violation of the power balance constraints. At optimality, if the vector \((l^{p,+}, l^{p,-}, l^{q,+}, l^{q,-})\) is nonzero then for the given master solution, $\hat{s}$, and uncertainty realization, $u$, there is no feasible \(x\) such that constraints in~\eqref{eqn:calSudef} can be satisfied. As a result, by a generalized theorem of the alternative~\citep[see][Section~5.8]{boyd2004convex} a feasibility cut~\eqref{eqn:cut} can be generated to ensure the master~$(M)$ cannot again select $\hat{s}$ in subsequent iterations. (We return to this in detail in Lemma~\ref{lemma:subset}.) Index~\(k\) corresponds to the \(k\)-th feasibility cut, and the scalar cut intercept, \(z^k\), accounts for all objective function terms in the dual of model~\eqref{prob:sub} that do not involve \(\hat{s}^p\) and \(\hat{s}^q\). 

The decomposition algorithm is not directly implementable because there are infinitely many subproblems, \((S^u)\). So, we instead seek the most violated inequality across all elements of the uncertainty set, which results in the following max-min problem:
\begin{equation}\label{prob:primalSub}
	\begin{aligned}
	\max_{u \in \cU} \quad \min_{l,x} \quad & 1^\top \left(l^{p,+}+l^{p,-}+l^{q,+}+l^{q,-} \right)\\
	\text{s.t.} \quad & \eqref{eqn:linearCon}\mbox{-}\eqref{eqn:lnonneg}.
	\end{aligned}
\end{equation}

To reformulate model~\eqref{prob:primalSub} in a computationally tractable manner we first take the dual of the inner minimization. We denote the dual variables for constraints \eqref{eqn:linearCon} and \eqref{eqn:cop1}-\eqref{eqn:qCon} by \(\lambda, \lambda^{op}, \lambda^{oq}, \lambda^p\), and \(\lambda^q\). For the second-order cone constraints in~\eqref{eqn:coneCon}, we denote the dual variables as \((\mu_i,\nu_i), \ i = 1, \dots, m_c\). 
Then taking the dual yields: 
\begin{subequations} \label{prob:dualSub}
	\begin{align}
	\max_{u \in \cU} \ \max_{\lambda, \lambda^{op}, \lambda^{oq}, \lambda^p, \lambda^q, \mu, \nu} \quad &- \lambda^\top b -\sum_{i = 1}^{m_c} \left( \nu_i f_i + \mu_i^\top a_i \right) - \nonumber \\ 
	& {\lambda^{op}}^\top (\overline{o}^p + h^p + \zeta^+ u^{p,+} - \zeta^- u^{p,-}) - {\lambda^{oq}}^\top (\overline{o}^q + h^q + \zeta^+ u^{q,+} - \zeta^- u^{q,-}) - \nonumber \\
	&  {\lambda^p}^\top \left(D \hat{s}^p+ u^{p,0} + u^{p,+} - u^{p,-}\right)  -{\lambda^q}^\top \left(D \hat{s}^q+ u^{q,0} + u^{q,+} - u^{q,-}\right)  \label{eqn:objdual} \\
	\text{s.t.} \qquad \qquad & \lambda^\top A + {\lambda^{op}}^\top A^{op} +  {\lambda^{oq}}^\top A^{oq} + {\lambda^p}^\top A^p + {\lambda^q}^\top A^q  \nonumber \\
	& \qquad \qquad - \sum_{i = 1}^{m_c} \left( {\mu_i}^\top B_i + \nu_i e_i^\top \right) = 0^\top \label{eqn:budgetlinCombine}\\
	& \| \mu_i \|_2 \leq \nu_i \qquad \qquad \quad \; \forall i = 1, \dots, m_c \label{eqn:budgetSOCP}\\
	& -1 \leq \lambda^p_i \leq 1 \qquad \qquad \forall i \in \cN \label{eqn:budgetLambdaP} \\
	& -1 \leq \lambda^q_i \leq 1 \qquad \qquad \forall i \in \cN \label{eqn:budgetLambdaQ} \\
	&\lambda,\lambda^{op},\lambda^{oq} \geq 0. 
	\end{align}
\end{subequations}
The optimal value of the inner maximization problem is a convex function of the \(4|\cN|\)-dimensional vector \(u\). The outer problem maximizes this convex function over the polytope~\(\cU\). We know that an optimal solution can be obtained by restricting attention to the extreme points of \(\cU\) \citep[e.g.,][Proposition~2.4.1]{bertsekas_2009}, and we denote this set by~\(\cU^E\). When \(\cU\) has an amenable structure this can allow for a finite reformulation. In what follows, we assume that \(\cU\) is defined as in equation~\eqref{eqn:setU}, and we introduce \(2|\cM|\) binary variables to model the extreme points:
\begin{subequations}
	\label{prob:dualSubIBudget}
	\begin{align}
	(SDI) \ \max \quad &- \lambda^\top b  -\sum_{i = 1}^{m_c} \left( \nu_i f_i + \mu_i^\top a_i \right)  - \sum_{i \in \cN} \left[ y^+_{m_i} \left(\overline{u}^p_i - u^{p,0}_i \right)(\lambda^{p}_i + \zeta^+_i \lambda^{op}_i) + \right. \nonumber \\
	& \left. y^-_{m_i} \left(\underline{u}^p_i - u^{p,0}_i \right) (\lambda^p_i + \zeta^-_i \lambda^{op}_i) + y^+_{m_i} \left(\overline{u}^q_i - u^{q,0}_i \right) (\lambda^q_i + \zeta^+_i \lambda^{oq}_i) + \right. \nonumber \\
	& \left.  y^-_{m_i} \left(\underline{u}^q_i - u^{q,0}_i \right) (\lambda^q_i + \zeta^-_i \lambda^{oq}_i)\right]
	- {\lambda^{op}}^\top (\overline{o}^p + h^p) - {\lambda^{oq}}^\top (\overline{o}^q + h^q) - \nonumber \\
	& \left[ {\lambda^p}^\top \left(D \hat{s}^p + u^{p,0} \right) + {\lambda^q}^\top \left(D \hat{s}^q + u^{q,0} \right)\right] \label{eqn:SDIobj} \\
	\text{s.t.} \quad & \eqref{eqn:budgetlinCombine} - \eqref{eqn:budgetLambdaQ} \nonumber\\
	& y_m^{+} + y_m^{-} \leq 1 \qquad \qquad \qquad \quad \forall m \in \cM \label{eqn:budget1}\\
	& \sum_{m \in \cM} \left(y_m^{+} + y_m^{-} \right) \leq \Gamma \label{eqn:budget}\\
	& y_m^+, y_m^- \in \{0,1\} \qquad \qquad \quad \;\; \forall m \in \cM \\
	& \lambda,\lambda^{op},\lambda^{oq} \geq 0. \label{eqn:budgetLambda}
	\end{align}
\end{subequations}

For \(i \in \cN_m\), we use \(y^+_m\) to indicate that \(u_i^{p,+}\) and \(u_i^{q,+}\) take their upper bound and \(y^-_m\) to indicate that \(u_i^{p,-}\) and \(u_i^{q,-}\) take their lower bound. The objective function in~\eqref{eqn:SDIobj} includes bilinear terms such as \(\lambda_i^p y_{m_i}\), where \(m_i\) is used to indicate bus \(i\)'s cluster. These are linearized in a straightforward way as shown in Appendix~\ref{appen:QCDual}. We use $(SDI)$ to denote the subproblem in dual form with integer-constrained variables.

Constraints~\eqref{eqn:budget1} enforce that at most one, instead of exactly one, end point of the feasible range is taken, and constraint \eqref{eqn:budget} requires that at most \(\Gamma\) clusters of uncontrollable injections take their end-point value. We include in Appendix~\ref{appen:QCDual} the full formulation of model~\eqref{prob:dualSubIBudget}, which is derived from the convex quadratic relaxation detailed in Appendix~\ref{appen:QCRelaxation}.

Algorithm~\ref{alg:Cut} formalizes our cutting-plane procedure, where at iteration \(k\) we solve the master problem, \((M)\), and obtain \((\hat{s}^{p,k},\hat{s}^{q,k})\). Then, using the uncertainty set defined in equation~\eqref{eqn:setU}, we solve model~\eqref{prob:dualSubIBudget}, and denote the optimal value by \(z_{feas}^k\) and part of the optimal solution by \(\lambda^{p,k}, \lambda^{q,k}\). If \(z^k_{feas} > 0\), we then generate the most violated cut as:
\begin{equation}
z_{feas}^k - {\lambda^{p,k}}^\top D (s^p - \hat{s}^{p,k}) - {\lambda^{q,k}}^\top D(s^q - \hat{s}^{q,k}) \leq 0. \label{eqn:cut1}
\end{equation}
With \(z^k = z_{feas}^k + {\lambda^{p,k}}^\top D \hat{s}^{p,k} + {\lambda^{q,k}}^\top D \hat{s}^{q,k}\), inequality~\eqref{eqn:cut1} is of form~\eqref{eqn:cut}. 
\begin{algorithm}
	\caption{Cutting-plane algorithm for model~\eqref{prob:rcACOPF}}
	\label{alg:Cut}
	\begin{algorithmic}[1]
		\State Initialize with iteration number \(k := 1\) and tolerance \(\epsilon > 0\);
		\State Solve master problem \((M)\) and obtain solution \((\hat{s}^{p,k}, \hat{s}^{q,k})\) and optimal value \(V^*\); 
		\State Solve \((SDI)\) with \((\hat{s}^{p,k}, \hat{s}^{q,k})\) and obtain solution \((\lambda^{p,k}, \lambda^{q,k})\) and optimal value \(z_{feas}^k\);
		\While{\(z_{feas}^k > \epsilon\)} 
		\State Append \(z_{feas}^k -{\lambda^{p,k}}^\top D (s^p - \hat{s}^{p,k}) - {\lambda^{q,k}}^\top D (s^q - \hat{s}^{q,k}) \leq 0\) to constraints~\eqref{eqn:cut} of \((M)\);
		\State Let \(k := k+1\);
		\State Solve \((M)\) and obtain solution \((\hat{s}^{p,k}, \hat{s}^{q,k})\);
		\If{\((M)\) is feasible}
		\State Obtain optimal value \(V^*\); 
		\Else
		\State Stop and return the status of infeasibility;
		\EndIf
		\State Solve \((SDI)\) with \((\hat{s}^{p,k}, \hat{s}^{q,k})\) and obtain solution \((\lambda^{p,k}, \lambda^{q,k})\) and optimal value \(z_{feas}^k\); $\!\!\!\!\!\!$
		\vspace{0.1cm}
		\EndWhile{\textbf{end while}}
		\State Output \(V^*\) as a lower bound on the optimal value of model~\eqref{prob:rcACOPF}, and output \((\hat{s}^{p,k}, \hat{s}^{q,k})\) as an \(\epsilon\)-feasible solution.
	\end{algorithmic}
\end{algorithm}
	
	\subsection{Convergence of the Algorithm}
	Given \(\epsilon > 0\), we show that in a finite number of iterations Algorithm~\ref{alg:Cut} either finds an \(\epsilon\)-feasible solution or terminates with a statement that model~\eqref{prob:rcACOPF}---and hence model~\eqref{prob:rACOPF}---is infeasible. Furthermore, the sequence of solutions generated by our algorithm converges to an optimal solution when the tolerance in the algorithm is \(\epsilon = 0\). We make the notion of an ``\(\epsilon\)-feasible" solution precise as follows.
	\begin{definition}
		Let \(\epsilon >0 \). An \(s \in \{s \mid \underline{s} \leq s \leq \overline{s} \}\) is \(\epsilon\)-feasible to model~\eqref{prob:rcACOPF} if for each \(u \in \cU\) there exists an \(\hat{s} \in \mathcal{B}_{\epsilon}(s)\) such that 
		\begin{equation}
		\left\{x\ \left|\ \begin{aligned}
		& Ax \leq b \\
		& \|B_i x + a_i \|_2 \leq e_i^\top x + f_i \quad \forall i = 1, \dots, m_c\\
		& A^{op} x \leq \overline{o}^{p} + h^p + \zeta^+ u^{p,+} - \zeta^- u^{p,-} \qquad A^{oq} x \leq \overline{o}^{q} + h^q + \zeta^+ u^{q,+} - \zeta^- u^{q,-}\\
		& A^p x = D \hat{s}^p + u^{p,0} + u^{p,+} - u^{p,-} \qquad \quad  A^q x = D \hat{s}^q + u^{q,0} + u^{q,+} - u^{q,-}
		\end{aligned} 
		\right\} \right.  \neq \emptyset,
		\end{equation}
		where \(\mathcal{B}_{\epsilon}(s)\) is an \(l_1\) ball with center \(s\) and radius \(\epsilon\).
	\end{definition}
	For an \(\epsilon\)-feasible $s$, the \(l_1\) distance from \(s\) to the corresponding \(\hat{s}\) is at most \(\epsilon\) for each \(u \in \cU\). The definition does not ensure that there is a uniform \(\hat{s}\) that works for all \(u \in \cU\). To establish convergence properties of Algorithm~\ref{alg:Cut}, we make the following assumptions:
	\begin{assumption}
		Function \(c(\cdot)\) is convex and continuous on domain defined by~\eqref{eqn:inSet}. \label{assumption1}
	\end{assumption}

	\begin{assumption}
		Set \begin{equation*}
		\left\{x\ \left|\ \begin{aligned}
		& Ax \leq b\\
		& \|B_i x + a_i \|_2 \leq e_i^\top x + f_i \quad \forall i = 1, \dots, m_c\\
		& A^{op} x \leq \overline{o}^{p} + h^p + \zeta^+ u^{p,+} - \zeta^- u^{p,-}\\
		& A^{oq} x \leq \overline{o}^{q} + h^q + \zeta^+ u^{q,+} - \zeta^- u^{q,-}
		\end{aligned} 
		\right\} \right. 
		\end{equation*}  
		is non-empty, and hence model~\eqref{prob:sub} is feasible, for all \(u \in \cU\). \label{assumption2} 
	\end{assumption}

	\begin{assumption}
		Set \(\cU\) is defined by~\eqref{eqn:setU}.  \label{assumption3}
	\end{assumption}

	Assumption~\ref{assumption1} is consistent with the treatment of thermal generators in the power systems literature and, to our knowledge, in industry practice. Thermal generation costs are dominated by the cost of fuel, and are typically modeled adequately via convex piecewise linear or quadratic functions. However, the physics of some power plants, e.g., certain combined-cycle plants, dictate no-operate regions that can lead to nonsmooth, nonconvex cost functions, which we do not address. Assumption~\ref{assumption2} should hold with great generality for an actual power system because the set is a relaxation of the system's constraints, which does not include load satisfaction. Assumption~\ref{assumption3} is revisited in Section~\ref{subsec:alphSel}. We now establish convergence properties of the sequence of solutions generated by Algorithm~\ref{alg:Cut}.
	\begin{lemma} \label{lemma:convexity}
		Let \(Z^u(s)\) denote the optimal value of model~\eqref{prob:sub} for a specific \(u \in \cU\), where \(\hat{s}\) on the right-hand side of constraints~\eqref{eqn:pCon} and~\eqref{eqn:qCon} is replaced by \(s\), and let \(Z(s)\) denote the analogous optimal value for model~\eqref{prob:primalSub}. If Assumptions~\ref{assumption2} and~\ref{assumption3} hold, then both \(Z^u(\cdot)\) and \(Z(\cdot)\) are convex on the domain \(\Re^{2|\cG|}\). 
	\end{lemma}
	\begin{proof}
		[Proof of Lemma~\ref{lemma:convexity}]
		The function $Z^u(s)$ is the optimal value of model~\eqref{prob:sub}, which is feasible by Assumption~\ref{assumption2} for any \(s \in \Re^{2 |\cG|}\), and hence has a finite optimal value. Thus \(Z^u(s)\) is also the optimal value of the dual of model~\eqref{prob:sub}. The dual's feasible region is independent of \(s\), and its objective function is an affine function of \(s\). Therefore, \(Z^u(\cdot)\) is the maximum of a collection of affine functions in \(s\), and hence convex. Furthermore, \(Z(\cdot)\) is the maximum of convex functions \(Z^u(\cdot)\) over the set of \(\cU\), and so \(Z(\cdot)\) is also convex.
	\end{proof}
	
	\begin{lemma} \label{lemma:subset}
		Let \(\cS^k = \{s \mid \underline{s} \leq s \leq \overline{s} ,\ z_{feas}^j - {\lambda^{p,j}}^\top D (s^p - \hat{s}^{p,j}) - {\lambda^{q,j}}^\top D (s^q - \hat{s}^{q,j}) \leq 0,\  \forall j = 1,\dots,k \}\), where these cuts are defined in~\eqref{eqn:cut1}. If Assumptions~\ref{assumption2} and~\ref{assumption3} hold then \(\cS \subseteq \cS^k,\  \forall k = 1,2,\dots \).
	\end{lemma}
	\begin{proof}[Proof of Lemma~\ref{lemma:subset}]
		At iteration \(k\) of Algorithm~\ref{alg:Cut}, the solution to model~\eqref{prob:dualSubIBudget} specifies an element of \(\cU\) via binary variables, and we denote this element \(u^k \in \cU \). By a theorem of the alternative for an SOCP model \citep[see][Section~5.8]{boyd2004convex}, any inequality of the form~\eqref{eqn:cut1} satisfies \(\cS^{u^k} \subseteq \{s \mid z_{feas}^k - {\lambda^{p,k}}^\top D (s^p - \hat{s}^{p,k}) - {\lambda^{q,k}}^\top D (s^q - \hat{s}^{q,k}) \leq 0\}\). Since \(\cS = {\cap_{u \in \cU} \cS^u} \cap \{s \mid \underline{s} \leq s \leq \overline{s}\}\), and each cut is produced for a specific \(u\), we have that \(\cS \subseteq \cS^k \) for all \(k\). 
	\end{proof}

	\begin{theorem} \label{thm:converge}
		Let Assumptions~\ref{assumption1}-\ref{assumption3} hold, and assume that model~\eqref{prob:rcACOPF} is feasible. Let \(\epsilon = 0\), and let \(\{\hat{s}^k\}\) denote the sequence of iterates produced by Algorithm~\ref{alg:Cut}. Every limit point of this sequence solves model~\eqref{prob:rcACOPF}.
	\end{theorem}
	\begin{proof}[Proof of Theorem~\ref{thm:converge}]
		If Algorithm~\ref{alg:Cut} terminates in a finite number of iterations, then it does so with \(z^k_{feas} = 0\). In this case, the associated solution solves model~\eqref{prob:rcACOPF} by Lemma~\ref{lemma:subset} because the master problem is a relaxation, and the proof is complete. Now assume that the algorithm produces an infinite sequence of iterates, and let \(\cS\) be defined as in~\eqref{eqn:inSet}. Set \(\cS\) is compact because it is a closed subset of \(\underline{s} \leq s \leq \overline{s}\). So, \(\{\hat{s}^k\}\) has at least one limit point in \(\cS\), which we denote as \(\hat{s}\), and we let \(\mathcal{K}\) index a corresponding convergent subsequence;  i.e., \(\lim_{k \in \mathcal{K}, k \to \infty} \hat{s}^k = \hat{s}\).
		
		Solving model~\eqref{prob:dualSubIBudget} yields a \(u^k \in \cU^E \subseteq \cU\), which represents a most violated element of the uncertainty set. Because these solutions are in \(\cU^E\), there are a finite number of possibilities. So, there is at least one \(\hat{u} \in \cU^E\) that occurs infinitely many times among the iterations indexed by \(\mathcal{K}\), and we let \(\mathcal{K}' \subset \mathcal{K}\) denote such a further subsequence. Let \(k,k' \in \mathcal{K}' \) with \(k' > k\). Then we have
		\begin{align}
		z_{feas}^k &\leq {\lambda^{p,k}}^\top D (\hat{s}^{p,k'} - \hat{s}^{p,k}) + {\lambda^{q,k}}^\top D (\hat{s}^{q,k'} - \hat{s}^{q,k}) \nonumber\\
		&\leq \| {\lambda^{p,k}}\| \| D (\hat{s}^{p,k'} - \hat{s}^{p,k})\| + \|{\lambda^{q,k}}\| \|D (\hat{s}^{q,k'}- \hat{s}^{q,k})\|.
		\end{align}
		
		From constraints \eqref{eqn:budgetLambdaP} and \eqref{eqn:budgetLambdaQ}, we know \(\|\lambda^p \|\) and \(\|\lambda^q \|\) are bounded. Both \(\hat{s}^k\) and \(\hat{s}^{k'}\) converge to \(\hat{s}\) so 
		\begin{equation}
		\lim_{\substack{k \to \infty\\k' \to \infty\\k' > k\\k',k \in \mathcal{K}'}}  \left(\| {\lambda^{p,k}}\| \|(\hat{s}^{p,k'} - \hat{s}^{p,k})\| + \|{\lambda^{q,k}}\| \|(\hat{s}^{q,k'}- \hat{s}^{q,k})\| \right) = 0. \label{eqn:limitsubzero}
		\end{equation}
		
		We let \(Z^u(s)\) denote the optimal value of model~\eqref{prob:sub} for a specific \(u \in \cU\), which is equivalent to the inner minimization problem of~\eqref{prob:primalSub}, and we let \(Z(s)\) denote the optimal value of~\eqref{prob:primalSub}, where the right-hand side is parametrized by \(s\) rather than \(\hat{s}\). Thus, we have:
		\begin{equation}
		\lim_{\substack{k \to \infty \\ k \in \mathcal{K}'}} Z^{\hat{u}}(s^k) = Z^{\hat{u}}(\hat{s}) = Z(\hat{s}) \leq 0, \label{eqn:limitzero}
		\end{equation}
		where the first equality holds by continuity of \(Z^{\hat{u}}(\cdot)\) from Lemma~\ref{lemma:convexity}, and the second equality holds because \(\hat{u}\) corresponds to a most violated point of \(\cU\). Thus, \(\hat{s}\) is feasible to model~\eqref{prob:rcACOPF}. Let $z^*$ denote the optimal value of model~\eqref{prob:rcACOPF}. Then \(c(\hat{s}) \geq z^*\).
		
		By Lemma~\ref{lemma:subset}, we have \(c(\hat{s}^k) \leq z^*,\ \forall k \in \mathcal{K}\), and hence by Assumption~\ref{assumption1}, we have that
		\begin{equation}
		\lim_{\substack{k \to \infty \\ k \in \mathcal{K}}} c(\hat{s}^k) = c(\hat{s}) \leq z^*. \label{eqn:limitLB}
		\end{equation}
		Thus, \(\hat{s}\) solves model~\eqref{prob:rcACOPF}.
	\end{proof}

	Finally we show that when Algorithm~\ref{alg:Cut} terminates, it returns an \(\epsilon\)-feasible solution to model~\eqref{prob:rcACOPF} in a finite number of iterations, if model~\eqref{prob:rcACOPF} is feasible.
	\begin{theorem} \label{thm:finite}
		Let Assumptions~\ref{assumption1}-\ref{assumption3} hold, and assume that model~\eqref{prob:rcACOPF} is feasible. Let \(\epsilon > 0\). Algorithm~\ref{alg:Cut} terminates with an \(\epsilon\)-feasible solution in a finite number of iterations.
	\end{theorem}
	\begin{proof}[Proof of Theorem~\ref{thm:finite}]
		Model~\eqref{prob:rcACOPF} is feasible, and hence \(\cS \neq \emptyset\). By Lemma~\ref{lemma:subset} \(\cS \subseteq \cS^k\), which is the feasible region of model~\eqref{prob:master} for all \(k = 1,2, \dots\). Therefore, Algorithm~\ref{alg:Cut} does not terminate with a status of infeasibility because model~\eqref{prob:master} is feasible for all \(k = 1,2,\dots\).
		
		We first prove by contradiction that the algorithm terminates in a finite number of iterations. Here \(\epsilon\) only determines the stopping criterion but does not affect the cuts generated in Algorithm~\ref{alg:Cut}. Suppose Algorithm~\ref{alg:Cut} does not terminate after a finite number of iterations. Thus, we have an infinite sequence of solutions \(\{\hat{s}^k\}\), and \(Z(\hat{s}^k) > \epsilon,\ \forall k = 1,2,\dots \). By the proof of Theorem~\ref{thm:converge}, every convergent subsequence of \(\{\hat{s}^k\}\) indexed by \(\mathcal{K}\) with a limit point \(\hat{s}\) satisfies \(Z(\hat{s}) \leq 0\). We have \(\lim_{k \in \mathcal{K}, k \to \infty} Z(s^k)= Z(\hat{s}) \leq 0\) because \(Z(\cdot)\) is convex and hence continuous. However, this contradicts that \(Z(s^k) > \epsilon > 0.\ \forall k = 1,2,\dots \). Therefore, the algorithm terminates in a finite number of iterations.
		
		If Algorithm~\ref{alg:Cut} terminates in iteration \(k < \infty\), then \(z^k_{feas} \leq \epsilon\). By hypothesis, model~\eqref{prob:primalSub} is feasible and has a finite optimal value. Hence, by strong duality, the optimal value of model~\eqref{prob:dualSubIBudget} is equal to that of model~\eqref{prob:primalSub} and is at most \(\epsilon\). Let \((\hat{s}^{p,k},\hat{s}^{q,k})\) denote the input of Algorithm~\ref{alg:Cut} (step 12) to model~\eqref{prob:dualSubIBudget}, or equivalently, to model~\eqref{prob:primalSub}. For each \(u \in \cU\), let \((x^u,l^{p,+,u},l^{p,-,u},l^{q,+,u},l^{q,-,u})\) denote the optimal solution of the inner minimization problem defined in~\eqref{prob:primalSub}. For each \(u \in \cU\), let \(s^{p,u} = \hat{s}^{p,k} - l^{p,+,u} + l^{p,-,u}\) and \(s^{q,u} = \hat{s}^{q,k} - l^{q,+,u} + l^{p,-,u}\). From the formulation of model~\eqref{prob:primalSub} we know that \((s^{p,u},s^{q,u})\) yields \(\left\{x \mid \eqref{eqn:linearCon}\mbox{-}\eqref{eqn:qCon} \right\} \neq \emptyset\), and \(\| s^u - \hat{s}^k \|_1 = 1^\top(l^{p,+} + l^{p,-} + l^{q,+} + l^{q,-}) = z^k_{feas} \leq \epsilon\); i.e., \(\hat{s}\) is an \(\epsilon\)-feasible solution.
	\end{proof}
	\subsection{Improving Convergence of Algorithm~\ref{alg:Cut}}
	\label{subsec:scenAppend}	
	It is well known that cutting-plane algorithms can converge slowly; see, e.g., \citet{nemirovskii1983problem}. This can occur because master problem solutions differ dramatically from one iteration to the next. There are multiple ways to improve such algorithms ranging from trust-region methods to level-set methods to bundle methods. We studied a bundle method by adding a quadratic regularization term to the master's objective function. This approach improved computational performance, but did not facilitate solving our largest test cases. The method is detailed in Appendix~\ref{appen:regu}.
	
	Therefore we considered a second method in which we identify extreme points, \(u \in \cU^E\), for which \(\cS^u\) characterizes important parts of the boundary of \(\cS\). In a Benders' decomposition algorithm for stochastic integer programs, \citet{crainic2016partial} include a subset of the scenario subproblems in the master problem in order to reduce generation of feasibility cuts. We employ a similar approach, but we discover the requisite elements \(u\) to be added to the master problem in the cutting-plane process instead of generating them upfront. (\citealp{lorca2018adaptive} employ a similar idea in their Algorithm~2.) In each iteration of Algorithm~\ref{alg:Cut} we record the \(\hat{u}\) obtained by solving \((SDI)\), and if a particular \(\hat{u}\) is repeatedly generated \(n_c\) times then, instead of appending the linear cutting planes~\eqref{eqn:cut1}, we add \(\hat{u}\) to a set \(\widehat{\cU}\) and use the modified master program:
	\begin{subequations}
		\begin{align}
		\min \quad & c(s)\\
		\text{s.t.} \quad & \underline{s} \leq s \leq \overline{s} \\
		& -{\lambda^{p,k}}^\top D s^p - {\lambda^{q,k}}^\top D s^q + z^k \leq 0 \qquad \qquad \forall k = 1,2, \dots \\
		& A x^{{u}} \leq b \qquad \qquad \qquad \qquad \qquad \qquad \qquad \; \forall {u} \in \widehat{\cU} \\
		& \| B_i x^{{u}} + a_i \|_2 \leq e_i^\top x^{u} + f_i, \qquad \qquad \qquad \;\;\; \forall i = 1, \dots, m_c,\ {u} \in \widehat{\cU} \\
		& A^{op} x^{{u}} \leq \overline{o}^{p} + h^p + \zeta^+ {u}^{p,+} - \zeta^- u^{p,-} \qquad \quad \forall {u} \in \widehat{\cU}\\
		& A^{oq} x^{{u}} \leq \overline{o}^{q} + h^q + \zeta^+ {u}^{q,+} - \zeta^- u^{q,-} \qquad \quad \forall {u} \in \widehat{\cU}\\
		& A^p x^{{u}} = D s^p + u^{p,0} + {u}^{p,+} - {u}^{p,-}\qquad \qquad \forall {u} \in \widehat{\cU} \\
		& A^q x^{{u}} = D s^q + u^{q,0} + {u}^{q,+} - {u}^{q,-} \qquad \qquad \forall {u} \in \widehat{\cU}.
		\end{align}
	\end{subequations} 

\section{Experimental Results}
\label{sec:exp}
In this section, we describe computational results to help understand the nature of our robust convex optimization problem and the performance of Algorithm~\ref{alg:Cut}, along with enhancements to that algorithm. The optimal value of model~\eqref{prob:rcACOPF} is a lower bound on that of the nonconvex model~\eqref{prob:rACOPF}. It is important to assess the tightness of this lower bound, while also answering the question of whether the robust solution generated by Algorithm~\ref{alg:Cut} is feasible for model~\eqref{prob:rACOPF}, at least for a selection of points from the uncertainty set. Doing so helps assess the robustness of our solution.

Throughout this section we use Algorithm~\ref{alg:Cut} with the scenario-appending technique of Section~\ref{subsec:scenAppend}. We use test cases from NESTA, the NICTA Energy System Test Case Archive \citep{coffrin2014nesta}. We select IEEE cases with 5, 9, 14, 118, and 300 buses and the Polish system winter peak cases with 2383 and 2746 buses. We refer to these by the number of buses (e.g., Case 5). All tests are run on a server with 20 Intel Xeon cores at 3.1 GHz and 256 GB of RAM. All models are constructed using version 0.18.0 of the JuMP package \citep{DunningHuchetteLubin2017} on the Julia platform. The mixed integer second-order cone programs (MISOCPs) and SOCPs are solved by Gurobi 7.52 \citep{gurobi2016}, where we set the option ``NumericFocus" to 3 for Case 2383 and Case 2746. All nonconvex optimization problems are solved by Ipopt 3.12.1 \citep{wachter2006implementation}, with the linear solver MA27. Prior to solving model~\eqref{prob:rcACOPF}, we run a bound-tightening process to improve the quality of the QC relaxation. This process is detailed in Appendix~\ref{appen:boundTightening} and involves solving a sequence of SOCPs. The appendix includes the requisite computational effort, which can be significant for cases with many buses. The results that we report in this section do {\em not} include time to carry out the bound-tightening process, which is performed once for each case as a preprocessing step.

We first introduce some modeling specifics used to build our test instances. Then we detail the tests to characterize our robust convex relaxation of the ACOPF problem and the computational performance of our algorithms.

\subsection{Modeling and Implementation Details}

\subsubsection{Uncertainty Set and Recourse Bounds} \label{subsec:alphSel}
We first specify construction of the uncertainty set, $\cU$, which includes both generation and demand uncertainty. Then we discuss two specific parameter selection schemes used in our tests. For each bus, \(i \in \cN\), nominal values of uncertain demand, \((d^p_i, d^q_i) \geq 0\), are known from the NESTA datasets. We model uncertain renewable generation at a subset of buses, \(\cN_G\), where selection of this subset is detailed in Appendix~\ref{appen:facLoc}. The nominal active generation from renewables is given by \(h^p_i = 0.05 {|\cN_G|}^{-1} \sum_{i \in \cN} d_i^p\) for \(i \in \cN_G\) and is zero otherwise. We fix the constant power factor at \(\gamma = 98\%\) and calculate \(h_i^q = \sqrt{\frac{1}{\gamma^2} - 1} \, h_i^p\). We assume the maximum allowable deviation of both generation and demand is a percentage of their nominal values, and the positive and negative deviation can differ. Therefore, we can parametrize the deviation by a set of percentages \((\alpha^{h,+},\alpha^{h,-},\alpha^{d,+},\alpha^{d,-})\), with \(\alpha^{h,+},\alpha^{h,-},\alpha^{d,+},\alpha^{d,-} \in \left[0,1\right]\):
\begin{subequations}
	\label{eqn:alphapmdef}
	\begin{align}
	&\overline{u}^p_i = (1+\alpha^{h,+})h_i^p - (1-\alpha^{d,+})d_i^p & \overline{u}^q_i = (1+\alpha^{h,+})h_i^q - (1-\alpha^{d,+})d_i^q \\
	&\underline{u}^p_i = (1-\alpha^{h,-})h_i^p - (1+\alpha^{d,-})d_i^p & \underline{u}^q_i = (1-\alpha^{h,-})h_i^q - (1+\alpha^{d,-})d_i^q .
	\end{align}
\end{subequations}
Handling demand spikes, or more generally, the right-skewed nature of electrical load~\citep[e.g.,][]{maisano2016lognormal,singh2010statistical} can pose challenges for system operators, and so we model asymmetric uncertainty sets. In particular, at bus \(i\), we set \(\alpha^{d,-} = 5 \alpha^{d,+}\) to focus on large negative deviation. We consider symmetric generation uncertainty, i.e., \(\alpha^{h,+} = \alpha^{h,-}\), which is commonly used for generation uncertainty~\citep[e.g.,][]{jiang2012robust,conejo2018adaptive}.

We assume that at each bus the upper bounds on recourse decisions in constraints~\eqref{eqn:oppbounds}-\eqref{eqn:oqmbounds} have the same ratio \(\beta\) to their corresponding maximum generation level; i.e.,
\begin{align}
& \overline{o}^p_i = \beta \sum_{g \in \cG_i} \overline{s}_g^p & \overline{o}^q_i = \beta \sum_{g \in \cG_i} \overline{s}_g^q \qquad \qquad \forall i \in \cN. \label{eqn:betadef}
\end{align}
The curtailment coefficients \(\zeta^+\) and \(\zeta^-\) can be derived under the current setup as follows:
\begin{align}
& \zeta_i^+ = \frac{\alpha^{h,+} h^p_i}{\alpha^{h,+} h^p_i + \alpha^{d,+}d^p_i} & \zeta_i^- = \frac{\alpha^{h,-} h^p_i}{\alpha^{h,-} h^p_i + \alpha^{d,-}d^p_i} \qquad \qquad  \forall i \in \cN .
\end{align}

Power systems are distinguished by numerous characteristics. The use of \(\alpha\) and \(\beta\) sketched above provides a relatively simple  way of parameterizing the tests that follow.

\subsubsection{Measure of Infeasibility} \label{subsec:infdef}
We measure infeasibility of a set point as follows. Given an \((\hat{s}^p,\hat{s}^q)\), it is possible that the nonconvex model~\eqref{prob:rACOPF} is infeasible for one or more values of \((u^{p,+},u^{p,-},u^{q,+},u^{q,-}) \in \cU\). Therefore, we modify the model to allow for additional flexibility in satisfying constraints~\eqref{eqn:balanceP} and~\eqref{eqn:balanceQ} through variables \(l_i^{p,+},l_i^{p,-},l_i^{q,+},l_i^{q,-}\). For a given \(u \in \cU\) we measure the magnitude of infeasibility by:
\begin{subequations}
	\label{prob:infeas}
	\begin{align}
	\tII(\hat{s},u) = \min \quad & \sum_{i \in \cN} \frac{\left(l^{p,+}_i + l^{p,-}_i + l^{q,+}_i + l^{q,-}_i\right)}{\sum_{i \in \cN} \left(|d^p_i| + |d^q_i|\right)} \\
	\text{s.t.} \quad & \text{constraints }\eqref{eqn:spConstr1}\mbox{-}\eqref{eqn:angCons} \nonumber \\
	& \sum_{k = (i,j,n) \in \cA} P_k + g_i^{sh} v_i^2+ o_i^{p,+} - o_i^{p,-} + l_i^{p,+} - l_i^{p,-}  \nonumber\\
	& \qquad = \sum_{g \in \cG_i} \hat{s}_g^p + (u_i^{p,0} + u_i^{p,+} - u_i^{p,-}) \qquad \qquad  \forall i \in \cN \label{eqn:balancePI}\\
	& \sum_{k = (i,j,n) \in \cA} Q_k - b_i^{sh} v_i^2 + o_i^{q,+} - o_i^{q,-} + l_i^{q,+} - l_i^{q,-} \nonumber\\
	& \qquad = \sum_{g \in \cG_i} \hat{s}_g^q + (u_i^{q,0} + u_i^{q,+} - u_i^{q,-}) \qquad \qquad \forall i \in \cN \label{eqn:balanceQI}\\
	& \text{constraints }\eqref{eqn:oppbounds}\mbox{-}\eqref{eqn:oqmbounds} \nonumber \\
	& o_i^{p,+}, o_i^{p,-}, o_i^{q,+}, o_i^{q,-} \geq 0 \qquad \qquad \qquad \qquad \quad \; \; \forall i \in \cN\\
	& l_i^{p,+}, l_i^{p,-}, l_i^{q,+}, l_i^{q,-} \geq 0 \qquad \qquad  \qquad \qquad \qquad \forall i \in \cN.
	\end{align}
\end{subequations}
The infeasibility measure, \(\tII(s,u)\), given by nonconvex model~\eqref{prob:infeas} yields a minimum normalized adjustment to the right-hand side of constraints~\eqref{eqn:balancePI} and~\eqref{eqn:balanceQI} needed to construct a feasible recourse solution for a given \(u \in \cU\). 

\subsubsection{Solving Nonconvex Problems} \label{subsec:nonconvex}
In this section, we discuss two nonconvex problems that we solve. First, we solve model~\eqref{prob:infeas} using Ipopt to measure the infeasibility of a solution~\(\hat{s}\) for a given scenario \(u\). Since we may not find a global minimum, we obtain an upper bound on the infeasibility measure for a fixed $u$. To quantify the gap on this infeasibility measure, we obtain a lower bound by solving the convex relaxation of model~\eqref{prob:infeas}, which is model~\eqref{prob:sub} under the same value of \(u\). We denote this lower bound by \(\underline{\tII}(\hat{s},u)\). 

The second nonconvex problem that we face is the robust ACOPF model. Model~\eqref{prob:rcACOPF}'s optimal value provides a lower bound for that of its nonconvex counterpart, i.e., model~\eqref{prob:rACOPF}. However, it is difficult to obtain a tight upper bound on model~\eqref{prob:rACOPF}'s optimal value~\citep[e.g.,][]{nguyen2019inner} because of the challenge in guaranteeing feasibility: for model~\eqref{prob:rcACOPF}, we can restrict attention to \(\cU^E\), but in the nonconvex setting, the worst-case violation need not be at an extreme point of \(\cU\). Ensuring feasibility of a candidate solution, $\hat{s}$, to model~\eqref{prob:rACOPF} requires checking every scenario, \(u \in \cU\), or equivalently, solving the nonconvex variant of model~\eqref{prob:primalSub}, which is difficult. Instead, we do the following. First, we compute a candidate solution, which we denote $s^*$, by finding a local minimum of model~\eqref{prob:rACOPF} in which $\cU$ is replaced by $\cU^E$. (We do not solve this model directly, but rather iteratively identify violated scenarios and add them to the master program.) While this is not guaranteed to be an upper bound because we cannot ensure feasibility with respect to all $u \in \cU$, we use this as a proxy for an upper bound. In addition, we compute an upper bound on \(\tII(s^*,u)\) by finding a local minimum to model~\eqref{prob:infeas}, for a large number of randomly selected elements from the set \(\cU\). 

\subsection{Descriptions of Tests and Results}	
We explore the relationship between \(\alpha^{d,-}\) and \(\beta\) in equations~\eqref{eqn:alphapmdef} and~\eqref{eqn:betadef}. For fixed values of the other $\alpha$-parameters, \(\beta\), and \(\Gamma\), we determine \(\alpha^{d,-}_{\max}\), the largest value of \(\alpha^{d,-}\) for which model~\eqref{prob:rcACOPF} is feasible in order to help design the uncertainty set for our test cases. Figure~\ref{fig:test1c} suggests the resulting \(\alpha^{d,-}_{\max}\) is concave in~\(\beta\).
\begin{figure}[H]
	\centering
	\includegraphics[width=\textwidth]{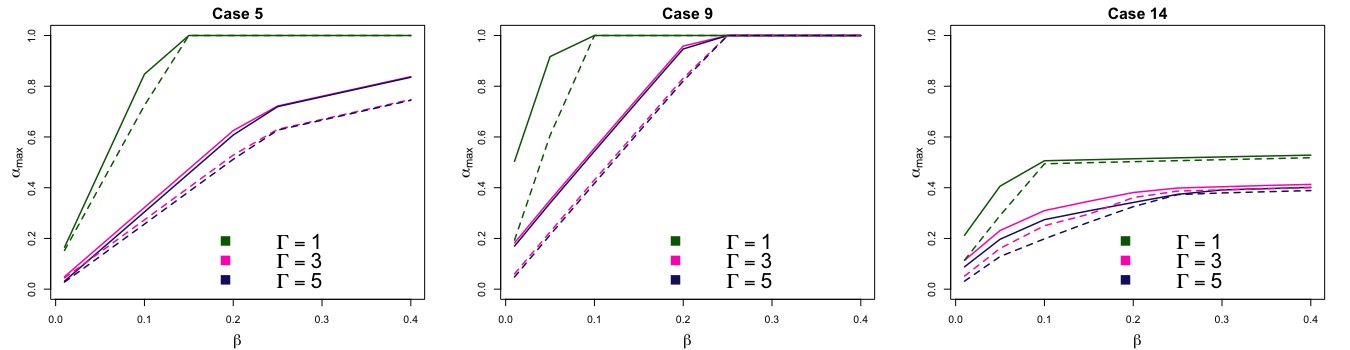}
	\caption{The plots specify \(\alpha^{d,-}_{\max}\) as a function of \(\beta\) for various values of \(\Gamma\); see equations~\eqref{eqn:alphapmdef} and~\eqref{eqn:betadef}. Here, \(\alpha^{d,-}_{\max}\) is the largest value for which model~\eqref{prob:rcACOPF} (solid line) / model~\eqref{prob:rACOPF} (dashed line) is feasible. }
	\label{fig:test1c}
\end{figure}
To understand this concavity, consider model~\eqref{prob:rcACOPF}, where we restrict \(u \in \cU^E\), and replace the objective function by \(0\). Let \(\Lambda\) denote the model's dual feasible region, which is the intersection of polyhedral and second-order cones. Writing this dual compactly we have: 
\begin{equation}
\max_{\lambda \in \Lambda} \quad \alpha^{d,-} (a^\top \lambda) + \beta(b^\top \lambda) + c^\top \lambda,
\end{equation}
where terms \(a\) and \(b\) come from~\eqref{eqn:cbalanceP} and~\eqref{eqn:cbalanceQ}. For the primal to be feasible, we must have
\begin{equation}
\alpha^{d,-} (a^\top \lambda) + \beta(b^\top \lambda) + c^\top \lambda \leq 0 \qquad \forall \lambda \in \Lambda.
\end{equation}
This condition implies: 
\begin{equation} \label{eqn:alphamax}
\alpha^{d,-}_{\max}  =	\inf_{\substack{\lambda \in \Lambda \\ a^\top \lambda > 0}}
\frac{-\beta(b^\top \lambda) - c^\top \lambda}{a^\top \lambda},
\end{equation}
which is consistent with the concave functions for model~\eqref{prob:rcACOPF} in Figure~\ref{fig:test1c} (solid lines).

The dashed lines in Figure~\ref{fig:test1c} repeat the $\alpha^{d,-}_{\max}$--$\beta$ relationship for the nonconvex ACOPF model~\eqref{prob:rACOPF}. Since model~\eqref{prob:rcACOPF} is a relaxation of model~\eqref{prob:rACOPF}, we see the former model can accommodate a slightly larger value of \(\alpha^{d,-}_{\max}\) for a given \(\beta\). While the gaps between solid and dashed lines in Figure~\ref{fig:test1c} give some insight to the difference between models~\eqref{prob:rACOPF} and~\eqref{prob:rcACOPF}, this relationship does not involve the objective function, \(c(s)\). Hence a small or large gap in Figure~\ref{fig:test1c} does not necessarily correspond to a small or large optimality gap. For example, in results that we report below, Case 5 has the largest optimality gap while the difference between the curves for Case 5 in Figure~\ref{fig:test1c} is modest.

All tests use \(\alpha^{h,+}  = \alpha^{h,-} = 1\), and \(\beta = 0.05\); \(\alpha^{d,-} = 5 \alpha^{d,+}\) is set at the case-specific \(\alpha^{d,-}_{\max}\) for \(\beta = 0.05\), and the stopping tolerance is \(\epsilon = 10^{-4} \left(\sum_{i \in \cN} |d_i^p| + |d_i^q| \right)\) for all tests. We first assess three properties of our robust ACOPF problem: the quality of the lower bound generated by model~\eqref{prob:rcACOPF}, the robustness of the solution to model~\eqref{prob:rcACOPF} in the nonconvex setting, and the performance of this robust solution relative to a deterministic alternative. 

For each budget parameter \(\Gamma\), a robust convex \(\epsilon\)-feasible solution, \(\hat{s}\), is first obtained by executing Algorithm~\ref{alg:Cut} with the scenario-appending technique from Section~\ref{subsec:scenAppend}. The cost associated with this solution, \(c(\hat{s})\), provides a lower bound for model~\eqref{prob:rACOPF}. We let $C_R$ denote the lower bound obtained by solving the robust convex relaxation, and we let $C_N$ denote our proxy for the upper bound obtained by finding a local minimum of model~\eqref{prob:rACOPF} with \(\cU\) replaced by $\cU^E$ as described in Section~\ref{subsec:nonconvex}. The gap is defined by \(g = 100 \times \frac{C_N - C_R}{C_N}\).

We measure the infeasibility of solution $\hat{s}$ by solving model~\eqref{prob:infeas} at every $u \in \cU^E$, and report the maximum as \(\tII(\hat{s},\hat{u})\).  We also obtain a deterministic nominal solution, \(\hat{s}^0\), and optimal value, $C_0$, by solving the deterministic QC relaxation; i.e., model~\eqref{prob:rcACOPF} with the singleton \(\cU\) defined under \(\alpha^{d,-} = \alpha^{d,+} = \alpha^{h,-} = \alpha^{h,+} = 0\), and keeping the same recourse adjustment range so that the results are comparable. We again measure the infeasibility of this nominal solution at every $u \in \cU^E$, and report the maximum as \(\tII(\hat{s}^0,\hat{u})\).

\begin{table}[h]
	\centering
	\resizebox{\linewidth}{!}{%
	\begin{tabular}{l | l l l l l l l l l}
		\hline
		Test Case & \(\Gamma\) & \(C_N\) & \(C_R\) & \(g\%\) & \(C_0\) & \(\tII(\hat{s},\hat{u})\) & \(\tII(\hat{s}^0,\hat{u})\) & \(\underline{\tII}(\hat{s}^0,\hat{u})\) \\
		\hline
		\multirow{3}{*}{Case 5} & 1 & 17101.9 & 15005.7 & 12.26 & 12651.2 & \(5.59 \times 10^{-2}\) & \(9.86 \times 10^{-2}\) & \(4.63 \times 10^{-2}\) \\
		& 3 & 19863.1 & 17715.2 & 10.81 & 12651.2 & \(5.71 \times 10^{-2}\) & \(1.43 \times 10^{-1}\) & \(1.19 \times 10^{-1}\) \\
		& 5 & 20181.0 & 17979.6 & 10.91 & 12651.2 & \(5.98 \times 10^{-2}\) & \(1.46 \times 10^{-1}\) & \(1.34 \times 10^{-1}\) \\
		\hline
		\multirow{3}{*}{Case 9} & 1 & 4751.4 & 4751.4 & 0.00 & 4059.1 & \(4.46 \times 10^{-5}\) & \(7.23 \times 10^{-2}\) & \(7.23 \times 10^{-2}\) \\
		& 3 & 5917.4 & 5917.3 & 0.00 & 4059.1 & \(4.09 \times 10^{-6}\) & \(1.81 \times 10^{-1}\) & \(1.81 \times 10^{-1}\)  \\
		& 5 & 7208.9 & 6035.6 & 16.27 & 4059.1 & \(3.92 \times 10^{-3}\) & \(1.91 \times 10^{-1}\) & \(1.91 \times 10^{-1}\) \\
		\hline
		\multirow{3}{*}{Case 14} & 1 & 233.0 & 232.9 & 0.02 & 209.0 & \(0\) & \(6.07 \times 10^{-2}\) & \(6.06 \times 10^{-2}\) \\
		& 3 & 252.9 & 252.9 & 0.02 & 209.0 & \(1.70 \times 10^{-2}\) & \(1.11 \times 10^{-1}\) & \(1.11 \times 10^{-1}\) \\
		& 5 & 260.3 & 260.2 & 0.02 & 209.0 & \(1.41 \times 10^{-2}\) & \(1.34 \times 10^{-1}\) & \(1.31 \times 10^{-1}\)  \\
		\hline
		\multirow{3}{*}{Case 30} & 1 & 187.8 & 186.7 & 0.54 & 164.7 & \(3.28 \times 10^{-3}\) & \(5.15 \times 10^{-2}\) & \(4.57 \times 10^{-2}\) \\
		& 3 & 201.6 & 200.6 & 0.50 & 164.7 & \(1.38 \times 10^{-2} \) & \(8.01 \times 10^{-2}\) & \(7.65 \times 10^{-2}\)  \\
		& 5 & 209.8 & 208.8 & 0.47 & 164.7 & \(1.01 \times 10^{-2}\) & \(9.87 \times 10^{-2}\) & \(9.65 \times 10^{-2}\) \\
		\hline
		\multirow{3}{*}{Case 118} & 1 & 3456.2 & 3426.5 & 0.86 & 3110.6 & \(1.87 \times 10^{-2}\) & \(5.21 \times 10^{-2}\) & \(4.81 \times 10^{-2}\) \\
		& 3 & 3808.3 & 3777.3 & 0.81 & 3110.6 & \(2.19 \times 10^{-2}\) & \(1.04 \times 10^{-1}\) & \(1.02 \times 10^{-1}\) \\
		& 5 & 4045.4 & 4008.1 & 0.92 & 3110.6 & \(1.59 \times 10^{-2}\) & \(1.37 \times 10^{-1}\) & \(1.35 \times 10^{-1}\) \\
		\hline
		\multirow{3}{*}{Case 300} & 1 & 15743.7 & 15116.7 & 3.98 & 13915.0 & \(2.93 \times 10^{-3}\) & \(2.56 \times 10^{-2}\) & \(2.51 \times 10^{-2}\)  \\
		& 3 & 17794.6 & 16832.5 & 5.41 & 13915.0 & \(2.66 \times 10^{-3}\) & \(5.75 \times 10^{-2}\) & \(5.69 \times 10^{-2}\)  \\
		& 5 & 18455.0 & 17522.3 & 5.05 & 13915.0 & \(2.07 \times 10^{-3}\) & \(7.96 \times 10^{-2}\) & \(7.91 \times 10^{-2}\)  \\
		\hline
		\multirow{3}{*}{Case 2383} & 1 & \(1629795.2\) & \(1611397.2\) & \(1.13\) & \(1562639.8\) & \(5.48 \times 10^{-3}\) & \(1.27 \times 10^{-2}\) & \(8.93 \times 10^{-3}\)  \\
		& 3 & \(1714285.1\) & \(1696575.3\) & \(1.03\) & \(1562639.8\) & \(5.79 \times 10^{-3}\) & \(2.71 \times 10^{-2}\) & \(2.59 \times 10^{-2}\)  \\
		& 5 & \(1789041.5\) & \(1772009.8\) & \(0.95\) & \(1562639.8\) & \(5.48 \times 10^{-3}\) & \(4.29 \times 10^{-2}\) & \(4.18 \times 10^{-2}\)  \\
		\hline
		\multirow{3}{*}{Case 2746} & 1 & \(1483630.4\) & \(1480689.5\) & \(0.20\) & \(1440355.8\) & \(9.94 \times 10^{-4}\) & \(1.51 \times 10^{-2}\) & \(1.42 \times 10^{-2}\) \\
		& 3 & \(1564579.5\) & \(1561091.4\) & \(0.22\) & \(1440355.8\) & \(1.04 \times 10^{-3}\) & \(4.20 \times 10^{-2}\) & \(4.12 \times 10^{-2}\)  \\
		& 5 & \(1623337.0\) & \(1619767.2\) & \(0.22\) & \(1440355.8\) & \(1.31 \times 10^{-3}\) & \(6.10 \times 10^{-2}\) & \(6.02 \times 10^{-2}\) \\
		\hline
	\end{tabular}}
	\caption{Robustness results of the robust convex relaxation solution and nominal solution.}
	\label{table:test12Table}
\end{table}

From Table~\ref{table:test12Table} we observe that many of the gaps between the lower bound and estimated upper bound are below 1.5\%, while the gaps for Case 5, Case 9 (\(\Gamma = 5\)), and Case 300 ($\Gamma=3$ and $5$) exceed 5\%. This result suggests that solving model~\eqref{prob:rcACOPF} can provide a tight lower bound for model~\eqref{prob:rACOPF}. The large gap of Case 5 is caused by the convex relaxation, given a specific realization of uncontrollable injections, not being tight. In the deterministic setting, \citet{coffrin2016strengthening} report a gap of about \(9.3\%\) for Case 5. We can also find a general trend that the robust optimality gaps shown in Table~\ref{table:test12Table} are larger than their deterministic counterpart described in~\citet{coffrin2016strengthening}. For a specific uncontrollable injection, the nonconvex feasible region may coincide with that of the QC relaxation near the optimum, but once we take the intersection of feasible regions under the robust setting, this may no longer be true. The degree of this phenomenon depends on the power system structure and level of uncertainty, which may explain the larger gaps in Cases 9 and 300.

For Case 5, the infeasibility measure for model~\eqref{prob:rcACOPF}'s solution is about \(6\%\) of the total demand. It is under \(2.5\%\) of the total demand for other cases, and does not grow with the size of the uncertainty set ($\Gamma$). For example, for \(\Gamma = 1\) of Case 30, the unmet demand is 1.3 MW, about \(0.33\%\) of the total demand. These results contrast with the corresponding infeasibility of the nominal solution in column \(\tII(\hat{s}^0,\hat{u})\), where the magnitude can be significantly larger. For all test cases but Case 5, \(\Gamma = 1\), the lower bound on the infeasibility measure (see \(\underline{\tII}(\hat{s}^0,\hat{u})\) as defined in Section~\ref{subsec:nonconvex}), is close to the upper bound, \(\tII(\hat{s}^0,\hat{u})\), and significantly larger than the upper bound of the infeasibility measure for the robust solution, \(\tII(\hat{s},\hat{u})\). This suggests that the nominal solution is inferior to the robust solution in terms of feasibility under the worst-case scenario.

As we discuss in Section~\ref{subsec:nonconvex}, we are unsure whether $C_N$ is a valid upper bound on the optimal value of model~\eqref{prob:rACOPF} because we relax the model, replacing $\cU$ with $\cU^E$, and then obtain a local minimum, $s^*$. For this reason, we solve model~\eqref{prob:infeas} to obtain $I(s^*,u)$, or rather an upper bound on this value, at interior points of $u \in \cU$ to assess potential infeasibilities. When we do so for the eight test cases, each under three values of $\Gamma$, using 1000 uniform random vectors from the corresponding $\cU$, we obtain \(I(s^*,u) = 0\) in each instance. While we are still unsure that $C_N$ is a valid upper bound, we have been unable to find evidence that $s^*$ is infeasible at interior points of $\cU$, and this helps support using~$C_N$ as a proxy for an upper bound.

To construct set $\cU$, we correlate the uncontrollable injections at different buses as described in Section~\ref{sec:formulation} and Appendix~\ref{appen:facLoc}. Of course, injections may not occur in a worst-case manner or in a manner with this type of correlation.  To assess the performance of our solution in a stochastic environment, we assume that \(u = (u^p,u^q)\) is a uniform random vector in the box specified by the bounds in equation~\eqref{eqn:alphapmdef}, and we sample \(1000\) realizations. Given a solution, \(\hat{s}\), obtained from model~\eqref{prob:rcACOPF}, for each realization $u$, we find a local minimum to the nonconvex model~\eqref{prob:infeas}, and denote the upper bound on the infeasibility measure, \(\tII\). Next, we solve the convex relaxation of model~\eqref{prob:infeas} as we describe in Section~\ref{subsec:nonconvex} to compute the lower bound on the infeasibility measure, \(\underline{\tII}\). 

We show the computational results in Table~\ref{table:test4Table}. For a batch of 1000 realizations, we denote the mean violation in the nonconvex setting by \(\mu_{\tII}\), and the expected maximum violation by \(\tII_{\max}\). Their counterparts under the convex relaxation are denoted~\(\mu_{\underline{\tII}}\) and \(\underline{\tII}_{\max}\). Due to its probabilistic nature, we replicate this test 20 times to obtain a point estimate for these four infeasibility measures as well as \(95\%\) confidence intervals. As expected, our robust solution for \(\Gamma = 5\) is feasible for all $u \in \cU$ for the convex relaxation~\eqref{prob:rcACOPF} by construction, and \(\mu_{\underline{\tII}}\) and \(\underline{\tII}_{\max}\) decrease monotonically as \(\Gamma\) increases. We largely observe a similar trend for the nonconvex infeasibility measures, although there are a few exceptions. 

\begin{table}[h]
	\centering
	\resizebox{\linewidth}{!}{%
	\begin{tabular}{ l | l | l l | l l }
		\hline
		Test Case & \(\Gamma\) 
		& \(\tII_{\max} \pm \) CI halfwidth 
		& \(\underline{\tII}_{\max} \pm \) CI halfwidth 
		& \(\mu_{\tII} \pm \) CI halfwidth 
		& \(\mu_{\underline{\tII}} \pm \) CI halfwidth\\
		\hline
		\multirow{3}{*}{Case 5} 
		& 1 
		& \(9.05 \times 10^{-2} \pm 8.71 \times 10^{-3}\) 
		& \(6.33 \times 10^{-2} \pm 1.32 \times 10^{-2}\)
		& \(2.52 \times 10^{-2} \pm 1.34 \times 10^{-3}\) 
		& \(3.08 \times 10^{-3} \pm 5.94 \times 10^{-4}\) \\
		& 3 
		& \(4.93 \times 10^{-2} \pm 6.89 \times 10^{-3}\) 
		& \(1.63 \times 10^{-4} \pm 1.43 \times 10^{-3}\) 
		& \(7.24 \times 10^{-3} \pm 6.03 \times 10^{-4}\) 
		& \(1.63 \times 10^{-7} \pm 1.43 \times 10^{-6}\) \\
		& 5 
		& \(5.54 \times 10^{-2} \pm 3.14 \times 10^{-3}\) 
		& \(0 \pm 0\) 
		& \(1.71 \times 10^{-2} \pm 6.38 \times 10^{-4}\) 
		& \(0 \pm 0\) \\
		\hline
		\multirow{3}{*}{Case 9} 
		& 1 
		& \(9.24 \times 10^{-2} \pm 1.64 \times 10^{-2}\)  
		& \(9.24 \times 10^{-2} \pm 1.64 \times 10^{-2}\) 
		& \(4.17 \times 10^{-3} \pm 7.58 \times 10^{-4}\) 
		& \(4.17 \times 10^{-3} \pm 7.58 \times 10^{-4}\) \\
		& 3 
		& \(0 \pm 0\) 
		& \(0 \pm 0\) 
		& \(0 \pm 0\) 
		& \(0 \pm 0\) \\
		& 5 
		& \(6.69 \times 10^{-4} \pm 1.53 \times 10^{-3}\) 
		& \(0 \pm 0\) 
		& \(7.93 \times 10^{-7} \pm 2.34 \times 10^{-6}\) 
		& \(0 \pm 0\) \\
		\hline
		\multirow{3}{*}{Case 14} 
		& 1 
		& \(4.50 \times 10^{-2} \pm 1.08 \times 10^{-2}\) 
		& \(4.49 \times 10^{-2} \pm 1.08 \times 10^{-2}\)
		& \(1.24 \times 10^{-3} \pm 2.17 \times 10^{-4}\) 
		& \(1.23 \times 10^{-3} \pm 2.16 \times 10^{-4}\)\\
		& 3 
		& \(1.57 \times 10^{-2} \pm 9.62 \times 10^{-4}\) 
		& \(6.36 \times 10^{-4} \pm 3.18 \times 10^{-3}\) 
		& \(1.71 \times 10^{-3} \pm 2.08 \times 10^{-4}\) 
		& \(7.61 \times 10^{-7} \pm 4.05 \times 10^{-6}\)\\
		& 5 
		& \(1.16 \times 10^{-2} \pm 9.65 \times 10^{-4}\) 
		& \(0 \pm 0\) 
		& \(7.95 \times 10^{-4} \pm 1.28 \times 10^{-4}\) 
		& \(0 \pm 0\) \\
		\hline
		\multirow{3}{*}{Case 30} 
		& 1 
		& \(2.94 \times 10^{-2} \pm 5.68 \times 10^{-3}\) 
		& \(2.84 \times 10^{-2} \pm 6.36 \times 10^{-3}\)
		& \(1.06 \times 10^{-3} \pm 2.57 \times 10^{-4}\) 
		& \(6.81 \times 10^{-4} \pm 2.04 \times 10^{-4}\)\\
		& 3 
		& \(1.26 \times 10^{-2} \pm 6.46 \times 10^{-4}\) 
		& \(6.42 \times 10^{-4} \pm 2.14 \times 10^{-3}\)
		& \(1.65 \times 10^{-3} \pm 2.38 \times 10^{-4}\) 
		& \(9.73 \times 10^{-7} \pm 3.56 \times 10^{-6}\)\\
		& 5 
		& \(8.56 \times 10^{-3} \pm 6.47 \times 10^{-4}\) 
		& \(0 \pm 0\)
		& \(6.24 \times 10^{-4} \pm 1.33 \times 10^{-4}\) 
		& \(0 \pm 0\) \\
		\hline
		\multirow{3}{*}{Case 118} 
		& 1 
		& \(6.56 \times 10^{-2} \pm 1.45 \times 10^{-2}\) 
		& \(5.72 \times 10^{-2} \pm 1.69 \times 10^{-2}\)
		& \(1.14 \times 10^{-2} \pm 7.56 \times 10^{-4}\) 
		& \(3.29 \times 10^{-3} \pm 4.52 \times 10^{-4}\)\\
		& 3 
		& \(2.42 \times 10^{-2} \pm 1.04 \times 10^{-2}\)
		& \(6.88 \times 10^{-3} \pm 1.32 \times 10^{-2}\)
		& \(9.44 \times 10^{-4} \pm 1.82 \times 10^{-4}\)
		& \(7.94 \times 10^{-6} \pm 1.60 \times 10^{-5}\)\\
		& 5 
		& \(3.98 \times 10^{-3} \pm 2.84 \times 10^{-3}\) 
		& \(0 \pm 0\) 
		& \(2.32 \times 10^{-5} \pm 1.79 \times 10^{-5}\) 
		& \(0 \pm 0\) \\
		\hline
		\multirow{3}{*}{Case 300} & 1 
		& \(3.62 \times 10^{-2} \pm 7.35 \times 10^{-3}\)
		& \(3.52 \times 10^{-2} \pm 7.38 \times 10^{-3}\)
		& \(2.57 \times 10^{-3} \pm 2.63 \times 10^{-4}\)
		& \(2.03 \times 10^{-3} \pm 2.45 \times 10^{-4}\)\\
		& 3 
		& \(3.78 \times 10^{-3} \pm 4.04 \times 10^{-3}\) 
		& \(3.15 \times 10^{-3} \pm 4.66 \times 10^{-3}\)
		& \(1.11 \times 10^{-4} \pm 2.08 \times 10^{-5}\) 
		& \(4.03 \times 10^{-5} \pm 1.17 \times 10^{-5}\)\\
		& 5 
		& \(1.38 \times 10^{-3} \pm 4.04 \times 10^{-3}\)
		& \(0 \pm 0\) 
		& \(7.99 \times 10^{-5} \pm 1.51 \times 10^{-5}\) 
		& \(0 \pm 0\) \\
		\hline
		\multirow{3}{*}{Case 2383} 
		& 1 
		& \(2.41 \times 10^{-2} \pm 5.39 \times 10^{-3}\) 
		& \(2.29 \times 10^{-2} \pm 5.44 \times 10^{-3}\)
		& \(3.31 \times 10^{-3} \pm 1.98 \times 10^{-4}\) 
		& \(1.23 \times 10^{-3} \pm 2.04 \times 10^{-4}\)\\
		& 3 
		& \(7.52 \times 10^{-3} \pm 4.03 \times 10^{-3}\) 
		& \(5.71 \times 10^{-3} \pm 5.13 \times 10^{-3}\)
		& \(2.05 \times 10^{-3} \pm 6.60 \times 10^{-5}\) 
		& \(1.38 \times 10^{-5} \pm 1.48 \times 10^{-5}\)\\
		& 5 
		&  \(4.89 \times 10^{-3} \pm 3.41 \times 10^{-4}\) 
		& \(0 \pm 0\) 
		& \(2.06 \times 10^{-3} \pm 5.76 \times 10^{-5}\) 
		& \(0 \pm 0\) \\
		\hline
		\multirow{3}{*}{Case 2746} 
		& 1 
		& \(3.22 \times 10^{-2} \pm 7.91 \times 10^{-3}\) 
		& \(3.12 \times 10^{-2} \pm 7.90 \times 10^{-3}\)
		& \(1.95 \times 10^{-3} \pm 2.35 \times 10^{-4}\) 
		& \(1.71 \times 10^{-3} \pm 2.21 \times 10^{-4}\)\\
		& 3 
		& \(5.75 \times 10^{-3} \pm 5.06 \times 10^{-3}\) 
		& \(4.59 \times 10^{-3} \pm 4.82 \times 10^{-3}\)
		& \(2.16 \times 10^{-5} \pm 1.14 \times 10^{-5}\) 
		& \(7.73 \times 10^{-6} \pm 8.70 \times 10^{-6}\)\\
		& 5 
		& \(3.74 \times 10^{-4} \pm 1.35 \times 10^{-4}\) 
		& \(0 \pm 0\) 
		& \(4.18 \times 10^{-6} \pm 1.43 \times 10^{-6}\) 
		& \(0 \pm 0\) \\
		\hline
	\end{tabular}}
	\caption{Computational results for solving model~\eqref{prob:rcACOPF} for a range of values of \(\Gamma\), and then assessing feasibility, along with 95\% confidence intervals, using 20 replications over \(1000\) uniformly distributed realizations over the hyper-rectangle governing \(u\). }
	\label{table:test4Table}
\end{table}

Next, we report the computational performance of Algorithm~\ref{alg:Cut} and its scenario-appending improvement scheme. If we append scenario-specific constraints to the master problem, the master becomes larger and takes longer to solve, but this helps decrease the number of iterations of the cutting-plane algorithm. In Table~\ref{table:test5Table} we show computational results for Cases~118 and 300, which best exemplify the effectiveness of the scenario-appending scheme. The table shows that direct application of Algorithm~\ref{alg:Cut} fails to obtain an $\epsilon$-feasible solution within 300 iterations, but by appending scenarios to the master, we solve Cases~118 and~300 in at most 12 iterations. The results that we report elsewhere in this section all use the improvement of appending scenarios to the master program in which $n_c$, the number of repetitions after which a scenario $\hat{u} \in \cU$ is appended to the master program, is $n_c=1$;  see Section~\ref{subsec:scenAppend}. 

\begin{table}[H]
	\centering
	\begin{tabular}{ l | l l | l l | l l}
		\hline
		\multirow{2}{*}{Parameters} & \multicolumn{2}{ c |}{No. of iterations} &  \multicolumn{2}{ c |}{\(\epsilon\)-feasibility achieved} & \multicolumn{2}{ c}{\(T\) (sec.)}\\
		& Case 118 & Case 300 & Case 118 & Case 300 & Case 118 & Case 300 \\
		\hline
		Algorithm~\ref{alg:Cut} & 300 & 300 & No & No & 2414 & 23042\\
		\hline
		\(n_c = 1\) & 5 & 6 & Yes & Yes & 53 & 306 \\
		\(n_c = 2\) & 6 & 5 & Yes & Yes & 56 & 226 \\
		\(n_c = 3\) & 7 & 9 & Yes & Yes & 63 & 366 \\
		\(n_c = 4\) & 8 & 11 & Yes & Yes & 70 & 447 \\
		\(n_c = 5\) & 9 & 12 & Yes & Yes & 77 & 463 \\
		\hline
	\end{tabular}
	\caption{Computational results of solving model~\eqref{prob:rcACOPF} with different improvement techniques for Case 118 and Case 300 with \(\Gamma = 3\).}
	\label{table:test5Table}
\end{table}

The scenario-appending method aims to reduce the number of iterations of Algorithm~\ref{alg:Cut}. To decrease the running time of each iteration, we compare the computational performance of two alternatives: solving the MISOCP~\eqref{prob:dualSubIBudget} directly, or enumerating all extreme points of $\cU$ and solving the corresponding SOCPs individually, which is possible when \(|\cM|\) is modest. For example, if \(\Gamma = 1\), there are only \(2|\cM|\) extreme points, and solving this moderate number of SOCPs each iteration may reduce computation time, especially when parallelizing the calculations. Furthermore, rather than identifying $u \in \cU$  for a most violated constraint, we can generate multiple feasibility cuts in one iteration to again attempt to reduce the number of overall iterations. For our tests, we generate cuts at the \(10\) most violated scenarios at each iteration, and we solve SOCPs in parallel with \(20\) threads.

We present the test results in Table~\ref{table:timeperformanceNT}. The number of extreme points of \(\cU\) is denoted by \(N\) and the number of iterations until \(\epsilon\)-feasibility is achieved is denoted ``iter.'' and the clock time of the two approaches is denoted by ``time.''  For all test cases, solving the SOCPs in parallel requires less time than solving MISOCPs. When \(|\cM|\) is small, the number of extreme points, $N$, of \(\cU\) is modest, and so solving SOCPs corresponding to each extreme point in parallel can be more efficient than solving MISOCPs. In our implementation, we note that solving the MISOCP---either directly or by enumerating SOCPs---is the computational bottleneck. Gurobi generates outer linearizations of SOCP subproblems in its branch-and-bound algorithm, which is the default setting. This facilitates warm starts for the resulting linear programs and tends to improve numerical performance over solving subproblems directly as SOCPs. Still, we encounter numerical problems when solving the MISOCP for Case 2746. Of course as the number of extreme points of \(\cU\) becomes large, enumeration of SOCPs is not viable, and solving the corresponding MISOCP, or developing an alternative, in such cases requires further work.

Table~\ref{table:test4Table} compares results for solutions $\hat{s}$ with optimal value $C_R$ from model~\eqref{prob:rcACOPF}; ${s}^*$ with optimal value $C_N$ from model~\eqref{prob:rACOPF} with $\cU$ replaced by $\cU^E$; and, $\hat{s}^0$ with optimal value $C_0$ from the nominal model.  The ``MISOCP'' and ``SOCP'' columns of Table~\ref{table:timeperformanceNT} concern the computational effort to obtain $\hat{s}$, and the two right-most columns concern the effort to compute $s^*$ (``Nonconvex'') and $\hat{s}^0$ (``Nominal''). As we indicate in Section~\ref{subsec:nonconvex} we do not solve model~\eqref{prob:rACOPF} under $\cU^E$ directly because it is too computationally expensive for our largest problems. Rather we iteratively identify violated scenarios and append them to the master program, and doing so involves solving nonconvex ACOPF problems in parallel using Ipopt. The resulting computational effort (``Nonconvex'') is less than that of ``SOCP'', but similar in order of magnitude. The times to solve the nominal problem are significantly less, but the solution it provides can be far from feasible; see Table~\ref{table:test4Table}.

\begin{table}[h]
	\centering
	\resizebox{\linewidth}{!}{%
	\begin{tabular}{ l | l  l | c  r | c r | c r | c }
		\hline
		&  & & \multicolumn{2}{c|}{MISOCP} 	& \multicolumn{2}{c|}{{SOCP}} & \multicolumn{2}{c|}{{Nonconvex}} & {Nominal}\\
		Test Case & \(\Gamma\) & \(N\) & iter.\ & time (sec.) & {iter.}\ & {time (sec.)} & {iter.}\ & {time (sec.)} & {time (sec.)} \\ \hline
		\multirow{3}{*}{Case 5} & 1 & 10 & 2 & {0.7}  & {2} & {0.4} & {2} & {0.2} & {0.05}\\
		& 3 & 80 & 2 & {0.5} & {2} & {0.7} & {2} & {0.4} & {0.07}\\
		& 5 & 32 & 2 & {0.6} & {1} & {0.2} & {2} & {0.3} & {0.05}\\
		\hline
		\multirow{3}{*}{Case 9}  & 1 & 10 & 2 & {1.2} & {2} & {0.4} & {2} & {0.3} & {0.09} \\
		& 3 & 80 & 2 & {0.9} & {2} & {0.8} & {2} & {0.5} & {0.10}\\
		& 5 & 32 & 2 & {1.1} & {1} & {0.3} & {2} & {0.4} & {0.08}\\
		\hline
		\multirow{3}{*}{Case 14}  & 1 & 10 & 2 & {1.9} & {2} & {0.9} & {2} & {0.4} & {0.2}\\
		& 3 & 80 & 2 & {2.2} & {2} & {1.5} & {3} & {1.3} & {0.2}\\
		& 5 & 32 & 2 & {2.0} & {1} & {0.4} & {2} & {0.5} & {0.1}\\
		\hline
		\multirow{3}{*}{Case 30}  & 1 & 10 & 2 & {4.3} & {2} & {1.4} & {2} & {0.6} & {0.3}\\
		& 3 & 80 & 2 & {4.4} & {2} & {2.7} & {2} & {1.6} & {0.4}\\
		& 5 & 32 & 2 & {5.3} & {1} & {0.8} & {2} & {0.9} & {0.3}\\
		\hline
		\multirow{3}{*}{Case 118}  & 1 & 10 & {2} & {20.3} & {3} & {15.6} & {2} & {3.3} & {1.7 }\\
		& 3 & 80 & 5 & {53.0} & {2} & {18.5} & {3} & {13.5} & {1.6}\\
		& 5 & 32 & 2 & {27.8} & {1} & {4.1} & {2} & {6.5} & {1.6}\\
		\hline
		\multirow{3}{*}{Case 300}  & 1 & 10 & {4} & {122.9} & {3} & {44.2} & {3} & {17.2} & {4.3}\\
		& 3 & 80 & 6 & {306.1} & {2} & {61.6} & {2} & {25.7} & {4.7}\\
		& 5 & 32 & {2} & {170.2}& {1} & {14.8} & {1} & {6.0} & {4.6}\\
		\hline
		\multirow{3}{*}{Case 2383} & 1 & 10 & {2} & {1508.5} & {2} & {344.8} & {2} & {87.1} & {76.4 }\\
		& 3 & 80 & {2} & {3008.4} & {2} & {848.6} & {3} & {366.4} & {76.2}\\
		& 5 & 32 & {1} & {2071.1} & {1} & {257.7} & {1} & {56.1} & {79.4 }\\
		\hline
		\multirow{3}{*}{Case 2746} & 1 & 10 & \(-\) & \(-\)& {2} & {557.6} & {2} & {250.8} & {102.5 }\\
		& 3 & 80 & \(-\) & \(-\)& {2} & {1175.3} & {2} & {330.5} & {121.9}\\
		& 5 & 32 & \(-\) & \(-\)& {1} & {250.8} & {2} & {272.3} & {95.6}\\
		\hline
	\end{tabular}}
	\caption{Comparison of solving model~\eqref{prob:rcACOPF} by Algorithm~\ref{alg:Cut} using \eqref{prob:dualSubIBudget} (``MISOCP'' column), or by solving enumerating SOCPs for each \(u \in \cU^E\). The ``Nonconvex'' column shows computation times for solving model~\eqref{prob:rACOPF} with $\cU$ replaced by $\cU^E$, and the right-most column shows that for the single-scenario nominal problem.}
	\label{table:timeperformanceNT}
\end{table}	

\section{Conclusions}
\label{sec:conclusion}
In this paper we present a model to relax the nonconvex robust ACOPF problem to a robust convex program with recourse. A cutting-plane algorithm is proposed to solve the convex relaxation, and within each iteration of the cutting-plane algorithm, an MISOCP is solved to generate a cut separating the incumbent solution from the robust convex feasible region.  In summary, we: 
\begin{itemize}
	\item formulated a two-stage robust model that permits full recourse decisions rather than simpler, e.g., linear, decision rules;
	\item established desirable convergence properties of a cutting-plane algorithm; 
	\item showed that our algorithm can provide a good lower bound for the nonconvex ACOPF problem~\eqref{prob:rACOPF};
	\item found the solution to the robust convex relaxation model~\eqref{prob:rcACOPF} is robust in the nonconvex setting, provided its deterministic QC relaxation is reasonably tight; and,
	\item reduced solution time in the cutting-plane algorithm by appending a small number of key scenarios to the master program. 
\end{itemize}
There are many possible ways to extend the result of this research. One important direction is to reduce the computational effort required to solve the ``separation problem,'' which is currently modeled as an MISOCP. Doing so would further facilitate scaling our algorithm to larger problems. {Finding a valid upper bound for our nonconvex robust ACOPF problem, or further characterizing our proxy for an upper bound, would be valuable by itself, but also has the potential to be combined with the optimization-based bound-tightening process~\citep[e.g.,][]{sundar2018OBBT} to further tighten the formulation and the lower bound that we obtain. Furthermore, our modeling framework, and associated solution algorithm, has the potential to be used in, or adapted to, applications of robust optimization with convex recourse in areas such as microgrid planning~\citep[e.g.,][]{khodaei2014resiliency}, location transportation~\citep[e.g.,][]{gabrel2014robust}, and call center staffing~\citep[e.g.,][]{gurvich_etal_2010,zan_etal_2014}.}
\vspace{0.2cm}

\bibliographystyle{plainnat} 
\bibliography{RO-ACOPF_ArXiv}

\begin{thebibliography}{54}
\providecommand{\natexlab}[1]{#1}
\providecommand{\url}[1]{\texttt{#1}}
\expandafter\ifx\csname urlstyle\endcsname\relax
  \providecommand{\doi}[1]{doi: #1}\else
  \providecommand{\doi}{doi: \begingroup \urlstyle{rm}\Url}\fi

\bibitem[Attarha et~al.(2018)Attarha, Amjady, and Conejo]{conejo2018adaptive}
A.~Attarha, N.~Amjady, and A.~J. Conejo.
\newblock Adaptive robust {AC} optimal power flow considering load and wind
  power uncertainties.
\newblock \emph{International Journal of Electrical Power \& Energy Systems},
  96:\penalty0 132--142, 2018.

\bibitem[Bai and Wei(2011)]{bai2011semidefinite}
X.~Bai and H.~Wei.
\newblock A semidefinite programming method with graph partitioning technique
  for optimal power flow problems.
\newblock \emph{International Journal of Electrical Power \& Energy Systems},
  33\penalty0 (7):\penalty0 1309--1314, 2011.

\bibitem[Bai et~al.(2008)Bai, Wei, Fujisawa, and Wang]{bai2008semidefinite}
X.~Bai, H.~Wei, K.~Fujisawa, and Y.~Wang.
\newblock Semidefinite programming for optimal power flow problems.
\newblock \emph{International Journal of Electrical Power \& Energy Systems},
  30\penalty0 (6):\penalty0 383--392, 2008.

\bibitem[Bernstein et~al.(2014)Bernstein, Bienstock, Hay, Uzunoglu, and
  Zussman]{bernstein2014power}
A.~Bernstein, D.~Bienstock, D.~Hay, M.~Uzunoglu, and G.~Zussman.
\newblock Power grid vulnerability to geographically correlated
  failures---analysis and control implications.
\newblock In \emph{IEEE INFOCOM 2014-IEEE Conference on Computer
  Communications}, pages 2634--2642. IEEE, 2014.

\bibitem[Bertsekas(2009)]{bertsekas_2009}
D.~P. Bertsekas.
\newblock \emph{Convex Optimization Theory}.
\newblock Athena Scientific, 2009.

\bibitem[Bertsimas and Sim(2003)]{bertsimas2003robust}
D.~Bertsimas and M.~Sim.
\newblock Robust discrete optimization and network flows.
\newblock \emph{Mathematical Programming}, 98\penalty0 (1):\penalty0 49--71,
  2003.

\bibitem[Bertsimas and Sim(2004)]{bertsimas2004price}
D.~Bertsimas and M.~Sim.
\newblock The price of robustness.
\newblock \emph{Operations Research}, 52\penalty0 (1):\penalty0 35--53, 2004.

\bibitem[Bienstock(2015)]{bienstock2015electrical}
D.~Bienstock.
\newblock \emph{Electrical Transmission System Cascades and Vulnerability: An
  Operations Research Viewpoint}.
\newblock SIAM, 2015.

\bibitem[Bienstock et~al.(2014)Bienstock, Chertkov, and
  Harnett]{bienstock2014chance}
D.~Bienstock, M.~Chertkov, and S.~Harnett.
\newblock Chance-constrained optimal power flow: Risk-aware network control
  under uncertainty.
\newblock \emph{SIAM Review}, 56\penalty0 (3):\penalty0 461--495, 2014.

\bibitem[Boyd and Vandenberghe(2004)]{boyd2004convex}
S.~Boyd and L.~Vandenberghe.
\newblock \emph{Convex Optimization}.
\newblock Cambridge University Press, 2004.

\bibitem[Carpentier(1962)]{carpentier1962}
J.~Carpentier.
\newblock Contribution \'{a} l'\'{e}tude du dispatching \'{e}conomique.
\newblock \emph{Bulletin de la Soci\'{e}t\'{e} Fran\c{c}aise des Electriciens},
  3:\penalty0 431--447, 1962.

\bibitem[Coffrin et~al.(2014)Coffrin, Gordon, and Scott]{coffrin2014nesta}
C.~Coffrin, D.~Gordon, and P.~Scott.
\newblock {NESTA}, the {NICTA} energy system test case archive.
\newblock \emph{arXiv preprint arXiv:1411.0359}, 2014.
\newblock URL \url{https://arxiv.org/pdf/1411.0359.pdf}.

\bibitem[Coffrin et~al.(2015{\natexlab{a}})Coffrin, Hijazi, and
  Van~Hentenryck]{coffrin2015strengthening}
C.~Coffrin, H.~L. Hijazi, and P.~Van~Hentenryck.
\newblock Strengthening convex relaxations with bound tightening for power
  network optimization.
\newblock In \emph{International Conference on Principles and Practice of
  Constraint Programming}, pages 39--57. Springer, 2015{\natexlab{a}}.

\bibitem[Coffrin et~al.(2015{\natexlab{b}})Coffrin, Hijazi, and
  Van~Hentenryck]{coffrin2016strengthening}
C.~Coffrin, H.~L. Hijazi, and P.~Van~Hentenryck.
\newblock Strengthening the {SDP} relaxation of {AC} power flows with convex
  envelopes, bound tightening, and lifted nonlinear cuts.
\newblock \emph{arXiv preprint arXiv:1512.04644}, 2015{\natexlab{b}}.

\bibitem[Coffrin et~al.(2016)Coffrin, Hijazi, and
  Van~Hentenryck]{coffrin2016qc}
C.~Coffrin, H.~L. Hijazi, and P.~Van~Hentenryck.
\newblock The {QC} relaxation: A theoretical and computational study on optimal
  power flow.
\newblock \emph{IEEE Transactions on Power Systems}, 31\penalty0 (4):\penalty0
  3008--3018, 2016.

\bibitem[Crainic et~al.(2016)Crainic, Hewitt, Maggioni, and
  Rei]{crainic2016partial}
T.~G. Crainic, M.~Hewitt, F.~Maggioni, and W.~Rei.
\newblock Partial benders decomposition strategies for two-stage stochastic
  integer programs.
\newblock Technical report, CIRRELT, 2016.

\bibitem[Dunning et~al.(2017)Dunning, Huchette, and
  Lubin]{DunningHuchetteLubin2017}
I.~Dunning, J.~Huchette, and M.~Lubin.
\newblock {JuMP}: A modeling language for mathematical optimization.
\newblock \emph{SIAM Review}, 59\penalty0 (2):\penalty0 295--320, 2017.
\newblock \doi{10.1137/15M1020575}.

\bibitem[Fang et~al.(2018)Fang, Hodge, Du, Zhang, and Li]{fang2018modelling}
X.~Fang, B.~Hodge, E.~Du, N.~Zhang, and F.~Li.
\newblock Modelling wind power spatial-temporal correlation in multi-interval
  optimal power flow: A sparse correlation matrix approach.
\newblock \emph{Applied Energy}, 230:\penalty0 531--539, 2018.

\bibitem[Gabrel et~al.(2014)Gabrel, Lacroix, Murat, and
  Remli]{gabrel2014robust}
V.~Gabrel, M.~Lacroix, C.~Murat, and N.~Remli.
\newblock Robust location transportation problems under uncertain demands.
\newblock \emph{Discrete Applied Mathematics}, 164:\penalty0 100--111, 2014.

\bibitem[Geoffrion(1972)]{geoffrion1972generalized}
A.~M. Geoffrion.
\newblock Generalized {Benders} decomposition.
\newblock \emph{Journal of Optimization Theory and Applications}, 10\penalty0
  (4):\penalty0 237--260, 1972.

\bibitem[{Gurobi Optimization, Inc.}(2016)]{gurobi2016}
{Gurobi Optimization, Inc.}
\newblock \emph{Gurobi Optimizer Reference Manual}, 2016.
\newblock URL \url{http://www.gurobi.com}.

\bibitem[Gurvich et~al.(2010)Gurvich, Luedtke, and Tezcan]{gurvich_etal_2010}
I.~Gurvich, J.~Luedtke, and T.~Tezcan.
\newblock Staffing call centers with uncertain demand forecasts: A
  chance-constrained optimization approach.
\newblock \emph{Management Science}, 56\penalty0 (7):\penalty0 1093--1115,
  2010.

\bibitem[Jabr(2006)]{jabr2006radial}
R.~A. Jabr.
\newblock Radial distribution load flow using conic programming.
\newblock \emph{IEEE Transactions on power systems}, 21\penalty0 (3):\penalty0
  1458--1459, 2006.

\bibitem[Jabr(2013)]{jabr2013adjustable}
R.~A. Jabr.
\newblock Adjustable robust {OPF} with renewable energy sources.
\newblock \emph{IEEE Transactions on Power Systems}, 28\penalty0 (4):\penalty0
  4742--4751, 2013.

\bibitem[Jiang et~al.(2012)Jiang, Wang, and Guan]{jiang2012robust}
R.~Jiang, J.~Wang, and Y.~Guan.
\newblock Robust unit commitment with wind power and pumped storage hydro.
\newblock \emph{IEEE Transactions on Power Systems}, 27\penalty0 (2):\penalty0
  800--810, 2012.

\bibitem[Jiang et~al.(2014)Jiang, Zhang, Li, and Guan]{jiang2014two}
R.~Jiang, M.~Zhang, G.~Li, and Y.~Guan.
\newblock Two-stage network constrained robust unit commitment problem.
\newblock \emph{European Journal of Operational Research}, 234\penalty0
  (3):\penalty0 751--762, 2014.

\bibitem[{Khodaei}(2014)]{khodaei2014resiliency}
A.~{Khodaei}.
\newblock Resiliency-oriented microgrid optimal scheduling.
\newblock \emph{{IEEE} Transactions on Smart Grid}, 5\penalty0 (4):\penalty0
  1584--1591, 2014.

\bibitem[Klima and Apt(2015)]{klima2015geographic}
K.~Klima and J.~Apt.
\newblock Geographic smoothing of solar {PV}: results from {Gujarat}.
\newblock \emph{Environmental Research Letters}, 10\penalty0 (10):\penalty0
  104001, 2015.

\bibitem[Kocuk et~al.(2016)Kocuk, Dey, and Sun]{kocuk2016strong}
B.~Kocuk, S.~S. Dey, and X.~A. Sun.
\newblock Strong {SOCP} relaxations for the optimal power flow problem.
\newblock \emph{Operations Research}, 64\penalty0 (6):\penalty0 1177--1196,
  2016.

\bibitem[Lavaei and Low(2012)]{lavaei2012zero}
J.~Lavaei and S.~H. Low.
\newblock Zero duality gap in optimal power flow problem.
\newblock \emph{IEEE Transactions on Power Systems}, 27\penalty0 (1):\penalty0
  92--107, 2012.

\bibitem[Liu and Ferris(2015)]{ferris2015security}
Y.~Liu and M.~C. Ferris.
\newblock Security constrained economic dispatch using semidefinite
  programming.
\newblock In \emph{2015 IEEE Power and Energy Society General Meeting}.
  Institute of Electrical and Electronics Engineers (IEEE), 2015.
\newblock \doi{10.1109/pesgm.2015.7286268}.
\newblock URL \url{http://dx.doi.org/10.1109/PESGM.2015.7286268}.

\bibitem[Lohmann et~al.(2016)Lohmann, Monahan, and Heinemann]{lohmann2016local}
G.~M. Lohmann, A.~H. Monahan, and D.~Heinemann.
\newblock Local short-term variability in solar irradiance.
\newblock \emph{Atmospheric Chemistry and Physics}, 16\penalty0 (10):\penalty0
  6365--6379, 2016.

\bibitem[Lorca and Sun(2018)]{lorca2018adaptive}
\'{A}. Lorca and X.~A. Sun.
\newblock The adaptive robust multi-period alternating current optimal power
  flow problem.
\newblock \emph{IEEE Transactions on Power Systems}, 33\penalty0 (2):\penalty0
  1993--2003, 2018.

\bibitem[Louca and Bitar(2017)]{louca2017robust}
R.~Louca and E.~Bitar.
\newblock Robust {AC} optimal power flow.
\newblock \emph{arXiv preprint arXiv:1706.090199}, 2017.

\bibitem[Low(2014{\natexlab{a}})]{low2014convex}
S.~H. Low.
\newblock Convex relaxation of optimal power flow - part {I}: Formulations and
  equivalence.
\newblock \emph{IEEE Transactions on Control of Network Systems}, 1\penalty0
  (1):\penalty0 15--27, 2014{\natexlab{a}}.

\bibitem[Low(2014{\natexlab{b}})]{low2014convexII}
S.~H. Low.
\newblock Convex relaxation of optimal power flow - part {II}: Exactness.
\newblock \emph{IEEE Transactions on Control of Network Systems}, 1\penalty0
  (2):\penalty0 177--189, 2014{\natexlab{b}}.

\bibitem[Lubin et~al.(2016)Lubin, Dvorkin, and Backhaus]{lubin2016robust}
M.~Lubin, Y.~Dvorkin, and S.~Backhaus.
\newblock A robust approach to chance constrained optimal power flow with
  renewable generation.
\newblock \emph{IEEE Transactions on Power Systems}, 31\penalty0 (5):\penalty0
  3840--3849, 2016.

\bibitem[Maisano et~al.(2016)Maisano, Radchik, and Ling]{maisano2016lognormal}
J.~Maisano, A.~Radchik, and T.~Ling.
\newblock A lognormal model for demand forecasting in the national electricity
  market.
\newblock \emph{The ANZIAM Journal}, 57\penalty0 (3):\penalty0 369--383, 2016.

\bibitem[Malvaldi et~al.(2017)Malvaldi, Weiss, Infield, Browell, Leahy, and
  Foley]{malvaldi2017spatial}
A.~Malvaldi, S.~Weiss, D.~Infield, J.~Browell, P.~Leahy, and A.~M. Foley.
\newblock A spatial and temporal correlation analysis of aggregate wind power
  in an ideally interconnected {Europe}.
\newblock \emph{Wind Energy}, 20\penalty0 (8):\penalty0 1315--1329, 2017.

\bibitem[Momoh et~al.(1999)Momoh, El-Hawary, and Adapa]{momoh1999review}
J.~A. Momoh, M.~E. El-Hawary, and R.~Adapa.
\newblock A review of selected optimal power flow literature to 1993. {Part}
  {II}. {Newton}, linear programming and interior point methods.
\newblock \emph{IEEE Transactions on Power Systems}, 14\penalty0 (1):\penalty0
  105--111, 1999.

\bibitem[Monticelli et~al.(1987)Monticelli, Pereira, and
  Granville]{monticelli1987security}
A.~Monticelli, M.~V.~F. Pereira, and S.~Granville.
\newblock Security-constrained optimal power flow with post-contingency
  corrective rescheduling.
\newblock \emph{IEEE Transactions on Power Systems}, 2\penalty0 (1):\penalty0
  175--180, 1987.

\bibitem[Nemirovsky and Yudin(1983)]{nemirovskii1983problem}
A.~S. Nemirovsky and D.~B. Yudin.
\newblock \emph{Problem Complexity and Method Efficiency in Optimization}.
\newblock A Wiley-Interscience publication. Wiley, 1983.

\bibitem[{Nguyen} et~al.(2019){Nguyen}, {Dvijotham}, and
  {Turitsyn}]{nguyen2019inner}
H.~D. {Nguyen}, K.~{Dvijotham}, and K.~{Turitsyn}.
\newblock Constructing convex inner approximations of steady-state security
  regions.
\newblock \emph{IEEE Transactions on Power Systems}, 34\penalty0 (1):\penalty0
  257--267, 2019.

\bibitem[Phan and Ghosh(2014)]{phan2014two}
D.~Phan and S.~Ghosh.
\newblock Two-stage stochastic optimization for optimal power flow under
  renewable generation uncertainty.
\newblock \emph{ACM Transactions on Modeling and Computer Simulation (TOMACS)},
  24\penalty0 (1):\penalty0 1--22, 2014.

\bibitem[Singh et~al.(2010)Singh, Pal, and Jabr]{singh2010statistical}
R.~Singh, B.~C. Pal, and R.~A. Jabr.
\newblock Statistical representation of distribution system loads using
  {Gaussian} mixture model.
\newblock \emph{IEEE Transactions on Power Systems}, 25\penalty0 (1):\penalty0
  29--37, 2010.

\bibitem[Stott et~al.(2009)Stott, Jardim, and Alsa{\c{c}}]{stott2009dc}
B.~Stott, J.~Jardim, and O.~Alsa{\c{c}}.
\newblock {DC} power flow revisited.
\newblock \emph{IEEE Transactions on Power Systems}, 24\penalty0 (3):\penalty0
  1290--1300, 2009.

\bibitem[Sundar et~al.(2018)Sundar, Nagarajan, Misra, Lu, Coffrin, and
  Bent]{sundar2018OBBT}
K.~Sundar, H.~Nagarajan, S.~Misra, M.~Lu, C.~Coffrin, and R~Bent.
\newblock Optimization-based bound tightening using a strengthened
  {QC}-relaxation of the optimal power flow problem.
\newblock \emph{arXiv preprint arXiv:1809.04565}, 2018.

\bibitem[Terry(2009)]{terry2009thesis}
T.~L. Terry.
\newblock \emph{Robust linear optimization with recourse: Solution methods and
  other properties}.
\newblock PhD thesis, University of Michigan, 2009.

\bibitem[Thiele et~al.(2009)Thiele, Terry, and Epelman]{thiele2009robust}
A.~Thiele, T.~Terry, and M.~Epelman.
\newblock Robust linear optimization with recourse.
\newblock \emph{Optimization-Online}, 2009.
\newblock URL
  \url{http://www.optimization-online.org/DB_FILE/2009/03/2263.pdf}.

\bibitem[Verma(2010)]{verma2010power}
A.~Verma.
\newblock \emph{Power grid security analysis: An optimization approach}.
\newblock PhD thesis, Columbia University, 2010.

\bibitem[W{\"a}chter and Biegler(2006)]{wachter2006implementation}
A.~W{\"a}chter and L.~T. Biegler.
\newblock On the implementation of an interior-point filter line-search
  algorithm for large-scale nonlinear programming.
\newblock \emph{Mathematical programming}, 106\penalty0 (1):\penalty0 25--57,
  2006.

\bibitem[Xie and Ahmed(2018)]{xie2018distributionally}
W.~Xie and S.~Ahmed.
\newblock Distributionally robust chance constrained optimal power flow with
  renewables: A conic reformulation.
\newblock \emph{IEEE Transactions on Power Systems}, 33\penalty0 (2):\penalty0
  1860--1867, 2018.

\bibitem[Zan et~al.(2014)Zan, Hasenbein, and Morton]{zan_etal_2014}
J.~Zan, J.~J. Hasenbein, and D.~P. Morton.
\newblock Asymptotically optimal staffing of service systems with joint {QoS}
  constraints.
\newblock \emph{Queueing Systems}, 78\penalty0 (4):\penalty0 359--386, 2014.

\bibitem[Zugno and Conejo(2015)]{zugno2015robust}
M.~Zugno and A.~J. Conejo.
\newblock A robust optimization approach to energy and reserve dispatch in
  electricity markets.
\newblock \emph{European Journal of Operational Research}, 247\penalty0
  (2):\penalty0 659--671, 2015.

\end{thebibliography}

\begin{appendix}	
	\section{Grouping Buses} \label{appen:facLoc}
	To construct the uncertainty set $\cU$, we solve a facility location problem for each test case to cluster the buses. That is, while we solve the test instances with the full resolution of network topology, as indicated in equation~\eqref{eqn:groupBounds}, uncertain injections at buses occur in concert within a cluster. We assume the distance between two directly connected buses is 1, and more generally, the distance between two buses is the length of the shortest path (counted in hops) between them. We select a total of $|\cM|=5$ buses to be the ``facilities" and assign each bus to a facility. All buses that are assigned to a facility are considered a cluster, i.e., elements of $\cN_m$. The detailed formulation is expressed as follows:
	\begin{table}[H]
		\small
		\begin{tabular}{ l l l l }
			\multicolumn{4}{l}{{\bf Indices and index sets}} \\
			\(i \in \cN\) &\(\qquad\) & set of buses;&\\
			\(J_i \subseteq \cN\) & \(\qquad\) & set of buses eligible to be associated with bus \(i\), \(i \in \cN\);
			\\
			\multicolumn{4}{l}{{\bf Parameters}} \\
			\(d_{ij}\)&\(\qquad\)& distance between bus \(i\) and \(j\), \(i,j \in \cN\);&\\
			\(N\)&\(\qquad\)& number of facilities ($|\cM|$);&\\
			\multicolumn{4}{l}{\bf Decision variables}\\
			\(x_{ij}\)& \(\qquad\) & indicator of whether bus \(i\) is assigned to bus \(j\), \(i,j \in \cN\);&\\
			\(y_{i}\)& \(\qquad\) & indicator of whether bus \(i\) is selected as a facility, \(i \in \cN\).&
		\end{tabular}
	\end{table}	
	{\bf Facility Location Problem:}
	\begin{subequations}\label{model:facility_location}
		\begin{align}
		\min \quad & \sum_{i \in \cN} \sum_{j \in J_i} d_{ij} x_{ij} \\
		\text{s.t.} \quad & \sum_{j \in J_i} x_{ij} = 1 \quad \qquad \forall i \in \cN\\
		& \sum_{i \in \cN} y_{i} = N \\
		& x_{ij} \in \{0,1\} \; \; \quad \qquad \forall i \in \cN, j \in J_i \\
		& y_i \in \{0,1\} \; \;\; \quad \qquad \forall i \in \cN. 
		\end{align}
	\end{subequations}
	To facilitate computational tractability, we control the size of \(J_i\) via a distance threshold so that we can only assign one bus to another if their distance is within the threshold. The clusters formed by model~\eqref{model:facility_location} define the sets $\cM$ and $\cN_m$ used in equation~\eqref{eqn:groupBounds} to construct $\cU$ as described at the beginning of Section~\ref{sec:formulation}.
	
	In addition to clustering buses, we distinguish uncertainty in load and in renewable generation as described in Section~\ref{subsec:alphSel}. The latter occurs only at a subset of buses, denote $\cN_G$, and we now describe construction of this set. Once clusters are formed, we select the two buses in each cluster that have the largest capacity, defined by the sum of the capacities of the incident lines. If the cluster is a singleton, then only that bus is selected, and the process for other clusters remains the same. The buses selected in this way form set~$\cN_G$.
	
	\vspace{1cm}
	
	\section{QC Relaxation} \label{appen:QCRelaxation}
	We use the quadratic convex (QC) relaxation from \citet{coffrin2016qc}. Nonconvex functions in the ACOPF problem, such as quadratic, cosine and sine functions are transformed into a collection of second-order cone constraints and linear constraints. A McCormick relaxation is applied to linearize multi-linear terms. We also assume the  difference in phase angles at adjacent buses \(i\) and \(j\) satisfies \(-\frac{\pi}{6} \leq \underline{\Delta}_{k} \leq \theta_i - \sigma_k - \theta_j \leq \overline{\Delta}_{k} \leq \frac{\pi}{6}\). The detailed formulation of the QC relaxation is expressed as follows:
	\begin{longtable}{ l l l l }
		\multicolumn{4}{l}{\bf Indices and index sets} \\
		\(i \in \cN\) & \(\qquad\) & set of buses;&\\
		\(k = (i,j,n) \in \cA\) & \(\qquad\) & set of lines; &\\
		\(g \in \cG\) & \(\qquad\) & set of generators;&\\
		\(g \in \cG_i\) &\(\qquad\) & set of generators that are connected to bus \(i \in \cN\);&\\
		\\
		\multicolumn{4}{l}{\bf Parameters} \\
		\(u_i^p\)& \(\qquad\) & uncontrollable active power injection at bus \(i \in \cN\);&\\
		\(u_i^q\)& \(\qquad\) & uncontrollable reactive power injection at bus \(i \in \cN\);&\\
		\(\underline{s}_g^p, \overline{s}_g^p\)& \(\qquad\) & lower and upper bound of active power generation by generator \(g\), \(g \in \cG\);&\\
		\(\underline{s}_g^q, \overline{s}_g^q\)& \(\qquad\) & lower and upper bound of reactive power generation by generator \(g\), \(g \in \cG\);&\\
		\(\underline{v}_i, \overline{v}_i\)& \(\qquad\) & lower and upper bound of voltage magnitude at bus \(i\), \(i \in \cN\);&\\
		\(\underline{\theta}_{i}, \overline{\theta}_{i}\)& \(\qquad\) & lower and upper bound of phase angle at bus \(i \in \cN\);&\\
		\(\underline{\Delta}_{k}, \overline{\Delta}_{k}\)& \(\qquad\) & lower and upper bound of phase angle difference of adjacent buses&\\
		& \(\qquad\) & on line \(k \in \cA\);& \\
		\(\underline{cs}_{k}, \overline{cs}_{k}\)& \(\qquad\) & lower and upper bound of cosine of phase angle difference of  \\
		&& adjacent buses on line \(k \in \cA\);&\\
		\(\underline{ss}_{k}, \overline{ss}_{k}\)& \(\qquad\) & lower and upper bound of sine of phase angle difference of \\
		&& adjacent buses on line \(k \in \cA\);&\\
		\(g_k\)& \(\qquad\) & conductance of line \(k\), \(k \in \cA\);&\\
		\(b_k\)& \(\qquad\) & susceptance of line \(k\), \(k \in \cA\);&\\
		\(b_k^c\) & \(\qquad\) & charging susceptance of line \(k\), \(k \in \cA\); & \\
		\(g_i^{sh}\) & \(\qquad\) & shunt conductance of bus \(i\), \(i \in \cN\); & \\
		\(b_i^{sh}\) & \(\qquad\) & shunt susceptance of bus \(i\), \(i \in \cN\); & \\
		\(W_k\) & \(\qquad\) & maximum apparent power flow on line \(k\), \(k \in \cA\);&\\
		\(\tau_{1,k}\) & \(\qquad\) & tap ratio of transformer at bus \(i\) on line \(k\), \(k \in \cA\); & \\
		\(\tau_{2,k}\) & \(\qquad\) & tap ratio of transformer at bus \(j\) on line \(k\), \(k \in \cA\); & \\
		\(\sigma_k\) & \(\qquad\) & phase angle shift of transformer on line \(k\), \(k \in \cA\); &\\
		\(\theta_{k}^u\)& \(\qquad\) & upper bound of absolute value of phase angle difference \\
		&&of line \(k\), \(k \in \cA\), \(\theta^u_{k} = \max(|\overline{\Delta}_{k}|,|\underline{\Delta}_{k}|)\);&\\
		\(\underline{v}_i^\delta\)& \(\qquad\) & sum of lower and upper bound of voltage magnitude at bus \(i\), \(i \in \cN\), &\\
		& & \(v^\delta_i = \overline{v}_i + \underline{v}_i\);&\\
		\(\theta_{k}^\phi\)& \(\qquad\) & mid-point of range of difference in phase angles on line \(k\), \(k \in \cA\), \\
		&&\(\theta^\phi_{k} = \frac{(\overline{\Delta}_{k} + \underline{\Delta}_{k})}{2}\);&\\
		\(\theta_{k}^\delta\)& \(\qquad\) & range of difference in phase angles on line \(k\), \(k \in \cA\), \(\theta^\delta_{k} = \frac{(\overline{\Delta}_{k} - \underline{\Delta}_{k})}{2}\);&\\
		\\
		\multicolumn{4}{l}{\bf Decision variables}\\
		\(s_{g}^p\) & \(\qquad\) & active power generation at generator \(g\), \(g \in \cG\);&\\
		\(s_{g}^q\) & \(\qquad\) & reactive power generation at generator \(g\), \(g \in \cG\);&\\
		\(v_{i}\) & \(\qquad\) & voltage magnitude at bus \(i\), \(i \in \cN\);&\\
		\(\theta_{i}\) & \(\qquad\) & phase angle at bus \(i\), \(i \in \cN\);&\\
		\(P_{k}\) & \(\qquad\) & active power flow on line \(k\), \(k \in \cA\);&\\
		\(Q_{k}\) & \(\qquad\) & reactive power flow on line \(k\), \(k \in \cA\);&\\
		\(\widehat{cs}_{k}\)& \(\qquad\) & approximation term of \(\cos(\theta_i - \sigma_k - \theta_j)\), \( k = (i,j,n) \in \cA\);&\\
		\(\widehat{ss}_{k}\)& \(\qquad\) & approximation term of \(\sin(\theta_i - \sigma_k - \theta_j)\), \( k = (i,j,n) \in \cA\);&\\
		\(\hat{v}_{i}\)& \(\qquad\) & approximation term of \(v_i^2\), \(i \in \cN\);&\\
		\(\widehat{vv}_{k}\)& \(\qquad\) & approximation term of \(\frac{v_i v_j}{\tau_{1,k} \tau_{2,k}}\), \( k = (i,j,n) \in \cA\);&\\
		\(\widehat{wc}_{k}\)& \(\qquad\) & approximation term of \(\frac{v_i v_j}{\tau_{1,k} \tau_{2,k}}\cos(\theta_i - \sigma_k - \theta_j)\), \( k = (i,j,n) \in \cA\);&\\
		\(\widehat{ws}_{k}\)& \(\qquad\) & approximation term of \(\frac{v_i v_j}{\tau_{1,k} \tau_{2,k}}\sin(\theta_i - \sigma_k - \theta_j)\), \( k = (i,j,n) \in \cA\).&\\
	\end{longtable}
	
	\noindent {\bf Formulation:}
	\begin{subequations}
		\begin{align}
		\min \quad & c(s^p,s^q)&\\
		\text{s.t.} \quad & \underline{v}_i \leq v_i \leq \overline{v}_i & \forall i \in \cN \label{app:vSimple}\\
		& \underline{\Delta}_{k} \leq \theta_i - \sigma_k - \theta_j \leq \overline{\Delta}_{k} & \forall k = (i,j,n) \in \cA \label{app:thetaSimple}\\
		& \underline{\theta}_{i} \leq \theta_{i} \leq \overline{\theta}_i & \forall i \in \cN \label{app:thetaB}\\
		&\underline{cs}_{k} \leq \widehat{cs}_{k} \leq \overline{cs}_{k} & \forall k = (i,j,n) \in \cA \label{app:csSimple}\\
		&\underline{ss}_{k} \leq \widehat{ss}_{k} \leq \overline{ss}_{k} & \forall k = (i,j,n) \in \cA \label{app:ssSimple}\\
		& \underline{s}_g^p \leq s_g^p \leq \overline{s}_g^p & \forall g \in \cG \label{app:spConstr}\\
		& \underline{s}_g^q \leq s_g^q \leq \overline{s}_g^q & \forall g \in \cG \label{app:sqConstr}\\
		& P_{k} = g_{k}\frac{\hat{v}_i}{(\tau_{1,k})^2} - g_{k}\frac{\widehat{wc}_{k}}{\tau_{1,k} \tau_{2,k}} - b_{k}\frac{\widehat{ws}_{k}}{\tau_{1,k} \tau_{2,k}} & \forall k = (i,j,n) \in \cA \label{app:Pk}\\
		& Q_{k} = - (b_{k} + \frac{b_k^c}{2})\frac{\hat{v}_i}{(\tau_{1,k})^2} + b_{k}\frac{\widehat{wc}_{k}}{\tau_{1,k} \tau_{2,k}} - g_{k}\frac{\widehat{ws}_{k}}{\tau_{1,k} \tau_{2,k}} & \forall k = (i,j,n) \in \cA \label{app:Qk}\\
		& P_k^2 + Q_k^2 \leq W_k^2 & \forall k \in \cA \label{app:pfUB}\\
		&\sum_{k= (i,j,n) \in \cA} P_k + g_i^{sh} \hat{v}_i = \sum_{g \in \cG_i} s_g^p + u_i^p & \forall i \in \cN \label{app:balanceP}\\
		&\sum_{k= (i,j,n) \in \cA} Q_k - b_i^{sh} \hat{v}_i = \sum_{g \in \cG_i} s_g^q + u_i^q & \forall i \in \cN \label{app:balanceQ}\\
		&\widehat{vv}_{k}^2 \leq \frac{\hat{v}_i}{\tau_{1,k}^2} \frac{\hat{v}_j}{\tau_{2,k}^2} & \forall k = (i,j,n) \in \cA \label{app:vvleq}\\
		&\widehat{cs}_{k} + \frac{1-\cos \left(\theta^u_{k} \right)}{\left({\theta^u_{k}} \right)^2}(\theta_i - \sigma_k - \theta_j)^2 \leq 1  & \forall k = (i,j,n) \in \cA \label{app:cs1}\\
		&\widehat{cs}_{k} \geq \frac{\cos \left(\overline{\Delta}_{k} \right) - \cos \left(\underline{\Delta}_{k} \right)}{\overline{\Delta}_{k} - \underline{\Delta}_{k}} (\theta_i - \sigma_k - \theta_j - \underline{\Delta}_{k}) + \cos \left(\underline{\Delta}_{k} \right) & \forall k = (i,j,n) \in \cA \label{app:cs2}\\
		&\widehat{ss}_{k} - \cos \left(\frac{\theta^u_{k}}{2} \right) (\theta_i - \sigma_k - \theta_j) \leq \sin \left(\frac{\theta^u_{k}}{2} \right)-\frac{\theta^u_{k}}{2}\cos \left(\frac{\theta^u_{k}}{2} \right) & \forall k = (i,j,n) \in \cA \label{app:ss1}\\
		&-\widehat{ss}_{k} + \cos \left(\frac{\theta^u_{k}}{2} \right) (\theta_i - \sigma_k - \theta_j) \leq \sin \left(\frac{\theta^u_{k}}{2} \right)-\frac{\theta^u_{k}}{2}\cos \left(\frac{\theta^u_{k}}{2} \right)& \forall k = (i,j,n) \in \cA \label{app:ss2}\\
		& v_i^2 - \hat{v}_i \leq 0 & \forall i \in \cN \label{app:vsqr1}\\
		& \hat{v}_i - (\overline{v}_i + \underline{v}_i)v_i \leq - \overline{v}_i\underline{v}_i & \forall i \in \cN \label{app:vsqr2}\\
		&\widehat{vv}_{k} \geq \frac{\underline{v}_i}{\tau_{1,k}} \frac{v_j}{\tau_{2,k}} + \frac{\underline{v}_j}{\tau_{2,k}} \frac{v_i}{\tau_{1,k}} - \frac{\underline{v}_i}{\tau_{1,k}} \frac{\underline{v}_j}{\tau_{2,k}} &\forall k = (i,j,n) \in \cA \label{app:vv1}\\
		&\widehat{vv}_{k} \geq \frac{\overline{v}_i}{\tau_{1,k}} \frac{v_j}{\tau_{2,k}} + \frac{\overline{v}_j}{\tau_{2,k}} \frac{v_i}{\tau_{1,k}} - \frac{\overline{v}_i}{\tau_{1,k}} \frac{\overline{v}_j}{\tau_{2,k}} & \forall k = (i,j,n) \in \cA \label{app:vv2}\\
		&\widehat{vv}_{k} \leq \frac{\underline{v}_i}{\tau_{1,k}} \frac{v_j}{\tau_{2,k}} + \frac{\overline{v}_j}{\tau_{2,k}} \frac{v_i}{\tau_{1,k}} -  \frac{\underline{v}_i}{\tau_{1,k}} \frac{\overline{v}_j}{\tau_{2,k}} & \forall k = (i,j,n) \in \cA \label{app:vv3}\\
		&\widehat{vv}_{k} \leq \frac{\overline{v}_i}{\tau_{1,k}} \frac{v_j}{\tau_{2,k}} + \frac{\underline{v}_j}{\tau_{2,k}} \frac{v_i}{\tau_{1,k}} - \frac{\overline{v}_i}{\tau_{1,k}} \frac{\underline{v}_j}{\tau_{2,k}} & \forall k = (i,j,n) \in \cA \label{app:vv4}\\
		&\widehat{wc}_{k} \geq \frac{\underline{v}_i \underline{v}_j}{\tau_{1,k} \tau_{2,k}} \widehat{cs}_{k} + \underline{cs}_{k} \widehat{vv}_{k} - \frac{\underline{v}_i \underline{v}_j}{\tau_{1,k} \tau_{2,k}} \underline{cs}_{k} & \forall k = (i,j,n) \in \cA \label{app:wc1}\\
		&\widehat{wc}_{k} \geq \frac{\overline{v}_i \overline{v}_j}{\tau_{1,k} \tau_{2,k}} \widehat{cs}_{k}  + \overline{cs}_{k} \widehat{vv}_{k} - \frac{\overline{v}_i \overline{v}_j}{\tau_{1,k} \tau_{2,k}} \overline{cs}_{k} & \forall k = (i,j,n) \in \cA \label{app:wc2}\\
		&\widehat{wc}_{k} \leq \frac{\underline{v}_i \underline{v}_j}{\tau_{1,k} \tau_{2,k}} \widehat{cs}_{k} + \overline{cs}_{k} \widehat{vv}_{k} - \frac{\underline{v}_i \underline{v}_j}{\tau_{1,k} \tau_{2,k}} \overline{cs}_{k} & \forall k = (i,j,n) \in \cA \label{app:wc3}\\
		&\widehat{wc}_{k} \leq \frac{\overline{v}_i \overline{v}_j}{\tau_{1,k} \tau_{2,k}} \widehat{cs}_{k} + \underline{cs}_{k} \widehat{vv}_{k} - \frac{\overline{v}_i \overline{v}_j}{\tau_{1,k} \tau_{2,k}} \underline{cs}_{k} & \forall k = (i,j,n) \in \cA \label{app:wc4}\\
		&\widehat{ws}_{k} \geq \frac{\underline{v}_i \underline{v}_j}{\tau_{1,k} \tau_{2,k}} \widehat{ss}_{k} + \underline{ss}_{k} \widehat{vv}_{k} - \frac{\underline{v}_i \underline{v}_j}{\tau_{1,k} \tau_{2,k}} \underline{ss}_{k} & \forall k = (i,j,n) \in \cA \label{app:ws1}\\
		&\widehat{ws}_{k} \geq \frac{\overline{v}_i \overline{v}_j}{\tau_{1,k} \tau_{2,k}} \widehat{ss}_{k}  + \overline{ss}_{k} \widehat{vv}_{k} - \frac{\overline{v}_i \overline{v}_j}{\tau_{1,k} \tau_{2,k}} \overline{ss}_{k} & \forall k = (i,j,n) \in \cA \label{app:ws2}\\
		&\widehat{ws}_{k} \leq \frac{\underline{v}_i \underline{v}_j}{\tau_{1,k} \tau_{2,k}} \widehat{ss}_{k} + \overline{ss}_{k} \widehat{vv}_{k} - \frac{\underline{v}_i \underline{v}_j}{\tau_{1,k} \tau_{2,k}} \overline{ss}_{k} & \forall k = (i,j,n) \in \cA \label{app:ws3}\\
		&\widehat{ws}_{k} \leq \frac{\overline{v}_i \overline{v}_j}{\tau_{1,k} \tau_{2,k}} \widehat{ss}_{k} + \underline{ss}_{k} \widehat{vv}_{k} - \frac{\overline{v}_i \overline{v}_j}{\tau_{1,k} \tau_{2,k}} \underline{ss}_{k} & \forall k = (i,j,n) \in \cA \label{app:ws4}\\
		& \widehat{ws}_{k} - \tan(\overline{\Delta}_{k}) \widehat{wc}_{k} \leq 0 & \forall k = (i,j,n) \in \cA \label{app:trig1}\\
		& \widehat{ws}_{k} - \tan(\underline{\Delta}_{k}) \widehat{wc}_{k} \geq 0 & \forall k = (i,j,n) \in \cA \label{app:trig2}\\
		& \frac{v^\delta_i v^\delta_j }{\tau_{1,k} \tau_{2,k}} (\widehat{wc}_{k} \cos(\theta_{k}^\phi) + \widehat{ws}_{k} \sin(\theta_{k}^\phi)) - \frac{\overline{v}_j}{\tau_{2,k}} \cos(\theta_{k}^\delta) \frac{v^\delta_j}{\tau_{2,k}} \frac{\hat{v}_i}{\tau_{1,k}^2} & \nonumber \\
		& - \frac{\overline{v}_i}{\tau_{1,k}} \cos(\theta_{k}^\delta) \frac{v^\delta_i}{\tau_{1,k}} \frac{\hat{v}_j}{\tau_{2,k}^2} \geq \frac{\overline{v}_i \overline{v}_j}{\tau_{1,k} \tau_{2,k}} \cos(\theta_{k}^\delta) (\frac{\underline{v}_i \underline{v}_j}{\tau_{1,k} \tau_{2,k}} - \frac{\overline{v}_i \overline{v}_j}{\tau_{1,k} \tau_{2,k}}) & \forall k = (i,j,n) \in \cA \label{app:nlc1}\\
		& \frac{v^\delta_i v^\delta_j}{\tau_{1,k}\tau_{2,k}} (\widehat{wc}_{k} \cos(\theta_{k}^\phi) + \widehat{ws}_{k} \sin(\theta_{k}^\phi)) - \frac{\underline{v}_j}{\tau_{2,k}} \cos(\theta_{k}^\delta) \frac{v^\delta_j}{\tau_{2,k}} \frac{\hat{v}_i}{\tau_{1,k}^2} & \nonumber \\
		& - \frac{\underline{v}_i}{\tau_{1,k}} \cos(\theta_{k}^\delta) \frac{v^\delta_i}{\tau_{1,k}} \frac{\hat{v}_j }{\tau_{2,k}^2}\geq -\frac{\underline{v}_i \underline{v}_j}{\tau_{1,k} \tau_{2,k}} \cos(\theta_{k}^\delta) (\frac{\underline{v}_i \underline{v}_j}{\tau_{1,k} \tau_{2,k}} - \frac{\overline{v}_i \overline{v}_j}{\tau_{1,k} \tau_{2,k}}) & \forall k = (i,j,n) \in \cA \label{app:nlc2}\\
		& \widehat{cs}_{k} = \widehat{cs}_{\tilde{k}} & \forall k = (i,j,n) \in \cA \label{app:cseq}\\
		& \widehat{ss}_{k} = -\widehat{ss}_{\tilde{k}} & \forall k = (i,j,n) \in \cA \label{app:sseq}\\
		& \widehat{wc}_{k} = \widehat{wc}_{\tilde{k}} & \forall k = (i,j,n) \in \cA \label{app:wceq}\\
		& \widehat{ws}_{k} = -\widehat{ws}_{\tilde{k}} & \forall k = (i,j,n) \in \cA \label{app:wseq}\\
		& \widehat{vv}_{k} = \widehat{vv}_{\tilde{k}} & \forall k = (i,j,n) \in \cA. \label{app:vveq}
		\end{align}
	\end{subequations}
	Constraints~\eqref{app:vSimple}-\eqref{app:ssSimple} are simple bounds for voltage magnitude, difference in phase angles, phase angles and approximation terms for \(\cos(\theta_i -\sigma_k - \theta_j)\) and \(\sin(\theta_i - \sigma_k - \theta_j)\), respectively. Using the bound-tightening techniques in \citet{coffrin2015strengthening}, we can derive the upper bounds and lower bounds of \(\cos(\theta_i - \theta_j)\) and \(\sin(\theta_i - \theta_j)\) as:
	\begin{align*}
	&\overline{cs}_{k} = \cos(\overline{\Delta}_{k}),\ \  \underline{cs}_{k} = \cos(\underline{\Delta}_{k}) & \qquad & \text{if $\overline{\Delta}_{k} \leq 0$}\\
	&\overline{cs}_{k} = \cos(\underline{\Delta}_{k}),\ \  \underline{cs}_{k} = \cos(\overline{\Delta}_{k}) & \qquad & \text{if $\underline{\Delta}_{k} \geq 0$}\\
	&\overline{cs}_{k} = 1,\ \  \underline{cs}_{k} = \min\{\cos(\underline{\Delta}_{k}), \cos(\overline{\Delta}_{k})\} & \qquad & \text{if $\overline{\Delta}_{k} \geq 0$ and $\underline{\Delta}_{k} \leq 0$},
	\end{align*}
	and
	\begin{align*}
	&\overline{ss}_{k} = \sin(\overline{\Delta}_{k})\\
	&\underline{ss}_{k} = \sin(\underline{\Delta}_{k}).
	\end{align*}
	Constraints~\eqref{app:Pk} and \eqref{app:Qk} are linearized versions of the power flow constraints in~\eqref{eqn:pf}, and constraint~\eqref{app:pfUB} replicates~\eqref{eqn:lineConstr}, while active and reactive power flow balance constraints~\eqref{app:balanceP} and \eqref{app:balanceQ} replicate~\eqref{eqn:balance}. Since we are modeling the deterministic ACOPF problem here, we do not have the recourse freedom variables \(o^{p,+}, o^{p,-}, o^{q,+},o^{q,-}\) in this formulation. Constraint~\eqref{app:vvleq} represents the correct quantitative relationship between \(\widehat{vv}_{k}\) and \(\hat{v}_i \hat{v}_j\), where \(\widehat{vv}\) can be considered as a relaxation of the bilinear term \(v_i v_j\). Constraint~\eqref{app:cs1}-\eqref{app:vsqr2} are linear and convex quadratic bounding approximations of cosine, sine and quadratic functions. For multi-linear terms, a McCormick scheme is applied to relax \(v_i v_j\), \(v_i v_j \cos(\theta_i - \theta_j)\) and \(v_i v_j \sin(\theta_i - \theta_j)\) in constraints~\eqref{app:vv1}-\eqref{app:ws4}.
	
	The model includes valid inequalities described in \citet{coffrin2016strengthening} to further tighten this convex relaxation. One set of valid constraints uses the trigonometric relationship \(\tan(\theta) = \frac{\sin(\theta)}{\cos(\theta)} \) to build valid inequalities on \(\widehat{wc}\) and \(\widehat{ws}\) as in~\eqref{app:trig1} and \eqref{app:trig2}, using the fact that \(\widehat{wc} \geq 0\). Another set of valid constraints, constraints~\eqref{app:nlc1}-\eqref{app:nlc2}, is called a lifted nonlinear cut. See \citet{coffrin2016strengthening} for a detailed derivation of the lifted nonlinear cut. Finally, we use constraints~\eqref{app:cseq}-\eqref{app:wseq} to make sure variables describing the forward flow are consistent with those describing the backward flow between the same pair of buses, where \(\tilde{k} = (j,i,n)\) represents the backward flow of the flow \(k = (i,j,n)\).
	
	\vspace{1cm}
	\section{Detailed Formulation of Model~\eqref{prob:dualSubIBudget}} \label{appen:QCDual}
	In this section we first expand the formulation of model~\eqref{prob:dualSubIBudget}, and then we explain the corresponding relationship between the dual variables and the primal constraints in formulation~\eqref{prob:dualSubIBudget} and Appendix~\ref{appen:QCRelaxation}.
	
	\begin{subequations}
		\footnotesize
		\begin{align} 
		\max \quad & -\sum_{k \in \cA} W_k \nu_{1,k} - \sum_{i \in \cN} \left[ \sum_{g \in \cG_i} \left(\hat{s}^p_g \lambda^p_i + \hat{s}^q_g \lambda^q_i\right) + (\overline{u}^p_i - u^{p,0}_i) (r^{p,+}_i + \zeta_i^+ r^{op,+})+ (\underline{u}^{p,-}_i - u^{p,0}_i) (r^{p,-}_i+ \zeta_i^- r^{op,-}) + \right.&\nonumber\\
		&(\overline{u}^q_i - u^{q,0}_i) (r^{q,+}_i + \zeta_i^+ r^{oq,+}) + (\underline{u}^{q,-}_i - u^{q,0}_i) (r^{q,-}_i + \zeta_i^+ r^{oq,-}) + u^{p,0}_i \lambda^p_i + u^{q,0}_i \lambda^q_i -& \nonumber\\
		& \left. \frac{1}{4} \mu_{4,i2} + \frac{1}{4} \nu_{4,i} + \overline{v}_i \lambda^{vu}_{i} + \underline{v}_i \lambda^{vl}_{i} + \overline{o}^p_i \lambda^{op,-}_{i} + \overline{o}^q_i \lambda^{oq,-}_{i} + \overline{v}_i \underline{v}_i \lambda^v_{i} + \lambda^{op}_i (\overline{o}^p_i + h^p_i) + \lambda^{oq}_i (\overline{o}^q_i + h^q_i) \right] - & \nonumber \\
		& \sum_{k = (i,j,n) \in \cA} \left[ -\frac{3}{4} \mu_{3,k2} + \frac{5}{4} \nu_{3,k} - \sqrt{\frac{1-\cos \theta^u_k}{(\theta^u_k)^2}} \mu_{3,k1} + \right. & \nonumber\\
		& \overline{cs}_{k} \lambda^{cs1}_{k} + \underline{cs}_{k} \lambda^{cs2}_{k} + \overline{ss}_{k} \lambda^{ss1}_{k} + \underline{ss}_{k} \lambda^{ss2}_{k} + (\overline{\Delta}_{k} + \sigma_k) \lambda^{\theta 1}_{k} + (\underline{\Delta}_{k} + \sigma_k) \lambda^{\theta 2}_{k} + & \nonumber\\
		& \left(\cos \underline{\Delta}_{k} - \frac{\cos \overline{\Delta}_{k} - \cos \underline{\Delta}_{k}}{\overline{\Delta}_{k} - \underline{\Delta}_{k}} (\underline{\Delta}_k + \sigma_k) \right) \lambda^{cs3}_{k}  + & \nonumber\\
		& \left(\sin \frac{\theta^u_{k}}{2} - \frac{\theta^u_{k}}{2} \cos \frac{\theta^u_{k}}{2}\right)(\lambda^{ss3}_{k} - \lambda^{ss4}_{k}) - & \nonumber \\
		& \frac{\underline{v}_i \underline{v}_j}{\tau_{1,k} \tau_{2,k}} \lambda^{vv1}_{k} - \frac{\overline{v}_i \overline{v}_j}{\tau_{1,k} \tau_{2,k}} \lambda^{vv2}_{k} - \frac{\underline{v}_i \overline{v}_j}{\tau_{1,k} \tau_{2,k}} \lambda^{vv3}_{k} - \frac{\overline{v}_i \underline{v}_j }{\tau_{1,k} \tau_{2,k}} \lambda^{vv4}_{k} + & \nonumber \\
		& \frac{\overline{v}_i \overline{v}_j}{\tau_{1,k} \tau_{2,k}} \cos \theta^{\delta}_{k} (\frac{\underline{v}_i \underline{v}_j}{\tau_{1,k} \tau_{2,k}} - \frac{\overline{v}_i \overline{v}_j}{\tau_{1,k} \tau_{2,k}}) \lambda^{lnc1}_{k} - \frac{\underline{v}_i \underline{v}_j}{\tau_{1,k} \tau_{2,k}} \cos \theta^{\delta}_{k} (\frac{\underline{v}_i \underline{v}_j}{\tau_{1,k} \tau_{2,k}} - \frac{\overline{v}_i \overline{v}_j}{\tau_{1,k} \tau_{2,k}}) \lambda^{lnc2}_{k} -  & \nonumber \\
		& \underline{cs}_{k} \frac{\underline{v}_i \underline{v}_j}{\tau_{1,k} \tau_{2,k}} \lambda^{wc1}_{k} - \overline{cs}_{k} \frac{\overline{v}_i \overline{v}_j}{\tau_{1,k} \tau_{2,k}} \lambda^{wc2}_{k} - \overline{cs}_{k} \frac{\underline{v}_i \underline{v}_j}{\tau_{1,k} \tau_{2,k}} \lambda^{wc3}_{k} - \underline{cs}_{k} \frac{\overline{v}_i \overline{v}_j}{\tau_{1,k} \tau_{2,k}} \lambda^{wc4}_{k} - & \nonumber \\
		& \underline{ss}_{k} \frac{\underline{v}_i \underline{v}_j}{\tau_{1,k} \tau_{2,k}} \lambda^{ws1}_{k} - \overline{ss}_{k} \frac{\overline{v}_i \overline{v}_j}{\tau_{1,k} \tau_{2,k}} \lambda^{ws2}_{k} - \overline{ss}_{k} \frac{\underline{v}_i \underline{v}_j}{\tau_{1,k} \tau_{2,k}} \lambda^{ws3}_{k} - \underline{ss}_{k} \frac{\overline{v}_i \overline{v}_j}{\tau_{1,k} \tau_{2,k}} \lambda^{ws4}_{k} \Bigg] &\\
		\text{s.t.} \quad & \lambda^{pt}_{k} - \mu_{1,k1} + \lambda^{p}_{i} = 0 \quad \qquad \qquad \qquad \qquad \qquad \qquad \qquad \qquad \qquad \qquad \qquad \quad \forall k = (i,j,n) \in \cA & \label{app:dP}\\
		& \lambda^{qt}_{k} - \mu_{1,k2} + \lambda^{q}_{i} = 0 \quad \qquad \qquad \qquad \qquad \qquad \qquad \qquad \qquad \qquad \qquad \qquad \quad \forall k = (i,j,n) \in \cA & \label{app:dQ} \\
		& \|\mu_{1,k}\| \leq \nu_{1,k} \quad \qquad \qquad \qquad \qquad \qquad \qquad \qquad \qquad \qquad \qquad \qquad \qquad \qquad \forall k \in \cA \label{app:SOCP1}\\
		& \|\mu_{2,k}\| \leq \nu_{2,k} \quad \qquad \qquad \qquad \qquad \qquad \qquad \qquad \qquad \qquad \qquad \qquad \qquad \qquad \forall k \in \cA \label{app:SOCP2}\\
		& \|\mu_{3,k}\| \leq \nu_{3,k} \quad \qquad \qquad \qquad \qquad \qquad \qquad \qquad \qquad \qquad \qquad \qquad \qquad \qquad \forall k \in \cA \label{app:SOCP3}\\
		& \|\mu_{4,i}\| \leq \nu_{4,i} \quad \qquad \qquad \qquad \qquad \qquad \qquad \qquad \qquad \qquad \qquad \qquad \qquad \qquad \; \forall i \in \cN \label{app:SOCP4}\\
		& -\sum_{k = (i,j,n) \in \cA} \left(\frac{\underline{v}_j}{\tau_{1,k} \tau_{2,k}} \lambda^{vv1}_{k} + \frac{\overline{v}_j}{\tau_{1,k} \tau_{2,k}} \lambda^{vv2}_{k} + \frac{\overline{v}_j}{\tau_{1,k} \tau_{2,k}} \lambda^{vv3}_{k} + \frac{\underline{v}_j}{\tau_{1,k} \tau_{2,k}} \lambda^{vv4}_{k} \right) - & \nonumber \\ 
		&\sum_{k = (j,i,n) \in \cA} \left(\frac{\underline{v}_j}{\tau_{1,k} \tau_{2,k}} \lambda^{vv1}_{k} + \frac{\overline{v}_j}{\tau_{1,k} \tau_{2,k}} \lambda^{vv2}_{k} + \frac{\underline{v}_j}{\tau_{1,k} \tau_{2,k}} \lambda^{vv3}_{k} + \frac{\overline{v}_j}{\tau_{1,k} \tau_{2,k}} \lambda^{vv4}_{k} \right) - & \nonumber\\
		& \mu_{4,i1} - (\overline{v}_i + \underline{v}_i) \lambda^v_{i} + \lambda^{vu}_{i} + \lambda^{vl}_{i} = 0 \qquad \qquad \qquad \qquad \qquad \qquad \qquad \qquad \qquad \; \forall i \in \cN & \label{app:dv}\\
		& \sum_{k = (i,j,n) \in \cA}\left[-g_k\frac{\lambda^{pt}_{k}}{\tau_{1,k}^2} + (b_k+\frac{b_k^c}{2}) \frac{\lambda^{qt}_{k}}{\tau_{2,k}^2}  \right]  - \mu_{4,i2} - \nu_{4,i} + \lambda^v_{i} + \lambda^p_i g_i^{sh} - \lambda^q_i b_i^{sh} - & \nonumber \\
		& \sum_{k = (i,j,n) \in \cA} \left[ \frac{\mu_{2,k2}}{\tau_{1,k}^2 \sqrt{2}} +\frac{\nu_{2,k}}{\tau_{1,k}^2 \sqrt{2}} + \cos \theta^\delta_{k} \frac{v^\delta_{j}}{\tau_{2,k}} \left(\frac{\overline{v}_j}{\tau_{2,k}} \frac{\lambda^{lnc1}_{k}}{\tau_{1,k}^2} + \frac{\underline{v}_j}{\tau_{2,k}} \frac{\lambda^{lnc2}_{k}}{\tau_{1,k}^2} \right) \right] - & \nonumber\\
		&\sum_{k = (j,i,n) \in \cA} \left[ \frac{\mu_{2,k3}}{\tau_{2,k}^2 \sqrt{2}} +\frac{\nu_{2,k}}{\tau_{2,k}^2 \sqrt{2}} + \cos \theta^\delta_{k} \frac{v^\delta_{j}}{\tau_{1,k}} \left(\frac{\overline{v}_j}{\tau_{1,k}} \frac{\lambda^{lnc1}_{k}}{\tau_{2,k}^2} + \frac{\underline{v}_j}{\tau_{1,k}} \frac{\lambda^{lnc2}_{k}}{\tau_{2,k}^2} \right) \right] = 0 \qquad \forall i \in \cN&  \label{app:dvhat}\\
		& \lambda^{vv1}_{k} + \lambda^{vv2}_{k} + \lambda^{vv3}_{k} + \lambda^{vv4}_{k} + \lambda^{vve}_{k} - \lambda^{vve}_{\tilde{k}} - \mu_{2,k1} - \underline{cs}_{k} \lambda^{wc1}_{k} - \overline{cs}_{k} \lambda^{wc2}_{k} -  & \nonumber \\
		& \overline{cs}_{k} \lambda^{wc3}_{k} - \underline{cs}_{k} \lambda^{wc4}_{k} - \underline{ss}_{k} \lambda^{ws1}_{k} - \overline{ss}_{k} \lambda^{ws2}_{k} - \overline{ss}_{k} \lambda^{ws3}_{k} - \underline{ss}_{k} \lambda^{ws4}_{k} = 0 \qquad \qquad \; \; \; \forall k \in \cA & \label{app:dvvhat}\\
		& g_k \lambda^{pt}_{k} - b_k \lambda^{qt}_{k} - \tan \overline{\Delta}_{k} \lambda^{\tan 1}_{k} - \tan \underline{\Delta}_{k} \lambda^{\tan 2}_{k} + \lambda^{wc1}_{k} + \lambda^{wc2}_{k} + \lambda^{wc3}_{k} + \lambda^{wc4}_{k} + &\nonumber \\
		& \lambda^{wce}_{k} - \lambda^{wce}_{\tilde{k}} + \frac{v^{\delta}_i v^{\delta}_j}{\tau_{1,k} \tau_{2,k}} \cos \theta^{\phi}_{k} (\lambda^{lnc1}_{k} + \lambda^{lnc2}_{k}) = 0  \qquad \qquad \qquad \qquad \qquad \qquad \; \forall k \in \cA & \label{app:dwc}\\
		& b_k \lambda^{pt}_{k} + g_k \lambda^{qt}_{k} + \lambda^{ws1}_{k} + \lambda^{ws2}_{k} + \lambda^{ws3}_{k} + \lambda^{ws4}_{k} + \lambda^{wse}_{k} + \lambda^{wse}_{\tilde{k}} + & \nonumber \\
		& \lambda^{\tan 1}_{k} + \lambda^{\tan 2}_{k} + \frac{v^{\delta}_i v^{\delta}_j}{\tau_{1,k} \tau_{2,k}} \sin \theta^{\phi}_{k} (\lambda^{lnc1}_{k} + \lambda^{lnc2}_{k}) = 0 \qquad \qquad \qquad \qquad \qquad \quad \;\; \forall k \in \cA & \label{app:dws}\\
		& -\frac{\underline{v}_i \underline{v}_j}{\tau_{1,k} \tau_{2,k}} \lambda^{wc1}_{k} - \frac{\overline{v}_i \overline{v}_j}{\tau_{1,k} \tau_{2,k}} \lambda^{wc2}_{k} - \frac{\underline{v}_i \underline{v}_j}{\tau_{1,k} \tau_{2,k}} \lambda^{wc3}_{k} - \frac{\overline{v}_i \overline{v}_j}{\tau_{1,k} \tau_{2,k}} \lambda^{wc4}_{k}  + & \nonumber\\
		&\lambda^{cs1}_{k} + \lambda^{cs2}_{k} + \lambda^{cs3}_{k} - \mu_{3,k2} + \nu_{3,k} +\lambda^{cse}_{k} - \lambda^{cse}_{\tilde{k}} = 0 \qquad \qquad \qquad \qquad \qquad \quad \; \forall k = (i,j,n) \in \cA & \label{app:dcs} \\
		& -\frac{\underline{v}_i \underline{v}_j}{\tau_{1,k} \tau_{2,k}} \lambda^{ws1}_{k} - \frac{\overline{v}_i \overline{v}_j}{\tau_{1,k} \tau_{2,k}} \lambda^{ws2}_{k} - \frac{\underline{v}_i \underline{v}_j}{\tau_{1,k} \tau_{2,k}} \lambda^{ws3}_{k} - \frac{\overline{v}_i \overline{v}_j}{\tau_{1,k} \tau_{2,k}} \lambda^{ws4}_{k} + & \nonumber \\
		& \lambda^{ss1}_{k} + \lambda^{ss2}_{k} + \lambda^{ss3}_{k} + \lambda^{ss4}_{k} +\lambda^{sse}_{k} + \lambda^{sse}_{\tilde{k}} = 0 \qquad \qquad \qquad \qquad \qquad \qquad \qquad \; \forall k = (i,j,n) \in \cA & \label{app:dss} \\
		& \sum_{k = (i,j,n) \in \cA} \left(\lambda^{\theta 1}_{k} + \lambda^{\theta 2}_{k} -   \frac{\cos \overline{\Delta}_{k} - \cos \underline{\Delta}_{k}}{\overline{\Delta}_{k} - \underline{\Delta}_{k}} \lambda^{cs3}_{k} - \cos \frac{\theta^u_{k}}{2} (\lambda^{ss3}_{k} + \lambda^{ss4}_{k}) - \right. & \nonumber \\
		& \left. \sqrt{\frac{1 - \cos \theta^u_{k}}{\theta^u_{k}}} \mu_{3,k1} \right)  + \sum_{k = (j,i,n) \in \cA} \left(-\lambda^{\theta 1}_{k} - \lambda^{\theta 2}_{k} + \frac{\cos \overline{\Delta}_{k} - \cos \underline{\Delta}_{k}}{\overline{\Delta}_{k} - \underline{\Delta}_{k}} \lambda^{cs3}_{k} + \right.& \nonumber \\
		& \left. \cos \frac{\theta^u_{k}}{2} (\lambda^{ss3}_{k} + \lambda^{ss4}_{k}) + \sqrt{\frac{1 - \cos \theta^u_{k}}{\theta^u_{k}}} \mu_{3,k1} \right) = 0 \qquad \qquad \qquad \qquad \qquad \qquad \; \forall i \in \cN & \label{app:dtheta}\\
		& \lambda^{op,-}_{i} \geq \lambda^p_i   \quad \qquad \qquad \qquad \qquad \qquad \qquad \qquad \qquad \qquad \qquad \qquad \qquad \qquad \forall i \in \cN & \label{app:dop1}\\
		& \lambda^{op}_{i} \geq -\lambda^p_i   \quad \qquad \qquad \qquad \qquad \qquad \qquad \qquad \qquad \qquad \qquad \qquad \qquad \qquad \forall i \in \cN & \label{app:dop2} \\
		& \lambda^{oq,-}_{i} \geq \lambda^q_i   \quad \qquad \qquad \qquad \qquad \qquad \qquad \qquad \qquad \qquad \qquad \qquad \qquad \qquad \forall i \in \cN & \label{app:doq1}\\
		& \lambda^{oq}_{i} \geq -\lambda^q_i   \quad \qquad \qquad \qquad \qquad \qquad \qquad \qquad \qquad \qquad \qquad \qquad \qquad \qquad \forall i \in \cN & \label{app:doq2}\\
		& -y^+_m \leq r^{p,+}_i \leq y^+_m  \quad \qquad \qquad \qquad \qquad \qquad \qquad \qquad \qquad \qquad \qquad \qquad \; \; \; \forall m \in \cM, i \in \cN_m & \label{app:drpp1}\\
		& -y^-_m \leq r^{p,-}_i \leq y^-_m \quad \qquad \qquad \qquad \qquad \qquad \qquad \qquad \qquad \qquad \qquad \qquad \; \; \;  \forall m \in \cM, i \in \cN_m & \label{app:drpm1}\\
		& -y^+_m \leq r^{q,+}_i \leq y^+_m \quad \qquad \qquad \qquad \qquad \qquad \qquad \qquad \qquad \qquad \qquad \qquad \; \; \;  \forall m \in \cM, i \in \cN_m & \label{app:drqp1}\\
		& -y^-_m \leq r^{q,-}_i \leq y^-_m \quad \qquad \qquad \qquad \qquad \qquad \qquad \qquad \qquad \qquad \qquad \qquad \; \; \;  \forall m \in \cM, i \in \cN_m & \label{app:drqm1}\\
		& \lambda^{p}_i - 1 + y^+_m \leq r^{p,+}_i \leq \lambda^{p}_i + 1 - y^+_m \quad \qquad \qquad \qquad \qquad \qquad \qquad \qquad \qquad \; \forall m \in \cM, i \in \cN_m & \label{app:drpp2}\\
		& \lambda^{p}_i - 1 + y^-_m \leq r^{p,-}_i \leq \lambda^{p}_i + 1 - y^-_m \quad \qquad \qquad \qquad \qquad \qquad \qquad \qquad \qquad \; \forall m \in \cM, i \in \cN_m & \label{app:drpm2}\\
		& \lambda^{q}_i - 1 + y^+_m \leq r^{q,+}_i \leq \lambda^{q}_i + 1 - y^+_m \quad \qquad \qquad \qquad \qquad \qquad \qquad \qquad \qquad \; \forall m \in \cM, i \in \cN_m & \label{app:drqp2}\\
		& \lambda^{q}_i - 1 + y^-_m \leq r^{q,-}_i \leq \lambda^{q}_i + 1 - y^-_m \quad \qquad \qquad \qquad \qquad \qquad \qquad \qquad \qquad \; \forall m \in \cM, i \in \cN_m & \label{app:drqm2}\\
		& -y^+_m \leq r^{op,+}_i \leq y^+_m \quad \qquad \qquad \qquad \qquad \qquad \qquad \qquad \qquad \qquad \qquad \qquad \; \forall m \in \cM, i \in \cN_m & \label{app:dropp1}\\
		& -y^-_m \leq r^{op,-}_i \leq y^-_m \quad \qquad \qquad \qquad \qquad \qquad \qquad \qquad \qquad \qquad \qquad \qquad \; \forall m \in \cM, i \in \cN_m & \label{app:dropm1}\\
		& -y^+_m \leq r^{oq,+}_i \leq y^+_m \quad \qquad \qquad \qquad \qquad \qquad \qquad \qquad \qquad \qquad \qquad \qquad \; \forall m \in \cM, i \in \cN_m & \label{app:droqp1}\\
		& -y^-_m \leq r^{oq,-}_i \leq y^-_m \quad \qquad \qquad \qquad \qquad \qquad \qquad \qquad \qquad \qquad \qquad \qquad \; \forall m \in \cM, i \in \cN_m & \label{app:droqm1}\\
		& \lambda^{op}_i - 1 + y^+_m \leq r^{op,+}_i \leq \lambda^{op}_i + 1 - y^+_m \quad \qquad \qquad \qquad \qquad \qquad \qquad \qquad \quad \forall m \in \cM, i \in \cN_m & \label{app:dropp2}\\
		& \lambda^{op}_i - 1 + y^-_m \leq r^{op,-}_i \leq \lambda^{op}_i + 1 - y^-_m \quad \qquad \qquad \qquad \qquad \qquad \qquad \qquad \quad \forall m \in \cM, i \in \cN_m & \label{app:dropm2}\\
		& \lambda^{oq}_i - 1 + y^+_m \leq r^{oq,+}_i \leq \lambda^{oq}_i + 1 - y^+_m \quad \qquad \qquad \qquad \qquad \qquad \qquad \qquad \quad \forall m \in \cM, i \in \cN_m & \label{app:droqp2}\\
		& \lambda^{oq}_i - 1 + y^-_m \leq r^{oq,-}_i \leq \lambda^{oq}_i + 1 - y^-_m \quad \qquad \qquad \qquad \qquad \qquad \qquad \qquad \quad \forall m \in \cM, i \in \cN_m & \label{app:droqm2}\\
		& y_m^+ + y_m^- \leq 1 \qquad \qquad \qquad \qquad \qquad \qquad \qquad \qquad \qquad \qquad \qquad \qquad \qquad \forall m \in \cM &\label{app:budget1p}\\
		& \sum_{m \in \cM} \left(y_m^+ + y_m^- \right) \leq \Gamma \label{app:budget}&\\
		& -1 \leq \lambda^{p} \leq 1& \\
		& -1 \leq \lambda^{q} \leq 1& \\
		& \lambda^{cs1}, \lambda^{ss1}, \lambda^{ss3}, \lambda^{v}, \lambda^{vu}, \lambda^{\theta 1}, \lambda^{vv3}, \lambda^{vv4}, \lambda^{wc3}, \lambda^{wc4}, \lambda^{ws3}, \lambda^{ws4} \geq 0& \\ 
		& \lambda^{op}, \lambda^{oq}, \lambda^{op,-}, \lambda^{oq,-}, \lambda^{\tan 1}, \nu_{1},\nu_{2}, \nu_{3},\nu_{4} \geq 0 & \\
		& \lambda^{vl}, \lambda^{cs2}, \lambda^{cs3}, \lambda^{ss2}, \lambda^{ss4}, \lambda^{\theta 2}, \lambda^{vv1}, \lambda^{vv2}, \lambda^{wc1}, \lambda^{wc2}, \lambda^{ws1}, \lambda^{ws2}, \lambda^{\tan 2}, \lambda^{lnc1}, \lambda^{lnc2} \leq 0 & \\
		& y^{+},y^{-} \in \{0,1\}^{|\cM|}.&
		\end{align}
	\end{subequations}
	
	The formulation above is an exact form of problem~\eqref{prob:dualSubIBudget}. In this formulation, \(\lambda^{pt}\) and \(\lambda^{qt}\) are the dual variables of line transmission constraints~\eqref{app:Pk} and~\eqref{app:Qk}, while \(\lambda^{p}\) and \(\lambda^{q}\) are the dual variables of the equivalent constraints of flow balance constraints~\eqref{app:balanceP} and~\eqref{app:balanceQ} in problem~\eqref{prob:dualSubIBudget}. For approximation of function \(\cos\) and \(\sin\), we use \(\lambda^{cs1}, \lambda^{cs2}, \lambda^{ss1}\) and \(\lambda^{ss2}\) to represent the dual variables of the upper bound and the lower bound of \(\widehat{cs}\) and \(\widehat{ss}\) respectively, and \(\lambda^{cs3}, \lambda^{ss3}\) and \(\lambda^{ss4}\) for constraints~\eqref{app:cs2}-\eqref{app:ss2}. Dual variables \(\lambda^{v}\) correspond to constraint~\eqref{app:vsqr2}. We denote the dual variables for the recourse freedom bounds as \(\lambda^{op,-}\), \(\lambda^{oq,-}\), \(\lambda^{op}\) and \(\lambda^{oq}\).
	
	We use \(\lambda^{vv1}\)-\(\lambda^{vv4}\), \(\lambda^{wc1}\)-\(\lambda^{wc4}\), \(\lambda^{ws1}\)-\(\lambda^{ws4}\) as the dual variables of McCormick relaxation constraints~\eqref{app:vv1}-\eqref{app:ws4}. The dual variables \(\lambda^{\tan 1}\) and \(\lambda^{\tan 2}\) correspond to the tangent tightening constraints~\eqref{app:trig1} and~\eqref{app:trig2}, and \(\lambda^{lnc1}\) and \(\lambda^{lnc2}\) correspond to the lifted nonlinear cuts~\eqref{app:nlc1} and~\eqref{app:nlc2}. For variable equality enforcement constraints~\eqref{app:cseq}-\eqref{app:vveq}, we use \(\lambda^{cse}, \lambda^{sse}, \lambda^{wce}, \lambda^{wse}\) and \(\lambda^{vve}\) as their dual variables. The remaining SOCP constraints~\eqref{app:pfUB}, \eqref{app:vvleq}, \eqref{app:cs1} are \eqref{app:vsqr1} have their dual variables as \((\mu_{1},\nu_{1})\), \((\mu_{2},\nu_{2})\), \((\mu_{3},\nu_{3})\) and \((\mu_{4},\nu_{4})\), respectively. Notice that those SOCP constraints can be rewritten in a standard form for duality derivation:
	\begin{align*}
	&\|(P_k, Q_k) \|_2 \leq W_k & \forall k \in \cA\\
	& \left\| \left(\widehat{vv}_k, \frac{\hat{v}_i}{\tau_{1,k}^2 \sqrt{2}}, \frac{\hat{v}_j}{\tau_{2,k}^2 \sqrt{2}} \right) \right\|_2 \leq \frac{\hat{v}_i/\tau_{1,k}^2 + \hat{v}_j/\tau_{2,k}^2}{\sqrt{2}}& \forall k = (i,j,n) \in \cA\\
	& \left\| \left(\sqrt{\frac{1-\cos \theta^u_{k}}{{\theta^u_{k}}^2}} (\theta_i - \sigma_k - \theta_j), \widehat{cs}_{k} - \frac{3}{4} \right) \right\|_2 \leq \frac{5}{4} - \widehat{cs}_{k} & \forall k = (i,j,n) \in \cA\\
	& \left\| \left(v_i, \hat{v}_i - \frac{1}{4} \right) \right\|_2 \leq \hat{v}_i + \frac{1}{4}& \forall i \in \cN,
	\end{align*}
	which means that \(\mu_{1,k}\), \(\mu_{2,k}\), \(\mu_{3,k}\) and \(\mu_{4,i}\) are vectors with cardinality \(2,3,2\) and \(2\), respectively, while \(\nu_{1,k}\), \(\nu_{2,k}\), \(\nu_{3,k}\) and \(\nu_{4,i}\) are scalars. The dual SOCP constraints are formed as~\eqref{app:SOCP1}-\eqref{app:SOCP4}.
	
	With the dual variables established, we build constraint~\eqref{app:dP} for the primal variable \(P\), \eqref{app:dQ} for \(Q\), \eqref{app:dv} for \(v\), \eqref{app:dvhat} for \(\hat{v}\), \eqref{app:dvvhat} for \(\widehat{vv}\), \eqref{app:dwc} for \(\widehat{wc}\), \eqref{app:dws} for \(\widehat{ws}\), \eqref{app:dcs} for \(\widehat{cs}\), \eqref{app:dss} for \(\widehat{ss}\), and \eqref{app:dtheta} for \(\theta\). Constraints~\eqref{app:dop1}-\eqref{app:doq2} characterize the dual constraints for primal variables \(o^{p,+}, o^{p,-}, o^{q,+}\) and \(o^{q,-}\). Constraints~\eqref{app:drpp1}-\eqref{app:droqm2} are the direct replicate of the linearization constraints of bilinear terms in problem~\eqref{prob:dualSubIBudget}. Constraints~\eqref{app:budget1p} guarantees that for each bus the uncontrollable injection can be either at the nominal value or at one of the bounds. Constraint~\eqref{app:budget} is the budget constraint.
	
	\vspace{1cm}
	\section{Bound Tightening Process}	\label{appen:boundTightening}
	\citet{coffrin2015strengthening,coffrin2016strengthening} introduce a bound-tightening process for the QC relaxation involving the bounds that appear in constraints~\eqref{eqn:thetaDiff} and~\eqref{eqn:volCons}. A new variable, \(\theta^d_{k}\), is created to represent the phase angle difference between two buses of the line \(k = (i,j,n) \in \cA\), and appended to \(x\). The constraints defining \(\theta^d_{k}\), \(\theta^d_{k} = \theta_{i} - \theta_{j}\), are included in the general form linear constraints \(Ax \leq b\). The process iteratively updates the bounds \(v_i=\underline{v}_i \ \mbox{or} \ \overline{v}_i\) at bus \(i \in \cN\) and the phase angle difference \(\theta^d_k=\underline{\Delta}_{k} \ \mbox{or} \ \overline{\Delta}_{k}\) of line \(k = (i,j,n) \in \cA\), by a set of QC relaxation problems with the objective function substituted by \(v_i,\ \forall i \in\cN\) or \(\theta^d_{k},\ \forall k = (i,j,n) \in \cA\):
	\begin{subequations}
		\label{prob:boundT}
		\begin{align}
		\min_{s,x,u} \text{ or } \max_{s,x,u} \quad & x_{loc} & \\
		\text{s.t.} \quad 	&  \underline{s}\leq s \leq \overline{s} & \\
		& Ax \leq b &  \\
		& \|B_i x + a_i \|_2 \leq e_i^\top x + f_i & \forall i = 1, \dots, m_c \\
		& A^p x = D s^p + u^{p} &\\
		& A^q x = D s^q + u^{q} &\\
		& A^{op} x \leq \overline{o}^p + (1 + \alpha^{h,+})h^p &\\
		& A^{oq} x \leq \overline{o}^q + (1 + \alpha^{h,+})h^q &\\
		& \underline{u}^p \leq u^p \leq \overline{u}^p &\\
		& \underline{u}^q \leq u^q \leq \overline{u}^q. &
		\end{align}
	\end{subequations}
	
	In this formulation we treat uncertain uncontrollable injections as decision variables so that the resulting upper and lower bounds are valid for all \((u^p,u^q)\) where \(\underline{u}^p \leq u^p \leq \overline{u}^p\) and \(\underline{u}^q \leq u^q \leq \overline{u}^q\), hence for all \(u \in \cU\). In the objective function the subscript \(loc\) references the position of \(v_i\) or \(\theta^d_{k}\) in the decision vector \(x\). The bounds are iteratively updated using the optimal solutions from problem~\eqref{prob:boundT} for each \(v_i\) and \(\theta^d_{k}\). The process terminates when changes in the bounds are negligible.
	
	We perform bound tightening as a preprocessing step prior to running optimization. We focus on the bounds on \(v_i\) and \(\theta^d_{k}\) because the tightness of our linear-quadratic relaxation for the sine and cosine functions and that of the McCormick relaxation for the multi-linear terms depends on the bounds of \(v_i\) and \(\theta^d_{k}\). As illustrated in \citet{coffrin2016strengthening}, tightening these bounds tightens the QC relaxation and allows for a tighter lower bound on the nonconvex ACOPF problem. We employ bound tightening in all results reported in Section~\ref{sec:exp}. We {show the runtimes of the bound-tightening process for each test case but }do not give detailed improvements from this process beyond indicating here that the optimal values of instances of model~\eqref{prob:rcACOPF} grow by 1-10\% by tightening these simple bounds.
	
	 \begin{table}[h]
	 	\centering
	 	\begin{tabular}{ l | r }
	 		\hline
	 		Test Case & Time (sec.) \\
	 		\hline
	 		Case 5 & 1.8\\
	 		Case 9 & 3.6\\
	 		Case 14 & 8.0\\
	 		Case 30 & 29.2\\
	 		Case 118 & 809.3\\
	 		Case 300 & 5759.0\\
	 		Case 2383 & 61001.3\\
	 		Case 2746 & 175095.7\\
	 		\hline
	 	\end{tabular}
	 	\caption{Computational effort for the bound-tightening process.}
	 	\label{table:btTable}
	 \end{table}
	
	\vspace{1cm}
	\section{Regularized Cutting-plane Algorithm}
	\label{appen:regu}
	In this section we describe the regularized cutting-plane algorithm mentioned in Section~\ref{subsec:scenAppend}. Given a current incumbent solution, \(\hat{s}\), we modify the master problem from model~\eqref{prob:master} by adding a quadratic regularization term, as indicated in model~\eqref{prob:masterR}. In general, the regularization term prevents large changes in incumbent solutions between iterations, which can stabilize the algorithm and encourage faster converge. 
	\begin{subequations}
		\label{prob:masterR}
		\begin{align}
		(M^R) \quad \min \quad & c(s^p,s^q) + \frac{\rho}{2} \|(s^p,s^q) - (\hat{s}^p,\hat{s}^q) \|_2^2 &\\
		\text{s.t.} \quad & \underline{s} \leq s \leq \overline{s} & \\
		& -{\lambda^{p,k}}^\top D s^p_i - {\lambda^{q,k}}^\top D s^q + z^{k} \leq 0 & \forall k = 1, 2, \dots \label{eqn:cutR}
		\end{align}
	\end{subequations}
	
	However, additional steps need to be taken to obtain a valid lower bound. When an \(\epsilon\)-feasible solution is reached, since the regularization term is appended to the master problem as shown in~\eqref{prob:masterR}, \(c(\hat{s}^p,\hat{s}^q)\) may not be a lower bound on the optimal value of model~\eqref{prob:rcACOPF}. However, we can solve the original master problem~\eqref{prob:master} with all the feasibility cuts (but without the regularization term) to obtain a valid lower bound, \(V^*\). Although a valid lower bound is obtained, the solution \((\tilde{s}^p,\tilde{s}^q)\) of this non-regularized master problem may not be equal to \((\hat{s}^p,\hat{s}^q)\), and it may not be an \(\epsilon\)-feasible solution. The algorithm needs to proceed until we obtain a \(\epsilon\)-feasible solution from solving the regularized master problem and the difference between \(V^*\) and \(c(\hat{s}^p,\hat{s}^q)\) is negligible (less than some tolerance \(\eta\)) so that we can approximate the lower bound value with the cost of this \(\epsilon\)-feasible solution. The modified algorithm is presented as Algorithm~\ref{alg:CutReg}.
	\begin{algorithm}[H]
		\small
		\caption{Regularized cutting-plane algorithm for model~\eqref{prob:rcACOPF}}
		\label{alg:CutReg}
		\begin{algorithmic}[1]
			\State  Let \((M^R)\) denote regularized master~\eqref{prob:masterR} and \((M)\) denote non-regularized master~\eqref{prob:master}; initialize iteration number \(k := 1\), tolerances \(\epsilon, \eta > 0\), and regularization weight, $\rho > 0$;
			\State Solve \((M^R)\) and obtain solution \((\hat{s}^{p,k}, \hat{s}^{q,k})\) and optimal value \(V^*\); 
			\State Solve \((SDI)\) with \((\hat{s}^p,\hat{s}^q) = (\hat{s}^{p,k}, \hat{s}^{q,k})\) and obtain solution \((\lambda^{p,k}, \lambda^{q,k})\) and optimal value \(z_{feas}^k\);
			\While{\(z_{feas}^k > \epsilon\) or \(\frac{UB - V^*}{V^*} > \eta \)} 
			\State Append \(z_{feas}^k -{\lambda^{p,k}}^\top D (s^p - \hat{s}^{p,k}) - {\lambda^{q,k}}^\top D (s^q - \hat{s}^{q,k}) \leq 0\) to constraints~\eqref{eqn:cutR} of \((M^R)\),~\eqref{eqn:cut} of \((M)\);
			\State Let \(k := k+1\);
			\State Solve \((M^{R})\) and obtain solution \((\hat{s}^{p,k}, \hat{s}^{q,k})\);
			\If{\((M^R)\) is feasible}
			\State Solve \((SDI)\) with \((\hat{s}^p,\hat{s}^q) = (\hat{s}^{p,k}, \hat{s}^{q,k})\) and obtain solution \((\lambda^{p,k}, \lambda^{q,k})\) and optimal value \(z_{feas}^k\);
			\If{\(z_{feas}^k \leq \epsilon\)}
			\State Obtain optimal value \(UB = c(\hat{s}^{p,k}, \hat{s}^{q,k})\);
			\State Solve \((M)\) and obtain solution \((\tilde{s}^{p}, \tilde{s}^{q})\) and optimal value \(V^*\);
			\State Solve \((SDI)\) with \((\hat{s}^p,\hat{s}^q) = (\tilde{s}^{p}, \tilde{s}^{q})\) and obtain solution \((\lambda^{p,k}, \lambda^{q,k})\) and optimal value \(z_{feas}^k\);
			\EndIf
			\Else
			\State Stop and return the status of infeasibility;
			\EndIf
			\vspace{0.2cm}
			\EndWhile{\textbf{end while}}
			\State Output \(V^*\) as lower bound on optimal value of model~\eqref{prob:rcACOPF}, and output \((\hat{s}^{p,k}, \hat{s}^{q,k})\) as an \(\epsilon\)-feasible solution.
		\end{algorithmic}
	\end{algorithm}
	
	Table~\ref{table:test5rTable} compares the computational performance of Algorithms~\ref{alg:Cut} and~\ref{alg:CutReg} on Cases~118 and~300 with \(\rho = 0.1, 1, 10\) and \(\eta = 10^{-4}\). It takes more than 300 iterations for Algorithm~\ref{alg:Cut} to reach an \(\epsilon\)-feasible solution, with both the violation and the lower bound improving slowly. We can see that adding a regularization term may decrease the number of iterations to convergence, but the average time for each iteration increases as \(\rho\) increases.
	\begin{table}[H]
		\centering
		\begin{tabular}{ l | l l | l l | l l}
				\hline
				\multirow{2}{*}{Parameters} & \multicolumn{2}{ c |}{No. of iterations} &  \multicolumn{2}{ c |}{\(\epsilon\)-feasibility achieved} & \multicolumn{2}{ c}{Time (sec.)}\\
				& Case 118 & Case 300 & Case 118 & Case 300 & Case 118 & Case 300 \\
				\hline
				\(\rho = 0\)& 300 & 300 & No & No & 2414 & 23042\\
				\(\rho = 0.1\) & 300 & 300 & No & No & 2640 & 30480\\
				\(\rho = 1\) & 178 & 300 & Yes & No & 2114 & 33950\\
				\(\rho = 10\) & 300 & 226 & No & Yes & 4292 & 27757\\
				\hline
		\end{tabular}
		\caption{Computational results for solving instances of model~\eqref{prob:rcACOPF} for Cases~118 and~300 with Algorithms~\ref{alg:Cut} and~\ref{alg:CutReg} with \(\Gamma = 3\) and with a limit of 300 iterations.} 
		\label{table:test5rTable}
	\end{table}
	To understand this effect, we plot the violation (in base-10 \(\log\) scale) and lower bound as a function of the iteration for Case~118 in Figure~\ref{fig:rhoPlot}. The red dots represent the value of \(UB\) corresponding to the \(\epsilon\)-feasible solutions from running Algorithm~\ref{alg:CutReg}. Without regularization, the master solution in the next iteration tends to move far from the incumbent solution. The corresponding cuts provide a global characterization of the feasible region, but it takes a long time to generate enough cuts to obtain an \(\epsilon\)-feasible solution. On the other hand, the regularized algorithm tends to generate cuts within a local area, as the new probing solution is close to the incumbent and moves quickly towards the feasible region. It takes longer to solve \((SDI)\) at a solution closer to the feasible region, which leads to a longer average time per iteration for Algorithm~\ref{alg:CutReg}.
	\begin{figure}
		\centering
		\includegraphics[width=\textwidth]{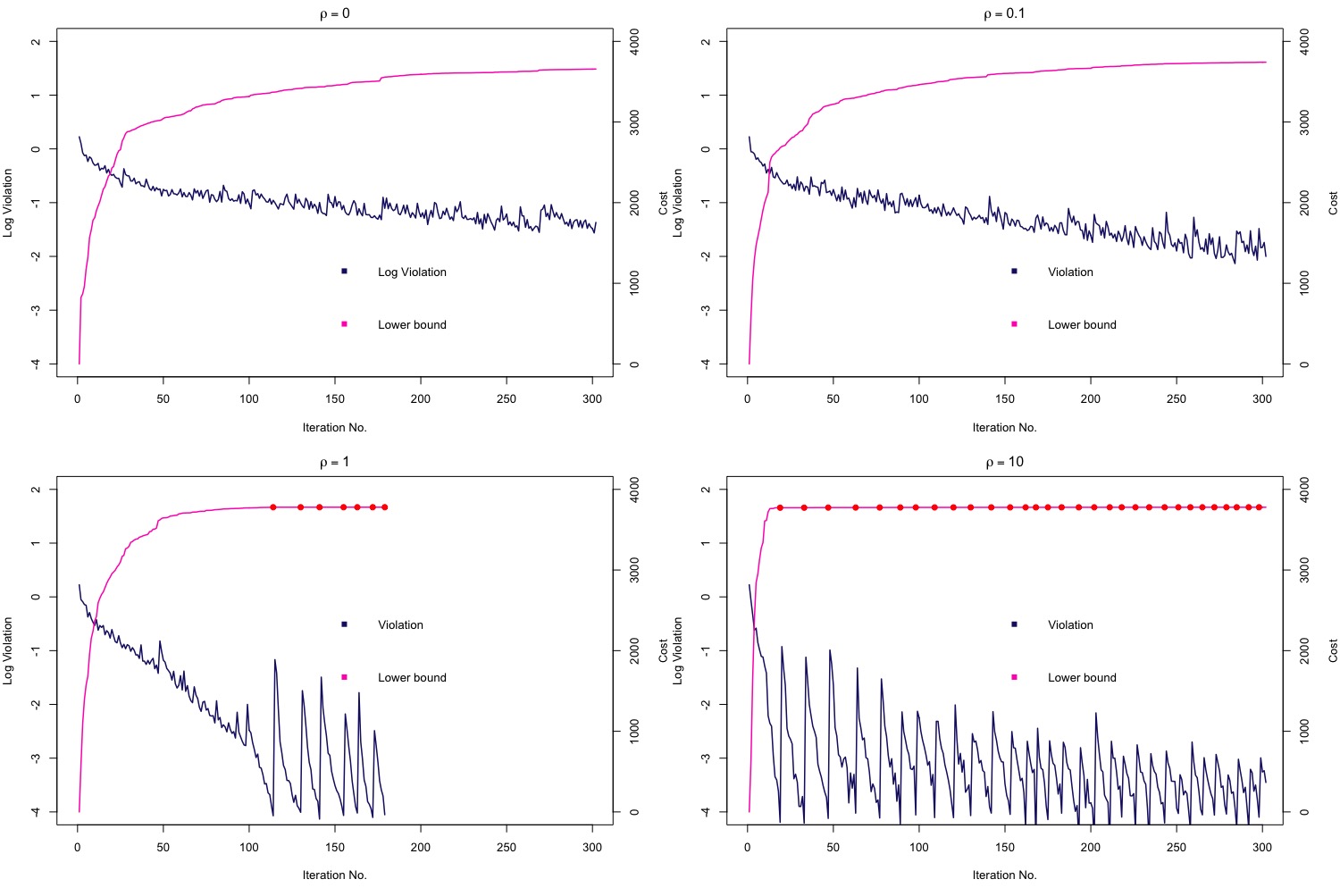}
		\caption{Computational performance of Algorithms~\ref{alg:Cut} and~\ref{alg:CutReg} for Case 118 with \(\rho = 0, 0.1,\,1\) and \(10\).} \label{fig:rhoPlot}
	\end{figure}
	
	Every time an \(\epsilon\)-feasible solution is obtained, the non-regularized master problem \((M)\) is solved to generate a lower bound. If there is still a large enough gap between the cost of that solution and the lower bound, the algorithm moves to the incumbent solution of \((M)\), which may lead to a large feasibility violation. This explains the large spikes in the plots of \(\rho = 1\) and \(\rho = 10\) in Figure~\ref{fig:rhoPlot}. The process between two spikes can be considered as exploitation of a local area. When \(\rho = 10\), there are many spikes which indicates that the algorithm reaches an \(\epsilon\)-feasible solutions frequently, but in this case the cuts generated only characterize the feasible region locally, which eventually requires many rounds of exploitation before convergence. Even with an appropriately chosen \(\rho\), the computational performance of Algorithm~\ref{alg:CutReg} is inferior to the scenario-appending technique presented in Section~\ref{subsec:scenAppend}.
	
\end{appendix}

\end{document}